\documentclass[a4paper,11pt]{article}

\usepackage[utf8]{inputenc}
\usepackage[british]{babel}

\usepackage{xcolor}

\usepackage[mono=false]{libertine}
\usepackage[T1]{fontenc}

\usepackage{amsthm}
\usepackage{amssymb}
\usepackage[cmintegrals,libertine]{newtxmath}
\usepackage[cal=euler, scr=boondoxo]{mathalfa}
\usepackage{booktabs}
\useosf

\usepackage{microtype}

\usepackage{numprint}
\usepackage{subcaption}

\usepackage[margin=2.5cm]{geometry}
\usepackage{graphicx}
\graphicspath{{./images/}}

\usepackage[font=small]{caption}
\usepackage{hyperref}
\hypersetup{
	colorlinks=true,
		citecolor=blue!60!black,
		linkcolor=red!60!black,
		urlcolor=green!40!black,
		filecolor=yellow!50!black,
	breaklinks=true,
	pdfpagemode=UseNone,
	bookmarksopen=false,
}

\usepackage{enumitem}

\newtheorem{theorem}{Theorem}
\newtheorem{lemma}[theorem]{Lemma}

\newtheorem{proposition}[theorem]{Proposition}
\newtheorem{corollary}[theorem]{Corollary}

\newtheorem{openquestion}{Open question}

\newcommand{\C}{\mathbb{C}}
\newcommand{\R}{\mathbb{R}}
\newcommand{\Z}{\mathbb{Z}}
\newcommand{\Q}{\mathbb{Q}}
\newcommand{\expec}{\mathbb{E}}
\newcommand{\prob}{\mathbb{P}}

\newcommand{\rmd}{\mathrm{d}}

\newcommand{\one}{\mathbf{1}}

\newcommand{\Dir}{\mathcal{D}}
\newcommand{\disc}{\mathbb{D}}

\DeclareMathOperator{\sech}{sech}
\DeclareMathOperator{\sn}{sn}
\DeclareMathOperator{\cn}{cn}
\DeclareMathOperator{\dn}{dn}
\DeclareMathOperator{\arcsn}{arcsn}
\DeclareMathOperator{\tr}{tr}
\DeclareMathOperator{\im}{Im}
\DeclareMathOperator{\re}{Re}

\newcommand{\DirProd}[2]{\left\langle{#1},{#2}\right\rangle_\Dir}

\title{Winding of simple walks on the square lattice}
\author{Timothy Budd\thanks{Institute for Mathematics, Astrophysics and Particle Physics, Radboud University Nijmegen. \hfill  \href{mailto:T.Budd@science.ru.nl}{\texttt{T.Budd@science.ru.nl}}}}

\begin{document}

\maketitle

\begin{abstract}
A method is described to count simple diagonal walks on $\Z^2$ with a fixed starting point and endpoint on one of the axes and a fixed winding angle around the origin.
The method involves the decomposition of such walks into smaller pieces, the generating functions of which are encoded in a commuting set of Hilbert space operators.
The general enumeration problem is then solved by obtaining an explicit eigenvalue decomposition of these operators involving elliptic functions.
By further restricting the intermediate winding angles of the walks to some open interval, the method can be used to count various classes of walks restricted to cones in $\Z^2$ of opening angles that are integer multiples of $\pi/4$.

We present three applications of this main result. 
First we find an explicit generating function for the walks in such cones that start and end at the origin.
In the particular case of a cone of angle $3\pi/4$ these walks are directly related to Gessel's walks in the quadrant, and we provide a new proof of their enumeration.
Next we study the distribution of the winding angle of a simple random walk on $\Z^2$ around a point in the close vicinity of its starting point, for which we identify discrete analogues of the known hyperbolic secant laws and a probabilistic interpretation of the Jacobi elliptic functions. 
Finally we relate the spectrum of one of the Hilbert space operators to the enumeration of closed loops in $\Z^2$ with fixed winding number around the origin.
\end{abstract}


\section{Introduction}

Counting of lattice paths has been a major topic in combinatorics (and probability and physics) for many decades. 
Especially the enumeration of various types of lattice walks confined to convex cones in $\Z^2$, like the positive quadrant, has attracted much attention in recent years, due mainly to the rich algebraic structure of the generating functions involved (see e.g \cite{bousquet-melou_walks_2010,bernardi_counting_2017} and references therein) and the relations with other combinatorial structures (e.g. \cite{bernardi_bijective_2007,kenyon_bipolar_2015}).
The study of lattice walks in non-convex cones has received much less attention. 
Notable exception are walks on the slit plane \cite{bousquet-melou_walks_2001,bousquet-melou_walks_2002} and the three-quarter plane \cite{bousquet-melou_square_2016}.
When describing the plane in polar coordinates, the confinement of walks to cones of different opening angles (with the tip positioned at the origin) may equally be phrased as a restriction on the angular coordinates of the sites visited by the walk.
One may generalize this concept by replacing the angular coordinate by a notion of winding angle of the walk around the origin, in which case one can even make sense of cones of angles larger than $2\pi$.
It stands to reason that a fine control over the winding angle in the enumeration of lattice walks brings us a long way in the study of walks in (especially non-convex) cones.

Although the winding angle of lattice walks seems to have received little attention in the combinatorics literature,  probabilistic aspects of the winding of long random walks have been studied in considerable detail \cite{belisle_windings_1989,belisle_winding_1991,rudnick_winding_1987,shi_windings_1998}.
In particular, it is known that under suitable conditions on the steps of the random walk the winding angle after $n$ steps is typically of order $\log n$, and that the angle normalized by $2/\log n$ converges in distribution to a standard hyperbolic secant distribution.
The methods used all rely on coupling to Brownian motion, for which the winding angle problem is easily studied with the help of its conformal properties.
Although quite generally applicable in the asymptotic regime, these techniques tell us little about the underlying combinatorics.
  
In this paper we initiate the combinatorial study of lattice walks with control on the winding angle, by looking at various classes of simple (rectilinear or diagonal) walks on $\Z^2$.
As we will see, the combinatorial tools described in this paper are strong enough to bridge the gap between the combinatorial study of walks in cones and the asymptotic winding of random walks.
Before describing the main results of the paper, we should start with some definitions. 

We let $\mathcal{W}$ be the set of \emph{simple diagonal walks} $w$ in $\Z^2$ of length $|w|\geq 0$ avoiding the origin, i.e. $w$ is a sequence $(w_i)_{i=0}^{|w|}$ in $\Z^2\setminus\{(0,0)\}$ with $w_{i}-w_{i-1} \in \{(1,1),(1,-1),(-1,-1),(-1,1)\}$ for $1\leq i\leq |w|$.
We define the \emph{winding angle $\theta_i^w \in \R$ of $w$ up to time $i$} to be the difference in angular coordinates of $w_i$ and $w_0$ including a contribution $2\pi$ (resp.\ $-2\pi$) for each full counterclockwise (resp. clockwise) turn of $w$ around the origin up to time $i$.
Equivalently, $(\theta_i^w)_{i=0}^{|w|}$ is the unique sequence in $\R$ such that $\theta_0^w = 0$ and $\theta_{i}^w-\theta_{i-1}^w$ is the (counterclockwise) angle between the segments $((0,0),w_{i-1})$ and $((0,0),w_{i})$ for $1\leq i\leq |w|$.
The \emph{(full) winding angle} of $w$ is then $\theta^w \coloneqq \theta_{|w|}^w$. 
See Figure \ref{fig:windingangle} for an example.

\begin{figure}
	\centering
	\includegraphics[width=6cm]{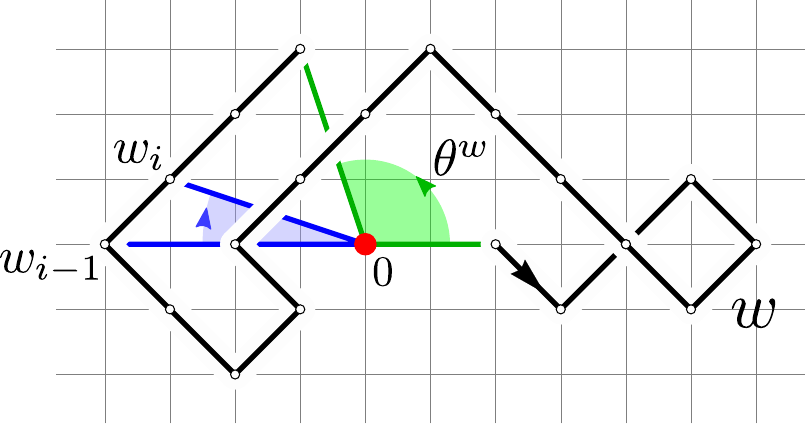}
	\caption{The winding angle of a simple diagonal walk $w\in\mathcal{W}$ of length $|w| = 19$ from $(2,0)$ to $(-1,3)$. The (full) winding angle $\theta^w$ is indicated in green, and the winding angle increment $\theta^w_{i}-\theta^w_{i-1}$ in blue (which is negative in this example).}\label{fig:windingangle}
\end{figure}

\paragraph{Main result} The \emph{Dirichlet space} $\mathcal{D}$ is the Hilbert space of complex analytic functions $f$ on the unit disc $\disc = \{ z\in \C : |z| < 1 \}$ that vanish at $0$ and have finite norm $\|f\|^2_{\Dir}$ with respect to the Dirichlet inner product
\[
\langle f,g\rangle_\Dir = \int_{\disc} \overline{f'(z)}\,g'(z) \rmd A(z)= \sum_{n=1}^\infty n\, \overline{[z^n]f(z)}\,\, [z^n]g(z),
\]
where the measure $\rmd A(z_1+iz_2)  \coloneqq \frac{1}{\pi}\rmd z_1\rmd z_2$ is chosen such that $\int_{\disc} \rmd A(z) = 1$.
See \cite{arcozzi_dirichlet_2011} for a review of its properties.
We denote by $(e_n)_{n=1}^\infty$ the standard orthogonal basis defined by $e_n(z) \coloneqq z^n$, which is unnormalized since $\|e_n\|^2_\Dir = n$.
With this notation $\DirProd{e_n}{f} = n [z^n] f(z)$ for any analytic function $f\in\Dir$.

For $k\in(0,1)$ we let $v_k :  \C \setminus \{z\in\R : z^2 \geq k\} \to \C$ be the analytic function defined by the elliptic integral
\begin{equation}\label{eq:vkmap}
v_k(z) \coloneqq \frac{1}{4K(k)} \int_0^z \frac{\rmd x}{\sqrt{(k-x^2)(1-k x^2)}}
\end{equation}
along the simplest path from $0$ to $z$,
where $K(k)$ is the complete elliptic integral of the first kind with elliptic modulus $k$ (see Appendix \ref{sec:ellipticappendix} for definitions and notation).
The appearance of this elliptic integral in lattice walks enumeration is a natural one since $\frac{2}{\pi} K(4t)$ is precisely the generating function for excursions of the simple diagonal walk from the origin (see \eqref{eq:Kexpansion} in Appendix \ref{sec:ellipticappendix}),
\begin{equation}\label{eq:standardexcursiongenfun}
\frac{2}{\pi}K(4t) = \sum_{n=0}^\infty \binom{2n}{n}^2 \,t^{2n}.
\end{equation}
The incomplete elliptic integral $v_k$ does not have a comparably simple combinatorial interpretation, but provides an important conformal mapping from a slit disk onto a rectangle in the complex plane, as detailed in Section \ref{sec:operatorJ}.

For fixed $k$ we use the conventional notation
\[
k' = \sqrt{1-k^2} \qquad \text{and}\qquad k_1 = \frac{1-k'}{1+k'}
\]
for the \emph{complementary modulus} $k'$ and the \emph{descending Landen transformation} $k_1$ of $k$, which both take values in $(0,1)$ again (see Appendix \ref{sec:ellipticappendix}).
Using these we introduce a family $(f_m)_{m=1}^\infty$ of analytic functions by setting (notice the $k_1$ in $v_{k_1}(z)$!)
\begin{equation}\label{eq:fmbasis}
f_m(z) \coloneqq \cos( 2\pi m (v_{k_1}(z) + 1/4)) - \cos(\pi m/2),
\end{equation}
which satisfies $f_m(0) = 0$.
Even though $v_{k_1}(z)$ has branch cuts at $z=\pm\sqrt{k_1}$, we will see (Lemma \ref{thm:isomorphism}, \ref{thm:fourierbasis}, \ref{thm:fmradius}) that $f_m(z)$ has radius of convergence around $0$ equal to $1/\sqrt{k_1} > 1$ and has finite norm with respect to the Dirichlet inner product, hence $f_m \in \Dir$. 
According to Proposition \ref{thm:Jdiagonalization} the norm of $f_m$ is given explicitly by  
\begin{equation}\label{eq:fnorm}
\| f_m\|^2_{\Dir} = \frac{m(q_k^{-m}-q_k^m)}{4}, 
\end{equation}
where $q_k \in(0,1)$ is the \emph{(elliptic) nome} of modulus $k$ (see \eqref{eq:nomedef} in Appendix \ref{sec:ellipticappendix}), which is analytic for $k$ in the unit disk.
Once properly normalized the family of functions $(\hat{f}_m)_{m=1}^\infty$ provides an orthonormal basis of $\Dir$, i.e.
\begin{equation*}
\hat{f}_m(z) \coloneqq \frac{f_m(z)}{\|f_m\|_{\Dir}} = \frac{\cos( 2\pi m (v_{k_1}(z) + 1/4)) - \cos(\pi m/2)}{\sqrt{\frac{m}{4}(q_k^{-m}-q_k^m)}}, \qquad\DirProd{\hat{f}_n}{\hat{f}_m} = \one_{n=m}.
\end{equation*}

The main technical result of this paper is the following.

\begin{theorem}\label{thm:mainresult}
	For $\ell,p\geq 1$ and $\alpha\in\frac{\pi}{2}\Z$, let $\mathcal{W}^{(\alpha)}_{\ell,p}$ be the set of (possibly empty) simple diagonal walks $w$ on $\Z^2\setminus\{(0,0)\}$ that start at $(p,0)$, end on one of the axes at distance $\ell$ from the origin, and have full winding angle $\theta^w=\alpha$.
	\begin{itemize}
		\item[(i)]
		Let $W^{(\alpha)}_{\ell,p}(t) \coloneqq \sum_{w\in\mathcal{W}^{(\alpha)}_{\ell,p}} t^{|w|}$ be the generating function of $\mathcal{W}^{(\alpha)}_{\ell,p}$.
		For $k=4t\in(0,1)$ fixed, there exists a compact self-adjoint operator $\mathbf{Y}^{(\alpha)}_k$ on $\Dir$ with eigenvectors $(f_m)_{m=1}^\infty$ such that
		\begin{equation}
		W^{(\alpha)}_{\ell,p}(t) = \langle e_\ell, \mathbf{Y}^{(\alpha)}_k e_p\rangle_\Dir, \qquad \mathbf{Y}^{(\alpha)}_k f_m = \frac{2K(k)}{\pi} \frac{1}{m} q_k^{m|\alpha|/\pi}\, f_m.
		\end{equation}
		\item[(ii)] Let $\mathcal{W}^{(\alpha,I)}_{\ell,p}\subset\mathcal{W}^{(\alpha)}_{\ell,p}$ be the subset of the aforementioned walks that have intermediate winding angles in an interval $I\subset\R$, i.e. $\theta_i^w \in I$ for $i=1,2,\ldots,|w|-1$, and let $W_{\ell,p}^{(\alpha,I)}(t)$ be the corresponding generating function.
		If $I = (\beta_-,\beta_+)$ with $\beta_\pm\in\frac{\pi}{2}\Z\cup\{\pm\infty\}$, $\alpha\in[\beta_-,\beta_+]\cap \frac{\pi}{2}\Z$ and $\alpha\neq0$ or $\alpha\neq \beta_\pm$, then the generating function $W_{\ell,p}^{(\alpha,I)}(t)$ is related to a matrix element of a compact self-adjoint operator on $\mathcal{D}$ with the same eigenvectors $(f_m)_{m=1}^\infty$, as described in the table below.
	\[	
	\begin{array}{*4{>{\displaystyle}c}rl}
	\toprule 
	\addlinespace
	\alpha & \beta_- & \beta_+ & W_{\ell,p}^{(\alpha,I)}(t) & &\text{Eigenvalues} \\ \hline
	\addlinespace
	>0 & 0 & \alpha & \frac{1}{\ell p}\langle e_\ell, \mathbf{A}^{(\alpha)}_k e_p\rangle &		\mathbf{A}^{(\alpha)}_k f_m =&  \frac{\pi}{2K(k)} \frac{m}{q_k^{-m\alpha/\pi}-q_k^{m\alpha/\pi}}\, f_m \\
	\addlinespace
	>0 & <0 & \alpha & \frac{1}{\ell} \langle e_\ell, \mathbf{J}^{(\alpha,\beta_-)}_k e_p\rangle &		\mathbf{J}^{(\alpha,\beta_-)}_k f_m =& \frac{q_k^{2m\beta_-/\pi}-1}{q_k^{2m\beta_-/\pi}-q_k^{2m\alpha/\pi}}q_k^{m\alpha/\pi}\, f_m\\
	\addlinespace
	\geq0 & <0 & >\alpha & \langle e_\ell, \mathbf{B}^{(\alpha,\beta_-,\beta_+)}_k e_p\rangle &		\mathbf{B}^{(\alpha,\beta_-,\beta_+)}_k f_m =& \frac{2K(k)}{\pi}\, \frac{q_k^{2m\beta_-/\pi}-1}{m\,q_k^{m\alpha/\pi}}\, \frac{q_k^{2m\alpha/\pi}-q_k^{2m\beta_+/\pi}}{q_k^{2m\beta_-/\pi}-q_k^{2m\beta_+/\pi}}\, f_m\\
	\bottomrule	
	\end{array}
	\]
	The remaining cases follow from the symmetries $(\alpha,\beta_-,\beta_+)\to(-\alpha,-\beta_+,-\beta_-)$ and $(\alpha,\beta_-,\beta_+)\to(\alpha,\alpha-\beta_+,\alpha-\beta_-)$, and the cases $\beta_\pm = \pm\infty$ agree with the corresponding limits $\beta\pm\to\pm\infty$ (using that $q_k \in (0,1)$).
		\item[(iii)] The statement of (ii) remains valid for $\beta_\pm\in\frac{\pi}{4}\Z\cup\{\pm\infty\}$ and $\alpha\in[\beta_-,\beta_+]\cap \frac{\pi}{2}\Z$ and $\alpha\neq 0$ or $\alpha \neq \beta_\pm$ as long as $\ell$ and $p$ are even.

	\end{itemize}
\end{theorem}

We emphasize that the convention for the index placement in $\mathcal{W}_{\ell,p}^{(\alpha,I)}$, that is used also for other sets of walks throughout this work, is that the first index corresponds to the endpoint of the walk and the second index to the starting point.
The reason for this choice is precisely the relation of the generating function $W^{(\alpha)}_{\ell,p}$ to the matrix element at position $(\ell,p)$ of a Hilbert space operator.
As we will see, composition of particular families of walks often corresponds to the composition of the respective operators (or multiplication of the respective infinite matrices).
The index placement reflects the natural right-to-left ordering in the composition of these operators.

\begin{figure}[ht]
	\begin{subfigure}[c]{0.333\textwidth}
		\centering
		\includegraphics[width=.95\linewidth]{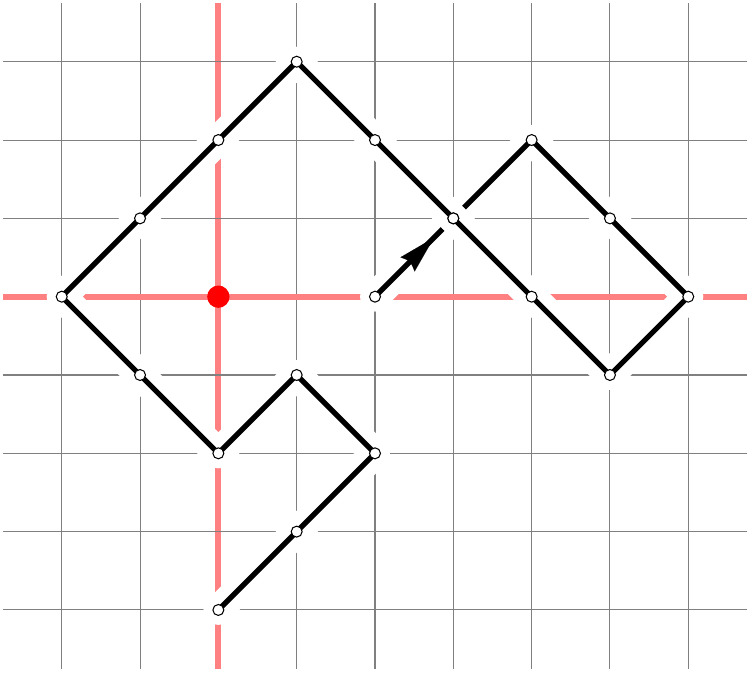}
		\caption{$\mathcal{W}_{4,2}^{(3\pi/2)}$}\label{fig:axeswalks1}
	\end{subfigure}%
	\begin{subfigure}[c]{0.333\textwidth}
		\centering
		\includegraphics[width=.95\linewidth]{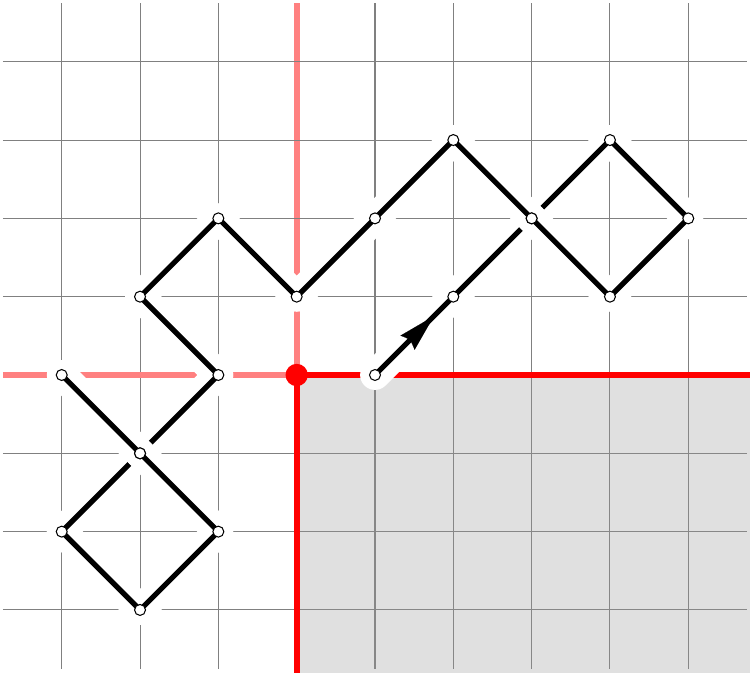}
		\caption{$\mathcal{W}_{3,1}^{(\pi,(0,3\pi/2))}$}\label{fig:axeswalks2}
	\end{subfigure}%
	\begin{subfigure}[c]{0.333\textwidth}
		\centering
		\includegraphics[width=.95\linewidth]{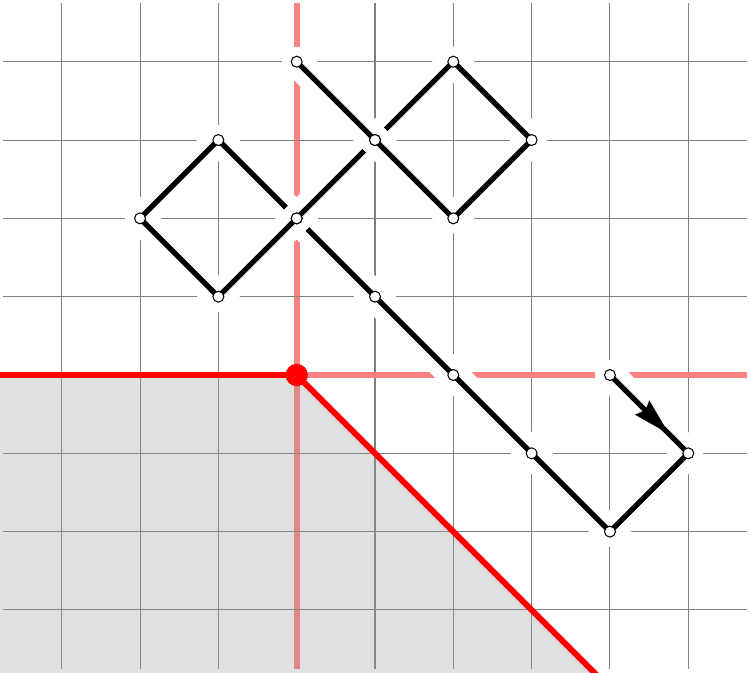}
		\caption{$\mathcal{W}_{4,4}^{(\pi/2,(-\pi/4,\pi))}$}\label{fig:axeswalks3}
	\end{subfigure}%
	\caption{Examples of walks enumerated by Theorem \ref{thm:mainresult}.}\label{fig:axeswalks}
\end{figure}

Theorem \ref{thm:mainresult} is stated for fixed real values of $k=4t\in(0,1)$, while from a combinatorial point of view it may be preferable to think of the generating functions $W_{\ell,p}^{(\alpha,I)}(t)$ as formal power series in the variable $t$.
This raises the question to what extent the eigenvalue decomposition can be understood on the level of formal power series.
To this end we prove in Proposition \ref{thm:fhatanalytic} that for any $m,p\geq 1$ the coefficient $[z^p]\hat{f}_m(z) = \frac{1}{p}\DirProd{e_p}{\hat{f}_m}$ is an analytic function in $k$ around $0$, such that we may interpret $\hat{f}_m(z) = \hat{f}_m(z;t)$ as taking values in the ring $\R[\![z,t]\!]$ of formal power series in the variables $z$ and $t$, see Section \ref{sec:analytic} for an explicit expansion for $m=1,\ldots,6$.
With this observation Theorem \ref{thm:mainresult} can be largely recast as a set of identities on formal power series.
To this end we denote the multivariate generating function of $\mathcal{W}_{\ell,p}^{(\alpha,I)}$ by
\begin{equation*}
W^{(\alpha,I)}(x,y;t) = \sum_{\ell,p\geq 1} x^\ell y^p W_{\ell,p}^{(\alpha,I)}(t) \in \R[\![x,y,t]\!],
\end{equation*}
and the eigenvalues of $\mathbf{A}_k^{(\alpha)}$, $\mathbf{J}_k^{(\alpha,\beta_-)}$, $\mathbf{B}_k^{(\alpha,\beta_-,\beta_+)}$ by $A_m^{(\alpha)}(t)$, $J_m^{(\alpha,\beta_-)}(t)$, $B_m^{(\alpha,\beta_-,\beta_+)}(t)$ respectively (this includes $\mathbf{Y}_k^{(\alpha)} = \mathbf{B}_k^{(\alpha,-\infty,\infty)}$ as a special case).
These eigenvalues, as given in Theorem \ref{thm:mainresult}(ii), are all analytic functions of $t$ in the unit disk (see \eqref{eq:nomeexpansion} and \eqref{eq:standardexcursiongenfun} for the power series representations of $q_k$ and $K(k)$). 
Provided $\alpha \neq 0$, the $m$th eigenvalue converges to $0$ as $m\to\infty$ in the formal topology of $\R[\![t]\!]$, i.e. for any $n\geq 0$ the coefficient $[t^n] A_m^{(\alpha)}(t)$ vanishes for all $m$ large enough. 
Theorem \ref{thm:mainresult} implies the convergent series identities in $\R[\![x,y,t]\!]$ for $\alpha\neq 0$,
\begin{align*}
W^{(\alpha,(0,\alpha))}(x,y;t) &= \sum_{m=1}^\infty A_m^{(\alpha)}(t) \hat{f}_m(x;t)\hat{f}_m(y;t), \\
W^{(\alpha,(\beta_-,\alpha))}(x,y;t) &= y\frac{\partial}{\partial y}\sum_{m=1}^\infty J_m^{(\alpha,\beta_-)}(t) \hat{f}_m(x;t)\hat{f}_m(y;t), \\
W^{(\alpha,(\beta_-,\beta_+))}(x,y;t) &= xy\frac{\partial}{\partial x}\frac{\partial}{\partial y}\sum_{m=1}^\infty B_m^{(\alpha,\beta_-,\beta_+)}(t) \hat{f}_m(x;t)\hat{f}_m(y;t),
\end{align*} 
with $\alpha$, $\beta_-$ and $\beta_+$ satisfying the respective assumptions described in Theorem \ref{thm:mainresult}(ii) and (iii).
We leave it as an open problem whether these identities and an expression for $\hat{f}_m(x;t)$ can be derived using exclusively formal power series.

As an example of how to compute explicit generating functions with the help of Theorem \ref{thm:mainresult}, let us look at the set $\mathcal{W}_{3,3}^{(\pi)}$ of simple diagonal walks from $(3,0)$ to $(-3,0)$ that have winding angle $\pi$ around the origin.
According to Theorem \ref{thm:mainresult}(i) its generating function satisfies
\begin{equation*}
W_{3,3}^{(\pi)}(t) = \DirProd{e_3}{\mathbf{Y}_k^{(\pi)} e_3} = \frac{2K(k)}{\pi}\sum_{m=1}^\infty \frac{1}{m} q_k^{m} \DirProd{e_3}{\hat{f}_m}^2 = 10\, t^6 + 280\, t^8 + 5661\, t^{10} + \cdots,
\end{equation*}
where we used the series expansions of $q_k$ and $K(k)$ (respectively \eqref{eq:nomeexpansion} and \eqref{eq:Kexpansion} in Appendix \ref{sec:ellipticappendix}) and that of $\DirProd{e_3}{\hat{f}_m}$ in Section \ref{sec:analytic}. 

\paragraph{Application: Excursions}
Theorem \ref{thm:mainresult} can be used to count many specialized classes of walks involving winding angles. 
The first quite natural counting problem we address is that of the \emph{(diagonal) excursions} $\mathcal{E}$ from the origin, i.e. $\mathcal{E}$ is the set of (non-empty) simple diagonal walks starting and ending at the origin with no intermediate returns (Figure \ref{fig:excursions1}).
Actually, in this case we may equally well consider simple rectilinear walks on $\Z^2$, thanks to the obvious linear mapping between the two types of walks (Figure \ref{fig:excursions2}).
\begin{figure}[t]
	\begin{subfigure}[c]{0.45\textwidth}
		\centering
		\includegraphics[height=.6\linewidth]{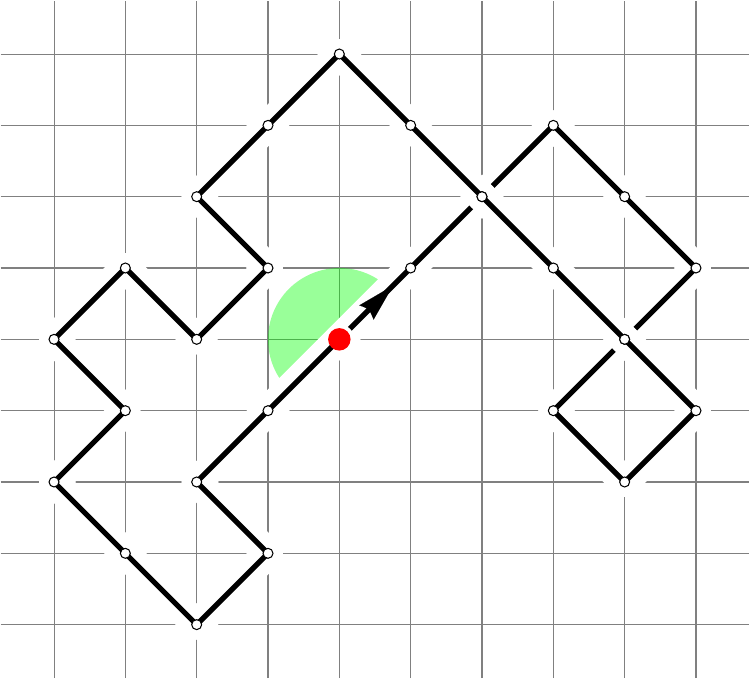}
		\caption{Diagonal excursion.}\label{fig:excursions1}
	\end{subfigure}%
	\begin{subfigure}[c]{0.45\textwidth}
		\centering
		\includegraphics[height=.6\linewidth]{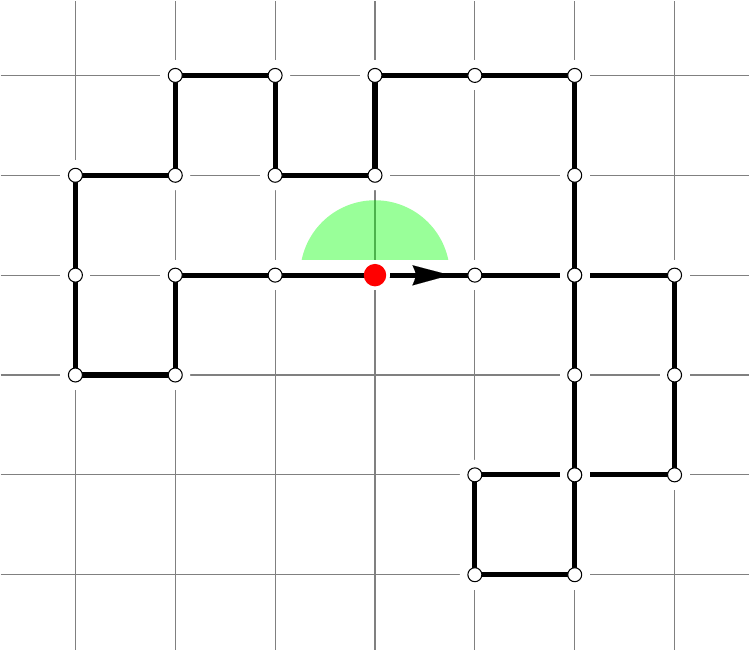}
		\caption{Rectilinear excursion.}\label{fig:excursions2}
	\end{subfigure}%
	\caption{Example of an excursion in diagonal and rectilinear form together with its winding angle.}\label{fig:excursions}
\end{figure}
Even though walks $w\in\mathcal{E}$ do not completely avoid the origin, we may still naturally assign a winding angle sequence to them by imposing that the first and last step do not contribute to the winding angle, i.e. $\theta^w_1 = \theta^w_0=0$ and $\theta^w = \theta^w_{|w|} = \theta^w_{|w|-1}$.
In Proposition \ref{thm:excursiongenfun} we prove that the generating function for excursions with winding angle $\alpha\in\frac{\pi}{2}\Z$ is given (for $k=4t\in(0,1)$ fixed) by
\[
F^{(\alpha)}(t) \coloneqq \sum_{w\in\mathcal{E}} t^{|w|} \one_{\{\theta^w = \alpha\}}= \frac{2\pi}{K(k)} \sum_{n=1}^\infty \frac{(1-q_k^{n})^2}{1-q_k^{4n}}q_k^{n(\frac{2}{\pi}|\alpha|+1)}.
\]
Since the summand is analytic in $t$ around $0$ and $O(t^{2n})$ for any $n$, the relation implies an identity of formal power series in $\R[\![t]\!]$ for any $\alpha\in\frac{\pi}{2}\Z$. 

Similarly to Theorem \ref{thm:mainresult}(ii) one may further restrict the full winding angle sequence of $w$ to lie in an open interval $I = (\beta_-,\beta_+)$ with $\beta_-\in\frac{\pi}{4}\Z_{<0}$, $\beta_+\in\frac{\pi}{4}\Z_{>0}$ and $\alpha \in I\cap \frac{\pi}{2}\Z$.
In this case it is more natural to also fix the starting direction, say $w_1 = (1,1)$, and we denote the corresponding generating function by $F^{(\alpha,I)}(t)$. Observe in particular that $F^{(\alpha,\R)}(t) = F^{(\alpha)}(t)/4$.
We prove in Theorem \ref{thm:conegenfun} that the generating function $F^{(\alpha,I)}(t)$ is given by the finite sum
\begin{equation}\label{eq:conegenfunintro}
F^{(\alpha,I)}(t)=\frac{\pi}{8\delta} \sum_{\sigma\in(0,\delta)\cap \frac{\pi}{2}\Z}\left( \cos\left( \frac{4\sigma\alpha}{\delta} \right) - \cos\left( \frac{4\sigma(2\beta_+ - \alpha)}{\delta} \right)\right) F\left(t,\frac{4\sigma}{\delta}\right), \qquad \delta \coloneqq 2(\beta_+-\beta_-),
\end{equation}
where $F(t,b)\coloneqq \sum_{\alpha\in\frac{\pi}{2}\Z} F^{(\alpha)}(t)\, e^{ib\alpha}$ is the ``characteristic function'' associated to $F^{(\alpha)}(t)$.
Since $F(t,b)=F(t,4-b)$, the terms in the right-hand side of \eqref{eq:conegenfunintro} corresponding to $\sigma$ and $\delta - \sigma$ are actually identical, therefore leaving only $\lfloor \frac{\delta}{\pi}\rfloor$ distinct terms.
For non-integer values of $b$ (see Proposition \ref{thm:excursiongenfun} for the full expression) $F(t,b)$ can be expressed in closed form as
\begin{equation}\label{eq:Fcharfun}
F(t,b)= \frac{1}{\cos\left(\frac{\pi b}{2}\right)}\left[ 1-\frac{\pi \tan\left(\frac{\pi b}{4}\right)}{2K(k)} \frac{\theta_1'\left(\frac{\pi b}{4},\sqrt{q_k}\right)}{\theta_1\left(\frac{\pi b}{4},\sqrt{q_k}\right)}\right],\qquad (b\in\R\setminus\Z)
\end{equation}
where $\theta_1(z,q)$ is the first Jacobi theta function (see \eqref{eq:theta1def} for a definition).
Again \eqref{eq:conegenfunintro} and \eqref{eq:Fcharfun} imply the equivalent identities on the level of formal power series in $\R[\![t]\!]$ or $\C[\![t]\!]$.
In Proposition \ref{thm:Falgebraic} we prove that the power series $F(t,b)$ is algebraic in $t$ for any $b\in \Q\setminus\Z$ but that is it transcendental for $b\in\Z$.
By inspecting the terms appearing in \eqref{eq:conegenfunintro} we find that $F^{(\alpha,I)}(t)$ is transcendental if and only if $\beta_+,\beta_-\in \frac{\pi}{2}\Z +\frac{\pi}{4}$ or both $\beta_+,\beta_-\in \pi\Z +\frac{\pi}{2}$ and $\alpha \in\pi\Z$ (see Theorem \ref{thm:conegenfun}).

\begin{figure}
	\centering
	\begin{subfigure}[c]{0.35\textwidth}
		\centering
		\reflectbox{\includegraphics[height=.7\linewidth,angle=90]{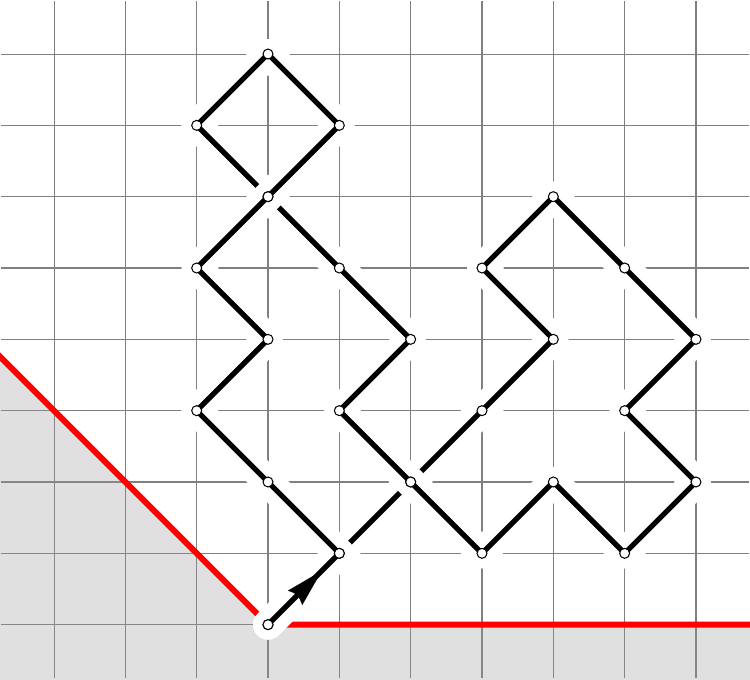}}
		\caption{}\label{fig:gesselwalk1}
	\end{subfigure}%
	\begin{subfigure}[c]{0.35\textwidth}
		\centering
		\reflectbox{\includegraphics[height=.7\linewidth,angle=90]{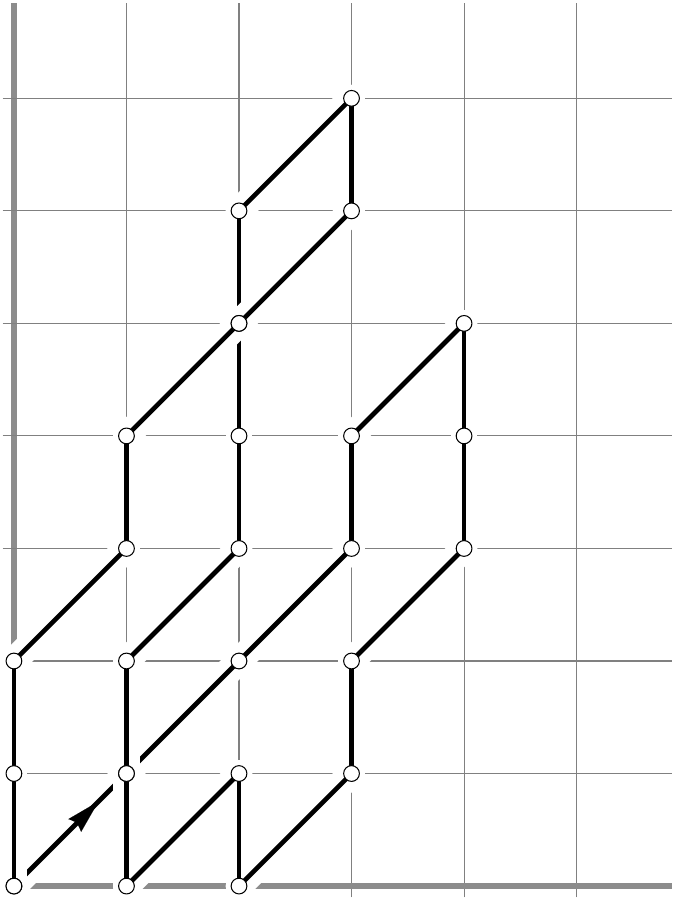}}
		\caption{}\label{fig:gesselwalk2}
	\end{subfigure}%
		\caption{Excursions (a) of length $2n+2$ staying in the angular interval $(-\pi/4,\pi/2)$ are in bijection with Gessel walks (b) in the quadrant starting and ending at the origin with $2n$ steps.}\label{fig:gesselwalk}
\end{figure}
A special case of excursions for which the generating function is algebraic is $\alpha = 0$ and $I=(-\pi/4,\pi/2)$, see Figure \ref{fig:gesselwalk1}.
After removal of the first and last step of the walk and a linear transformation, these correspond precisely to walks in the quadrant starting and ending at the origin with steps in $\{(-1,0),(1,0),(-1,-1),(1,1)\}$, also known as Gessel walks.
Around 2000 Ira Gessel conjectured that the generating function for such walks is given (in our notation) by
\[
\frac{1}{t^2}F^{(0,(-\pi/4,\pi/2))}(t) = \sum_{n=0}^\infty t^{2n}\,16^n \frac{(5/6)_n(1/2)_n}{(2)_n(5/3)_n} = 1+2 t^2+11 t^4+85 t^6+\cdots,
\]
where $(a)_n = \Gamma(a+n)/\Gamma(a)$ is the descending Pochhammer symbol (see also Sloane's \emph{Online Encyclopedia of Integer sequences} (OEIS) sequence \href{https://oeis.org/A135404}{A135404}).
The first computer-aided proof of this conjecture appeared in \cite{kauers_proof_2009}, and it was followed by several ``human'' proofs in \cite{bostan_human_2017,bousquet-melou_elementary_2016,bernardi_counting_2017}.
Here we provide an alternative proof using Theorem \ref{thm:conegenfun}.
Indeed, we have explicitly 
\[
F^{(0,(-\pi/4,\pi/2))}(t) = \frac{1}{8} F\left(t,\frac{4}{3}\right) + \frac{1}{8} F\left(t,\frac{8}{3}\right) = \frac{1}{4} F\left(t,\frac{4}{3}\right) = \frac{1}{2} \left[ \frac{\sqrt{3}\pi}{2K(4t)}  \frac{\theta_1'\left(\frac{\pi}{3},\sqrt{q_k}\right)}{\theta_1\left(\frac{\pi}{3},\sqrt{q_k}\right)}-1\right], 
\]
where used that $F(t,b)=F(t,4-b)$.
According to our discussion above this is an algebraic power series in $t$, a fact about $F^{(0,(-\pi/4,\pi/2))}(t)$ that was first observed in \cite{bostan_complete_2010}.
In Corollary \ref{thm:gessel} we deduce an explicit algebraic equation for $F^{(0,(-\pi/4,\pi/2))}(t)$, and check that it agrees with a known equation for $\sum_{n=0}^\infty t^{2n+2}\,16^n \frac{(5/6)_n(1/2)_n}{(2)_n(5/3)_n}$.

\paragraph{Application: Unconstrained random walks} Let $(W_i)_{i\geq0}$ be a simple random walk on $\Z^2$ started at the origin. 
A natural question is to ask for the (approximate) distribution of the winding angle $\theta_j^z$ of the random walk around some point $z\in\R^2$ up to time $j$.
As mentioned before, this question has been addressed successfully in the literature in the limit $j\to\infty$ using coupling to Brownian motion.
If $z\in(\Z + 1/2)^2$, then $2\theta_j^z/\log j$ is known \cite{rudnick_winding_1987,belisle_windings_1989,belisle_winding_1991,shi_windings_1998} to converge in distribution to the hyperbolic secant law with density $\frac{1}{2} \sech(\pi x/2)\rmd x$ (recall that $\sech y= 1/\cosh y = 2/(e^{y}+e^{-y})$).
If instead $z\in\Z^2$ and one conditions the random walk not to hit $z$ before time $j$, then $2\theta_j^z/\log j$ converges to a ``hyperbolic secant squared law'' with density $\frac{\pi}{4} \sech^2(\pi x/2)\rmd x$ \cite{rudnick_winding_1987}.

In Section \ref{sec:unconstrained} we complement these results by deriving exact statistics at finite $j$ with the help of Theorem \ref{thm:mainresult}.
To this end we look at the winding angles around two points in the vicinity of the starting point, namely $(-1/2,-1/2)$ and the origin itself. 
To be precise, let $\theta_{j+1/2}^\square$ be the winding angle of $(W_i)_i$ around $(-1/2,-1/2)$ up to time $j+1/2$, i.e. halfway its step from $W_j$ to $W_{j+1}$ (see Figure \ref{fig:freewinding}).
Similarly, let $\theta_{j+1/2}^\bullet$ be the winding angle around the origin, ignoring the first step and with the convention that $\theta_{j+1/2}^\bullet = \infty$ if $(W_i)_i$ has returned to the origin strictly before time $j+1$. 
\begin{figure}
	\centering
	\includegraphics[width=.6\linewidth]{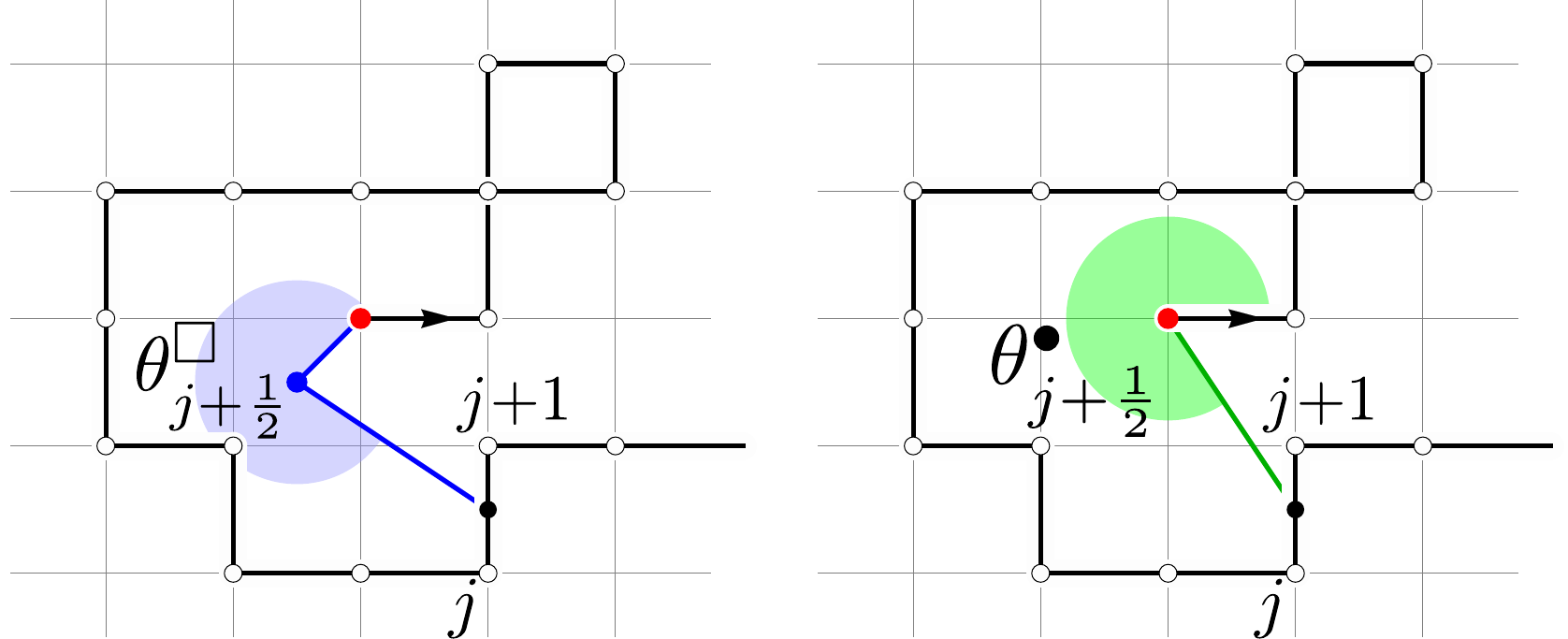}
	\caption{The winding angles $\theta^\square_{j+\frac12}$ and $\theta^\bullet_{j+\frac12}$ of a simple random walk at time $j+\frac{1}{2}$ around, respectively, a dual lattice point adjacent to the origin and the origin itself.\label{fig:freewinding}} 
\end{figure}

If $\zeta_k$ is a geometric random variable with parameter $k\in(0,1)$, i.e. $\prob[ \zeta_k = j ] = k^j(1-k)$ for $j\geq 0$, then the following ``discrete hyperbolic secant laws'' hold (Theorem \ref{thm:hypsecant}):
\begin{align*}
\prob\big[ \theta^\square_{\zeta_k+1/2} \in (\alpha-{\textstyle\frac{\pi}{2}}, \alpha+{\textstyle\frac{\pi}{2}}) \big] &= c_k \sech\left(4T_k\alpha\right) + {\textstyle\frac{k-1}{k}}\one_{\{\alpha=0\}}&&\text{for }\alpha\in{\textstyle\frac{\pi}{2}}\Z,\\
\prob\big[ \theta^{\bullet}_{\zeta_k+1/2} \in (\alpha-{\textstyle\frac{\pi}{2}}, \alpha+{\textstyle\frac{\pi}{2}})\big] &= C_k  \sech\left(4T_k(\alpha-{\textstyle\frac{\pi}{4}})\right)\sech\left(4T_k(\alpha+{\textstyle\frac{\pi}{4}})\right) &&\text{for }\alpha\in{\textstyle\frac{\pi}{2}\Z + \frac{\pi}{4}}, 
\end{align*} 
where $c_k$, $C_k$ and $T_k$ are explicit functions of $k$ (see Theorem \ref{thm:hypsecant}).

As a consequence we find probabilistic interpretations of the Jacobi elliptic functions $\cn$ and $\dn$ (see Appendix \ref{sec:ellipticappendix}) as characteristic functions of winding angles,
\[
\expec \exp\left( i b \{ \theta^\square_{\zeta_k+1/2} \}_{\pi\Z+\frac{\pi}{2}} \right) = \cn(K(k) b, k), \qquad  \expec \exp\left( i b \{ \theta^\square_{\zeta_k-1/2} \}_{\pi\Z} \right) = \dn(K(k) b, k).
\]
Here $\{\cdot\}_{A}$ denotes rounding to the nearest element of $A\subset\R$ and we set $\theta_{-1/2}^\square = 0$ by convention.

Since $\cn(y,1) = \dn(y,1) = \sech(y) = \expec e^{iy\eta}$ is the characteristic function of the aforementioned hyperbolic secant distribution, we may directly conclude the convergence in distribution as $k\to 1$ of the winding angle $\theta^\square_{\zeta_k+1/2}/K(k)$ at geometric time $\zeta_k$.
A more delicate singularity analysis, which is beyond the scope of this paper, would yield the same distributional limit for $2\theta^\square_{j+1/2}/\log(j)$ as $j\to\infty$, in accordance with previous work.

\paragraph{Application: Loops}
The last application we discuss utilizes the fact that the eigenvalues of the operators in Theorem \ref{thm:mainresult} have much simpler expressions than the components $\DirProd{e_p}{f_m}$ of the eigenvectors.
It is therefore worthwhile to seek combinatorial interpretations of traces of (combinations of) operators, the values of which only depend on the eigenvalues.
When writing out the trace in terms of the basis $(e_p)_{p=1}^\infty$ it is clear that such an interpretation must involve walks that start and end at arbitrary but equal distance from the origin.
If the full winding angle is taken to be a multiple of $2\pi$ then such a walk forms a \emph{loop}, i.e. it starts and ends at the same point.

A natural combinatorial set-up is described in Section \ref{sec:loops}.
There we consider the set $\mathcal{L}_n$ of \emph{rooted loops of index $n$}, $n\in\Z\setminus\{0\}$, which are simple diagonal walks avoiding the origin that start and end at an arbitrary but equal point in $\Z^2$ and have winding angle $2\pi n$ around the origin.
The set $\mathcal{L}_n = \mathcal{L}_n^{\text{even}} \cup \mathcal{L}_n^{\text{odd}}$ naturally partitions into loops that visit only sites of even respectively odd parity ($(x,y)\in\Z^2$ with $x+y$ even respectively odd), see Figure \ref{fig:rootedloops}.
We introduce the generating functions 
\begin{equation*}
L_n(t) \coloneqq \sum_{w\in\mathcal{L}_n}\frac{t^{|w|}}{|w|} \qquad L_n^{\text{odd}}(t) =\sum_{w\in\mathcal{L}^{\text{odd}}_n}\frac{t^{|w|}}{|w|}, \qquad  L_n^{\text{even}}(t) = \sum_{w\in\mathcal{L}^{\text{even}}_n}\frac{t^{|w|}}{|w|},
\end{equation*}
where we have included a factor $1/|w|$ for convenience, such that the generating functions of the set $\mathcal{L}_n$ is actually $t L_n'(t)$.
Observe that $L_1(t)$ is precisely the generating function of \emph{unrooted} loops of index $1$, i.e. rooted loops modulo rerooting (but preserving orientation), because the equivalence class of a rooted loop $w$ of index $1$ contains precisely $|w|$ elements. 
This is not true for $n \geq 2$, since $\mathcal{L}_n$ contains rooted loops $w$ that cover themselves multiple times and thus have equivalence classes with less than $|w|$ elements. 

\begin{figure}[t]
	\begin{subfigure}[c]{0.5\textwidth}
		\centering
		\includegraphics[height=.55\linewidth]{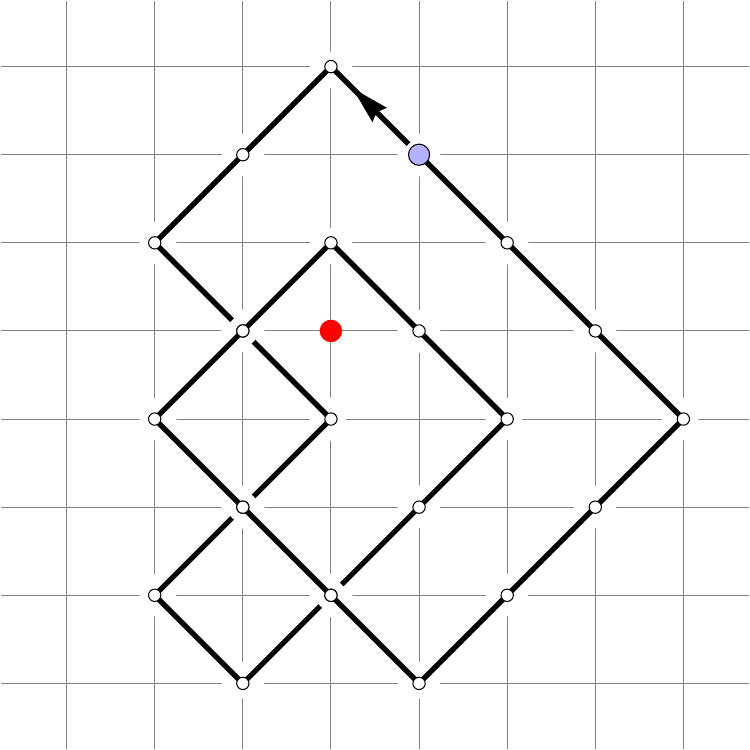}
		\caption{$\mathcal{L}_2^{\text{odd}}$}\label{fig:rootedloops1}
	\end{subfigure}%
	\begin{subfigure}[c]{0.5\textwidth}
		\centering
		\includegraphics[height=.55\linewidth]{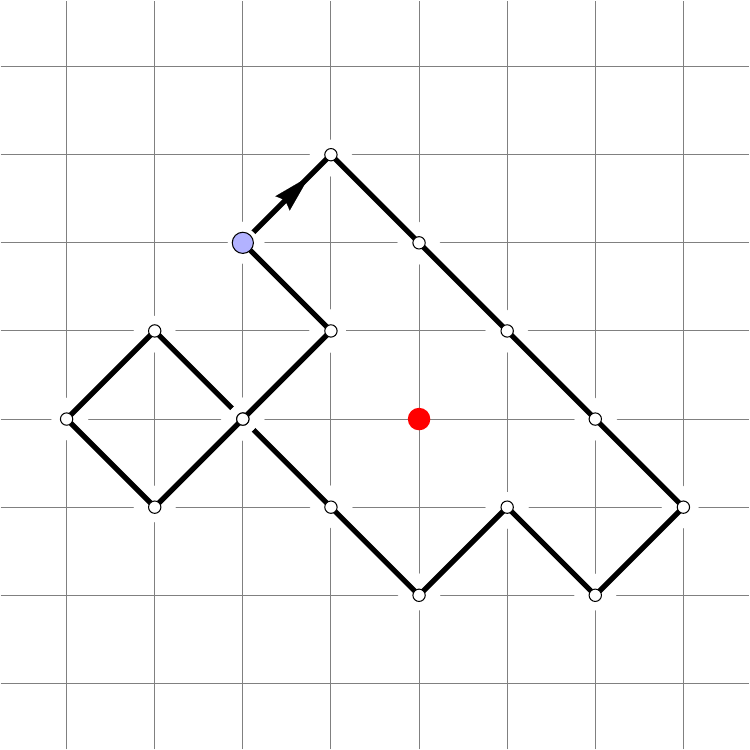}
		\caption{$\mathcal{L}_{-1}^{\text{even}}$}\label{fig:rootedloops2}
	\end{subfigure}%
	\caption{Two examples of rooted loops with different index and parity. }\label{fig:rootedloops}
\end{figure}

Theorem \ref{thm:loopgenfun} states that these generating functions for $n>0$ are given by
\[
L_n(t) = \frac{1}{n} \tr \mathbf{J}_k^{(2\pi n,-\infty)} = \frac{1}{n} \frac{q_k^{2n}}{1-q_k^{2n}}, \qquad L_n^{\text{odd}}(t) = \frac{1}{n} \frac{q_k^{2n}}{1-q_k^{4n}}, \qquad  L_n^{\text{even}}(t) = \frac{1}{n} \frac{q_k^{4n}}{1-q_k^{4n}}.
\]
Here $\mathbf{J}_k^{(2\pi n,-\infty)}$ is the operator $\mathbf{J}_k^{(\alpha,\beta_-)}$ with $\alpha = 2\pi n$ and $\beta_-=-\infty$ appearing in Theorem \ref{thm:mainresult}(ii) whose $m$th eigenvalue $q_k^{2 n m}$ can be obtained as the $\beta_-\to-\infty$ limit of the displayed expression (see the final remark of Theorem \ref{thm:mainresult}(ii) ). 
Note as before that these imply the equivalent identities on the level of power series in $\R[\![t]\!]$. 

\begin{figure}[b]
	\centering
	\includegraphics[height=.3\linewidth]{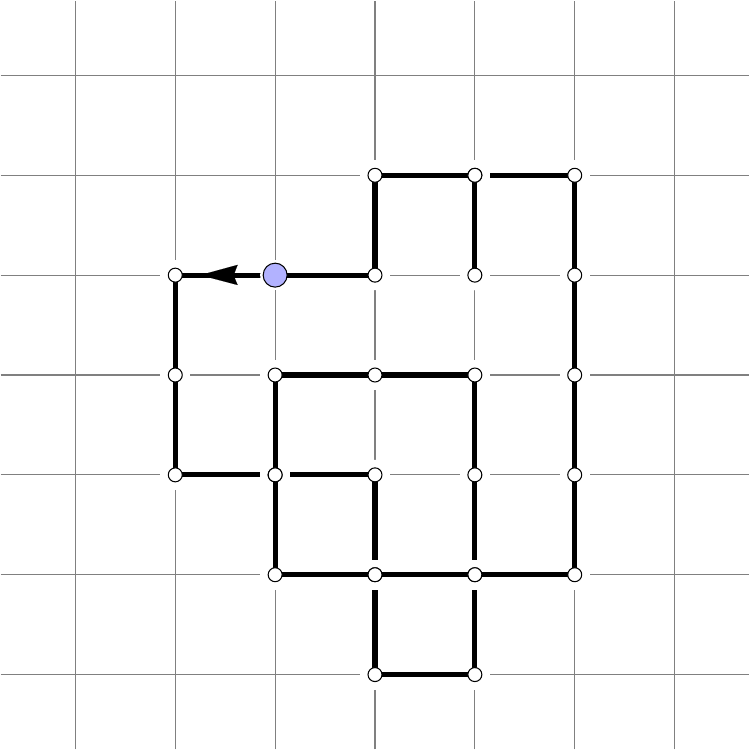}
	\hspace{-0.1cm}
	\includegraphics[height=.3\linewidth]{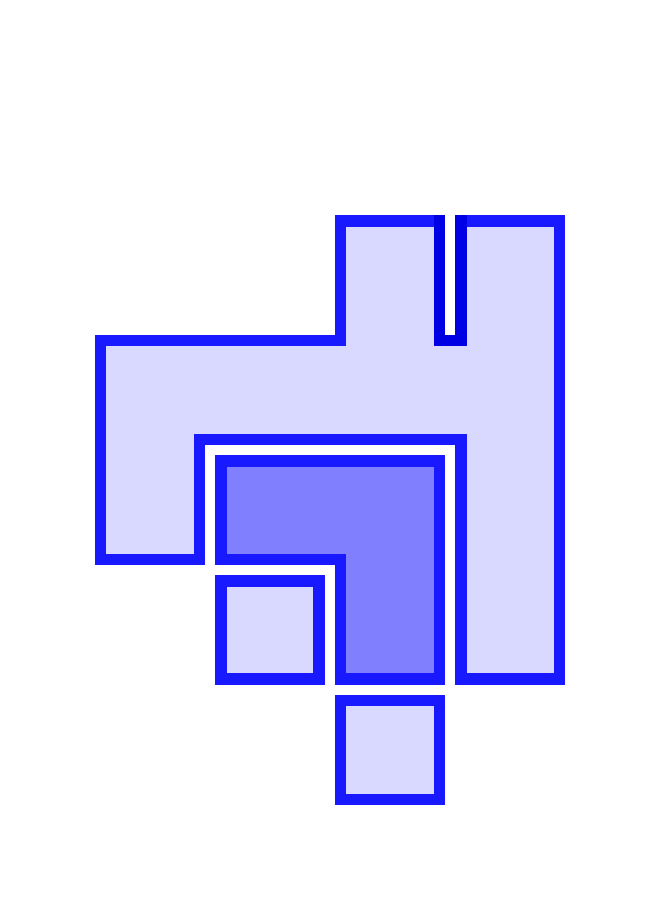}
	\caption{A simple walk on $\Z^2$ together with its clusters of index $1$ (light blue) and $2$ (dark blue). The largest cluster $c$ has area $|c|=9$ and boundary length $|\partial c|=20$ (notice that both sides of the ``slit'' contribute to the length). }\label{fig:loopcomponents}
\end{figure}%

A simple probabilistic consequence is the following. 
Let $(W_i)_{i=0}^{2\ell}$ be a simple random walk on $\Z^2$ started at the origin and conditioned to return after $2\ell$ steps.
For each point $z\in\R^2$ we let the \emph{index} $I_z$ be the signed number of times $(W_i)_{i=0}^{2\ell}$ winds around $z$ in counterclockwise direction, i.e. $2\pi I_z$ is the winding angle of $(W_i)_{i=0}^{2\ell}$ around $z$.
If $z$ lies on the trajectory of $(W_i)_{i=0}^{2\ell}$, then we set $I_z=\infty$.
We let \emph{the clusters $C_n$ of index $n$} be the set of connected components of $\{z\in\R^2:I_z=n\}$, and for $c\in C_n$ we let $|c|$ and $|\partial c|$ respectively be the area and boundary length of component $c$.
See Figure \ref{fig:loopcomponents} for an example.
The expectation values for the total area and boundary length of the clusters of index $n$ then satisfy
\begin{align*}
\expec\left[ \sum_{c\in C_n} |c|\right] &= \frac{4^{2\ell}}{\binom{2\ell}{\ell}^2}\frac{2\ell}{n} [k^{2\ell}]\frac{q_k^{2n}}{1-q_k^{4n}} \sim \frac{\ell}{2\pi n^2},\\
\expec\left[ \sum_{c\in C_n} (|\partial c|-2)\right] &= \frac{4^{2\ell}}{\binom{2\ell}{\ell}^2}\frac{4\ell}{n} [k^{2\ell}]\frac{q_k^{2n}}{1+q_k^{2n}} \sim \frac{2\pi^3 \ell}{\log^2 \ell},
\end{align*}
with the asymptotics as $\ell\to\infty$ indicated.
The first result should be compared to the analogous result for Brownian motion: Garban and Ferreras proved in \cite{garban_expected_2006} using Yor's work \cite{yor_loi_1980} that the expected area of the set of points with index $n$ with respect to a unit time Brownian bridge in $\R^2$ is equal to $1/(2\pi n^2)$.
Perhaps surprisingly, we find that the expected boundary length all the components of index $n$ (minus twice the expected number of such components) grows asymptotically at a rate that is independent of $n$, contrary to the total area.

\begin{openquestion}
Does $\expec\left[\sum_{c\in C_0\text{ finite}} |c|\right]$, i.e. the total area of the finite clusters of index $0$, have a similarly explicit expression? Based on the results of \cite{garban_expected_2006} we expect it to be asymptotic to $\pi \ell/30$ as $\ell\to\infty$. 	
\end{openquestion}

Finally we mention one more potential application of the enumeration of loops in Theorem \ref{thm:loopgenfun} in the context of \emph{random walk loop soups} \cite{lawler_random_2007}, which are certain Poisson processes of loops on $\Z^2$.
A natural quantity to consider in such a system is the \emph{winding field} which roughly assigns to any point $z\in\R^2$ the total index of all the loops in the process \cite{camia_conformal_2016,camia_brownian_2016}.
Theorem \ref{thm:loopgenfun} may be used to compute explicit expectation values (one-point functions of the corresponding vertex operators to be precise) in the massive version of the loop soups.
We will pursue this direction elsewhere.

\paragraph{Discussion} 
The connection between the enumeration of walks and the explicitly diagonalizable operators on Dirichlet space may seem a bit magical to the reader. 
So perhaps some comments are in order on how we arrived at this result, which originates in the combinatorial study of planar maps.

A \emph{planar map} is a connected multigraph (a graph with multiple edges and loops allowed) that is properly embedded in the $2$-sphere (edges are only allowed to meet at their extremities), viewed up to orientation-preserving homeomorphisms of the sphere. 
The connected components of the complement of a planar map are called the \emph{faces}, which have a \emph{degree} equal to the number of bounding edges. 
There exists a relatively simple multivariate generating function for \emph{bipartite} planar maps, i.e. maps with all faces of even degree, that have two distinguished faces of degree $p$ and $\ell$ and a fixed number of faces of each degree (see e.g. \cite{collet_simple_2014}).
The surprising fact is that this generating function has a form that is very similar to that of the generating function $W_{\ell,p}^{(\pi,(-\pi,\pi))}(t)$ of diagonal walks from $(p,0)$ to $(-\ell,0)$ that avoid the slit $\{(x,0) : x\leq 0\}$ until the end.
A combinatorial explanation (in the particular case of critical planar maps) appears in \cite{budd_peeling_2018} using a peeling exploration \cite{budd_peeling_2016,budd_geometry_2017}, 

If one further decorates the planar maps by a \emph{rigid $O(n)$ loop model} \cite{borot_recursive_2012}, then the combinatorial relation extends to one between walks of fixed winding angle $\mathcal{W}_{\ell,p}^{(\alpha)}$ with $\alpha\in\pi\Z$ and planar maps with two distinguished faces and certain collections of non-intersecting loops separating the two faces.
The combinatorics of the latter has been studied in considerable detail in \cite{borot_recursive_2012,borot_loop_2012,borot_nesting_2016,budd_peeling_2018}, which has inspired our treatment of the simple walks on $\Z^2$ in this paper.
Further details on the connection and an extension to more general lattice walks with \emph{small steps} (i.e. steps in $\{-1,0,1\}^2$) will be provided in forthcoming work. 

Finally, we point out that these methods can be used to determine Green's functions (and more general resolvents) for the Laplacian  on regular lattices in the presence of a conical defect or branched covering, which are relevant to the study of various two-dimensional statistical systems.
As an example, in recent work \cite{kenyon_greens_2017} Kenyon and Wilson computed the Green's function on the branched double cover of the square lattice, which has applications in local statistics of the uniform spanning tree on $\Z^2$ as well as dimer systems.  

\subsection*{Acknowledgements}
The author would like to thank Kilian Raschel, Alin Bostan and Ga\"etan Borot for their suggestions on how to prove Corollary \ref{thm:gessel}, and Thomas Prellberg for suggesting to extend Theorem \ref{thm:hypsecant} to absorbing boundary conditions.
The author is indebted to an anonymous referee for numerous corrections and suggestions to improve the exposition.
This work was supported by a public grant as part of the Investissement d’avenir project,
reference ANR-11-LABX-0056-LMH, LabEx LMH, as well as ERC-Advanced grant 291092, “Exploring the Quantum Universe” (EQU).
Part of this work was done while the author was at the Niels Bohr Institute, University of Copenhagen and Institut de Physique Théorique, CEA, Université Paris-Saclay.

\section{Winding angle of walks starting and ending on an axis}\label{sec:walksaxis}

Our strategy towards proving Theorem \ref{thm:mainresult} will be to first prove part (ii) for three special cases (see Figure \ref{fig:JBAwalks}), 
\[
\mathcal{J}_{\ell,p} \coloneqq \mathcal{W}_{\ell,p}^{(\frac{\pi}{2},(-\frac{\pi}{2},\frac{\pi}{2}))},\qquad  \mathcal{B}_{\ell,p} \coloneqq \mathcal{W}_{\ell,p}^{(0,(-\frac{\pi}{2},\frac{\pi}{2}))}, \qquad  \mathcal{A}_{\ell,p} \coloneqq \mathcal{W}_{\ell,p}^{(\frac{\pi}{2},(0,\frac{\pi}{2}))}.
\]
\begin{figure}[t]
	\begin{subfigure}[b]{0.333\textwidth}
		\centering
		\includegraphics[width=.9\linewidth]{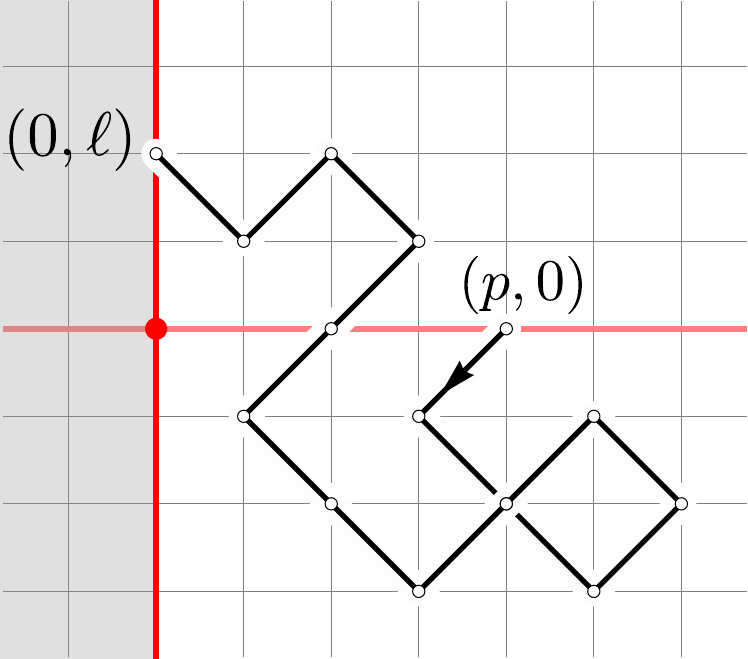}
		\caption{$\mathcal{J}_{\ell,p}$,\,\, $J_{\ell,p}(t) = \frac{1}{\ell} \DirProd{e_\ell}{\mathbf{J}_k e_p}$}\label{fig:jwalk}
	\end{subfigure}%
	\begin{subfigure}[b]{0.333\textwidth}
		\centering
		\includegraphics[width=.9\linewidth]{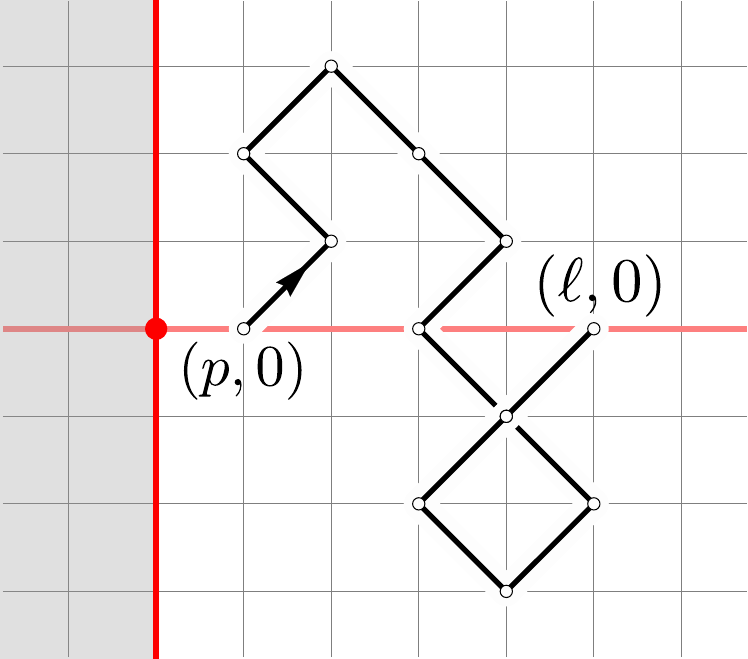}
		\caption{$\mathcal{B}_{\ell,p}$,\,\, $B_{\ell,p}(t)=\DirProd{e_\ell}{\mathbf{B}_k e_p}$}\label{fig:bwalk}
	\end{subfigure}%
	\begin{subfigure}[b]{0.333\textwidth}
		\centering
		\includegraphics[width=.9\linewidth]{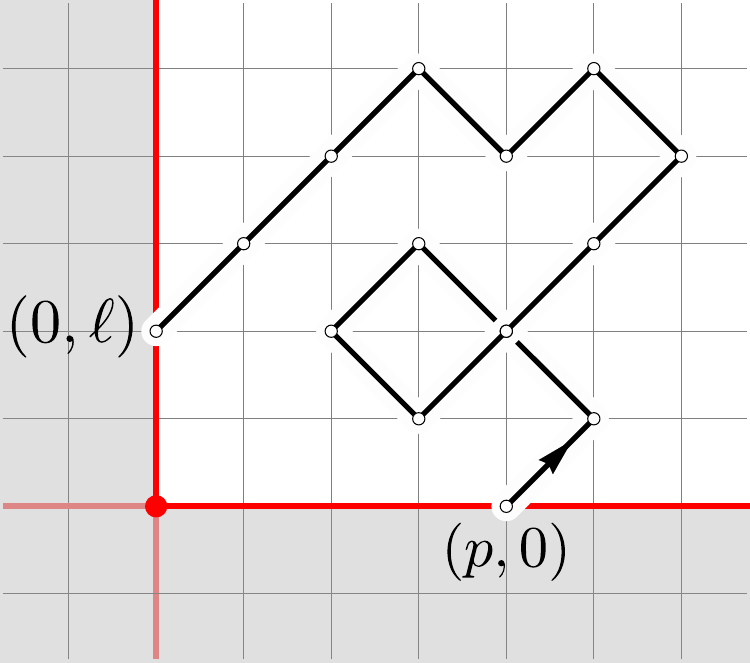}	
		\caption{$\mathcal{A}_{\ell,p}$,\,\, $ A_{\ell,p}(t)=\frac{1}{\ell p} \DirProd{e_\ell}{\mathbf{A}_k e_p}$}\label{fig:awalk}
	\end{subfigure}%
	\caption{Examples of walks contributing to the three types of building blocks. }\label{fig:JBAwalks}
\end{figure}%
We define three linear operators $\mathbf{J}_k$, $\mathbf{B}_k$ and $\mathbf{A}_k$ on $\Dir$ by specifying their matrix elements with respect to the standard basis $(e_i)_{i\geq 1}$ in terms of the corresponding generating functions $J_{\ell,p}(t)$, $B_{\ell,p}(t)$ and $A_{\ell,p}(t)$ (with $k=4t$) as
\[
\frac{1}{\ell} \DirProd{e_\ell}{\mathbf{J}_k e_p} = J_{\ell,p}(t), \qquad \DirProd{e_\ell}{\mathbf{B}_k e_p} = B_{\ell,p}(t), \qquad \frac{1}{\ell p} \DirProd{e_\ell}{\mathbf{A}_k e_p} = A_{\ell,p}(t).
\]
The motivation to define the operators in this way is given in Proposition \ref{thm:JAB}.

The results in this section will require some Hilbert space terminology (see \cite[Chapter IV]{reed1981functional} or \cite[Chapter 1]{zhu_operator_2007} for introductions).
A linear operator $\mathbf{T}$ on a Hilbert space $\mathcal{H}$ is \emph{bounded} if there exists a constant $C>0$ such that $\|\mathbf{T} y\|_{\mathcal{H}} \leq C \|y\|_{\mathcal{H}}$ for all $y\in \mathcal{H}$.
A bounded operator $\mathbf{T}$ is said to be \emph{compact} if $\mathbf{T}(\{y\in\mathcal{H} : \|y\|_{\mathcal{H}} \leq 1\})$ is compact in the \emph{norm topology} on $\mathcal{H}$, i.e. the topology on $\mathcal{H}$ induced by the distance $d(x,y) = \| x - y\|_{\mathcal{H}}$.
Finally, a linear operator $\mathbf{T}$ is \emph{self-adjoint} if $\langle x, \mathbf{T} y\rangle_{\mathcal{H}} = \langle \mathbf{T} x, y\rangle_{\mathcal{H}}$ for all $x,y \in \mathcal{H}$. 
For our purposes, the most useful characterization is that $\mathbf{T}$ is self-adjoint and compact if and only if there exists an orthonormal basis $(a_n)_{n=1}^\infty$ of $\mathcal{H}$ and a sequence of real numbers $(\lambda_n)_{n=1}^\infty$ satisfying $\lambda_n\to 0$ as $n\to\infty$ such that $\mathbf{T} y = \sum_{n=1}^\infty \lambda_n \langle y, a_n\rangle_\mathcal{H} a_n$ for all $y\in\mathcal{H}$.

\begin{proposition}\label{thm:JAB}
The linear operators $\mathbf{J}_k$, $\mathbf{B}_k$, $\mathbf{A}_k$ on $\Dir$ are bounded and satisfy $\mathbf{J}_k = \mathbf{A}_k\mathbf{B}_k$.
Moreover $\mathbf{J}_k$ and $\mathbf{A}_k$ are compact.
\end{proposition}
\begin{proof}
First we verify that for fixed $t\in (0,1/4)$ we have an exponential bound $A_{\ell,p}(t) < C e^{-\ell-p}$ for some $C>0$. 
To this end observe that $\mathcal{A}_{\ell,p}$ can be identified with a subset of (unconstrained) diagonal walks starting at $0$ and ending on the line $\{ (x,y) \in \Z^2 : y - x = p+\ell\}$.
Hence 
\begin{equation}\label{eq:Abound}
A_{\ell,p}(t) \leq \sum_{m,n=0}^\infty t^{m+n}\, 2^n \binom{m+n}{m}\binom{m}{\frac{m+p+\ell}{2}}\one_{\{\frac{m+p+\ell}{2}\in\Z\}} = \frac{1}{\sqrt{1-4t}} \left( \frac{1-2t-\sqrt{1-4t}}{2t}\right)^{\ell+p},
\end{equation}
which falls off exponentially in $\ell+p$.
It follows that $\mathbf{A}_k$ has finite Hilbert-Schmidt norm,
\[
\sum_{p=1}^\infty \frac{\| \mathbf{A}_k e_p \|_\Dir^2}{\|e_p\|_\Dir^2} = \sum_{p,\ell=1}^\infty p\, \ell\, A_{\ell,p}^2(t) <\infty.
\]
In particular, $\mathbf{A}_k$ is bounded and compact, see e.g. \cite[Theorem VI.22(e)]{reed1981functional}.

On the other hand, the matrix elements of $\mathbf{B}_k$ with respect to the orthonormal basis $(e_p/\sqrt{p})_{p\geq 1}$ are given by $B_{\ell,p}(t)/\sqrt{\ell p} \leq B_{\ell,p}(t) \leq b_{\ell-p}$, where $b_n$ is the generating function of (unconstrained) diagonal walks from the origin to $(n,0)$.
In particular $b_n$ is smaller than the right-hand side of \eqref{eq:Abound} when $\ell+p = n$ and therefore falls off exponentially, ensuring that $(b_{n})_{n\in\Z}$ are the Fourier coefficients of some continuous real $2\pi$-periodic function $\phi : \R\to\R$.
By a classical result \cite[Appendix]{hartman_spectra_1954} the Toeplitz operator associated to $\phi$, i.e. an operator with matrix elements $(b_{\ell-p})_{\ell,p\geq 1}$, is bounded because $\phi$ is bounded.
Since the operator norm of this Toeplitz operator bounds that of $\mathbf{B}_k$, $\mathbf{B}_k$ is also bounded.

For $\ell,p,m\geq 1$, composition of walks determines a bijection between pairs of walks in $\mathcal{A}_{\ell,m}\times \mathcal{B}_{m,p}$ and walks in $\mathcal{J}_{\ell,p}$ for which the last intersection with the horizontal axis occurs at $(m,0)$.
Hence
\[
\frac{1}{\ell}\DirProd{e_\ell}{\mathbf{J}_ke_p} = J_{\ell,p}(t) = \sum_{m=1}^\infty A_{\ell,m}(t)B_{m,p}(t) = \sum_{m=1}^\infty \frac{1}{\ell m}\DirProd{e_\ell}{\mathbf{A}_ke_m}\DirProd{e_m}{\mathbf{B}_ke_p} = \frac{1}{\ell} \DirProd{e_\ell}{\mathbf{A}_k\mathbf{B}_ke_p},
\]
implying that $\mathbf{J}_k = \mathbf{A}_k\mathbf{B}_k$. 
Since $\mathbf{A}_k$ is compact and $\mathbf{B}_k$ is bounded, their composition $\mathbf{J}_k$ is bounded and compact \cite[Theorem 1.15]{zhu_operator_2007}.
\end{proof}

\subsection{The operator $\mathbf{J}_k$} \label{sec:operatorJ}

We wish to enumerate the walks $w\in\mathcal{J}_{\ell,p}$, $p,\ell\geq 1$, that start at $(p,0)$ and end at $(0,\ell)$ while maintaining strictly positive first coordinate until the end.
By looking at both coordinates of the walks separately, we easily see that these walks are in bijection with pairs of simple walks of equal length on $\Z$, the first of which starts at $p$ and ends at $0$ while staying positive until the end, while the second starts at $0$ and ends at $\ell$ without further restrictions.
For fixed length $n$, such walks only exist if both $n-p$ and $n-\ell$ are non-negative even integers, in which case the ballot theorem tells us that the former walks are counted by $\frac{p}{n}\binom{n}{(n+p)/2}$ and the latter by $\binom{n}{(n+\ell)/2}$.
Therefore the generating function $J_{\ell,p}(t)$ is given explicitly by
\begin{equation}\label{eq:Jexpression}
J_{\ell,p}(t) = \sum_{n=1}^\infty \frac{p}{n}\binom{n}{(n+\ell)/2}\binom{n}{(n+p)/2} \one_{\{n-p\text{ and }n-\ell\text{ non-negative and even}\}}t^{n} .
\end{equation}
It is non-trivial only when $p+\ell$ is even, in which case it has radius of convergence equal to $1/4$.

For fixed $k=4t\in(0,1)$ we denote by $\psi_k : \disc \to \C$ the analytic function given by
\[
\psi_k(x) = \frac{1-\sqrt{1-k x^2}}{\sqrt{k}\, x},
\]
which maps the unit disk $\disc$ biholomorphically onto a strict subset $\psi_k(\disc)\subset \disc$, its inverse being given by $2/(\sqrt{k}(z+1/z))$ for $z\in\psi_k(\disc)$.
It induces a linear operator $\mathbf{R}_k$ on the Dirichlet space $\mathcal{D}$ of analytic functions $f$ on $\disc$ that vanish at the origin by setting $\mathbf{R}_k f \coloneqq f \circ \psi_k$.

\begin{lemma}\label{thm:RJoperator}
	The linear operator $\mathbf{R}_k$ is bounded and $\mathbf{J}_k = \mathbf{R}_k^\dagger\mathbf{R}_k$.
	In particular, $\mathbf{J}_k$ is self-adjoint and injective.
\end{lemma}
\begin{proof}
Since the Dirichlet inner product is preserved under conformal mapping, in the sense that $\langle f \circ \psi, g\circ \psi\rangle_{\Dir(\disc)} = \langle f, g\rangle_{\Dir(\psi(\disc))}$ for any biholomorphic function $\psi$ on $\disc$ \cite[Section 2.1]{arcozzi_dirichlet_2011}, we have 
\[
\| \mathbf{R}_kf\|_{\Dir}^2 = \| f\circ \psi_k\|_{\Dir}^2 = \int_{\psi_k(\disc)} |f'(z)|^2 \rmd A(z) \leq \| f \|_{\Dir}^2,
\]
implying that $\mathbf{R}_k$ is bounded. 

For $n-p$ non-negative and even we may compute
\begin{equation}\label{eq:lagrangeinversion}
[x^n] \psi_k(x)^p = \left(\frac{k}{4}\right)^{n/2} [z^n] \left( \frac{1-\sqrt{1-4z^2}}{2z}\right)^p =  \left(\frac{k}{4}\right)^{n/2} \frac{p}{n} [u^{n-p}] \left(1+u^2\right)^n = \left(\frac{k}{4}\right)^{n/2} \frac{p}{n} \binom{n}{(n+p)/2},
\end{equation}
where in the second equality we applied Lagrange inversion to the Catalan series $(1-\sqrt{1-4z^2})/(2z)$.
Therefore 
\begin{align*}
\DirProd{e_\ell}{\mathbf{R}_k^\dagger\mathbf{R}_k e_p} &=\DirProd{\mathbf{R}_ke_\ell}{\mathbf{R}_k e_p} = \sum_{n=1}^\infty n\, \overline{[x^n]\psi_k(x)^\ell}\, [x^n]\psi_k(x)^p \\
&= \sum_{n=1}^\infty \left(\frac{k}{4}\right)^n n \, \frac{\ell}{n} \binom{n}{(n+\ell)/2} \frac{p}{n} \binom{n}{(n+p)/2}\one_{\{n-p\text{ and }n-\ell\text{ non-negative and even}\}}\\
	& = \ell \, J_{\ell,p}(t) = \DirProd{e_\ell}{\mathbf{J}_k e_p},
\end{align*}
which shows that $\mathbf{J}_k = \mathbf{R}_k^\dagger \mathbf{R}_k$.
It follows that $\mathbf{J}_k$ is self-adjoint and injective, since $\langle f, \mathbf{J}_k f\rangle_\Dir = \int_{\psi_k(\disc)} |f'(z)|^2 \rmd A(z)= 0$ iff $f$ is the zero function. 
\end{proof}

The injectivity and self-adjointness of $\mathbf{J}_k$ imply the following very useful property for the triple of operators $\mathbf{J}_k$, $\mathbf{A}_k$ and $\mathbf{B}_k$.

\begin{lemma}\label{thm:simdiag}
The operators $\mathbf{J}_k$, $\mathbf{A}_k$ and $\mathbf{B}_k$ are self-adjoint, commuting, and simultaneously diagonalizable.
\end{lemma}
\begin{proof}
It is clear from the definitions that $\mathbf{B}_k$ and $\mathbf{A}_k$ are self-adjoint, and according to Lemma \ref{thm:RJoperator} the same is true for $\mathbf{J}_k$.
The relation $\mathbf{J}_k = \mathbf{A}_k\mathbf{B}_k$ from Proposition \ref{thm:JAB} then implies that all three operators (mutually) commute.
Since both $\mathbf{J}_k$ and $\mathbf{A}_k$ are compact, they are simultaneously diagonalizable \cite[Corollary 3.2.10]{zimmer_essential_1990}, in the sense that there exists an orthonormal basis $(b_m)_{m=1}^\infty$ of $\Dir$ such that each $b_m$ is an eigenvector of both $\mathbf{J}_k$ and $\mathbf{A}_k$. 
According to Lemma \ref{thm:RJoperator}, $\mathbf{J}_k$ is injective and therefore $\mathbf{A}_k$ must be too.
Observe that, for any $m\geq 1$, $\mathbf{J}_k b_m = \mathbf{B}_k \mathbf{A}_k b_m$ with $\mathbf{A}_k b_m \neq 0$, and therefore $b_m$ is an eigenvector of $\mathbf{B}_k$ too.
\end{proof}

The remainder of the subsection is devoted to diagonalizing $\mathbf{J}_k$, while in Section \ref{sec:operatorBA} we will determine the corresponding eigenvalues for the remaining two operators $\mathbf{A}_k$ and $\mathbf{B}_k$.
Finally, in Section \ref{sec:proofmainresult} we will prove Theorem \ref{thm:mainresult} by taking suitable compositions of the operators $\mathbf{J}_k$, $\mathbf{B}_k$, and $\mathbf{A}_k$. 

In order to diagonalize $\mathbf{J}_k$ it suffices to find an orthogonal basis $(f_m)_{m=1}^\infty$ of $\Dir$ consisting of analytic functions that are also orthogonal with respect to the Dirichlet inner-product on $\psi_k(\disc)$,
\[\langle f, g \rangle_{\Dir(\psi_k(\disc))} \coloneqq  \int_{\psi_k(\disc)} \overline{f'(z)}\,g'(z)\, \rmd A(z). \]
Indeed, if $\DirProd{f_n}{f_m} = \langle f_n,f_m\rangle_{\Dir(\psi_k(\disc))} = 0$ for $m\neq n$ then 
\begin{equation*}
\DirProd{f_n}{\mathbf{J}_kf_m} = \DirProd{f_n}{\mathbf{R}_k^\dagger \mathbf{R}_kf_m} = \langle f_n, f_m \rangle_{\Dir(\psi_k(\disc))} = \frac{\langle f_m,f_m\rangle_{\Dir(\psi_k(\disc))}}{\langle f_m,f_m\rangle_{\Dir} } \DirProd{f_n}{f_m}\qquad\text{ for all }m,n\geq 1,
\end{equation*}
implying that $f_m$ is an eigenvector of $\mathbf{J}_k$ with eigenvalue $\langle f_m,f_m\rangle_{\Dir(\psi_k(\disc))}/\langle f_m,f_m\rangle_{\Dir}$.

To obtain such a basis we seek an injective holomorphic mapping that takes both $\disc$ and $\psi_k(\disc)$ to sufficiently simple domains. 
As we will see the elliptic integral $v_{k_1}(z)$ introduced in \eqref{eq:vkmap} does this job.
First we notice (using \eqref{eq:ellipticintegrals} in Appendix \ref{sec:ellipticappendix}) that $v_k(z)$ can be expressed in terms of the inverse function $\arcsn(\cdot, k)$ (in a suitable neighbourhood of the origin) of the Jacobi elliptic function $\sn(\cdot,k)$  with modulus $k$, 
\begin{equation}
v_k(z) = \frac{1}{4K(k)} \int_0^z \frac{\rmd x}{\sqrt{(k-x^2)(1-k x^2)}} = \frac{\arcsn\left(\frac{z}{\sqrt{k}}, k\right)}{4K(k)}.
\end{equation}
Denote by $z_{k}$ the analytic function 
\begin{equation}\label{eq:zkdef}
z_{k}(v) = \sqrt{k} \sn( 4 K(k) v, k),
\end{equation}
which will provide the inverse to $v_{k}$ on a suitable domain (see Lemma \ref{thm:mapping} below).
As depicted in Figure \ref{fig:vk1map} and proved in the next lemma, after the removal of two slits $v_{k_1}$ maps both $\disc$ and $\psi_k(\disc)$ to rectangles, with the same width but different heights. 
 
\begin{figure}[ht]
	\centering
	\includegraphics[width=.7\linewidth]{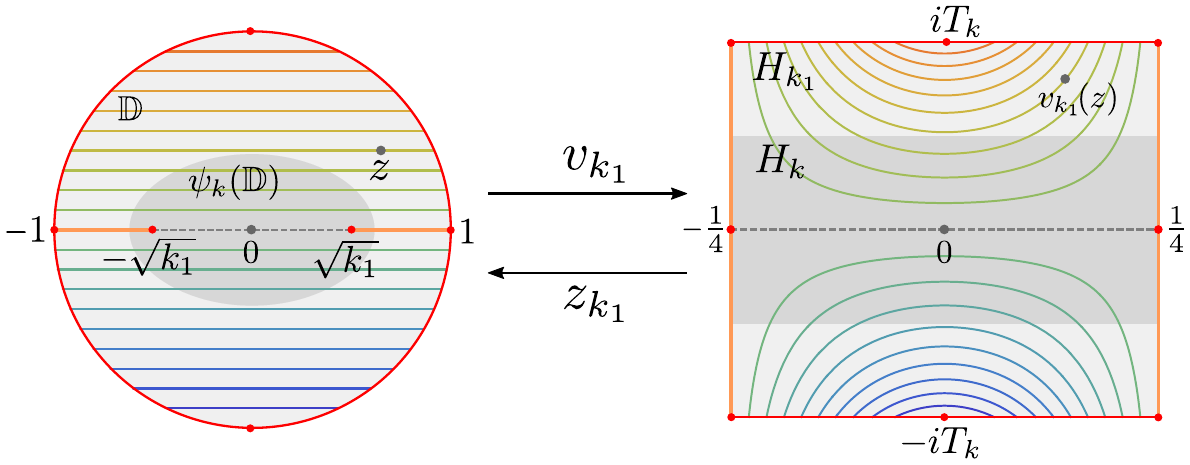}
	\caption{The analytic function $v_{k_1}$ maps the disk with two slits onto a rectangle of width $\frac12$ and height $2 T_k = K(k')/(2K(k))$. The shaded domains represent $\psi_k(\disc)$ and its image $H_{k}$ under $v_{k_1}$.\label{fig:vk1map}}
\end{figure}

\begin{lemma}\label{thm:mapping}
The analytic function $v_{k_1}$ maps $\disc\setminus \{z\in\R : z^2 \geq k_1\}$, respectively $\psi_k(\disc) \setminus \{z\in\R : z^2 \geq k_1\}$, biholomorphically onto the open rectangle $H_{k_1}\coloneqq (-1/4,1/4) + i ( -T_{k_1}/2, T_{k_1}/2 )$, respectively $H_{k}\coloneqq(-1/4,1/4) + i ( -T_k/2, T_k/2 )$, where $T_k = K(k')/(4K(k))$ satisfies $T_{k_1}=2 T_k$ (see \eqref{eq:landenk} in Appendix \ref{sec:ellipticappendix}).
The inverse mapping $z_{k_1}(v)$ extends to a surjective holomorphic map from the infinite strip $\R + i(-T_k,T_k)$ to $\disc$. 
Moreover, $z_{k_1}(-v) = z_{k_1}(v+1/2) = - z_{k_1}(v)$ and 
\begin{enumerate}[label = {(\roman*)}]
\item $z_{k_1}$ maps $[-\tfrac{1}{4}, \tfrac{1}{4}]$ bijectively to $ [-\sqrt{k_1},\sqrt{k_1}]$ with $z_{k_1}(0)=0$ and $z_{k_1}(\pm \tfrac{1}{4}) = \pm \sqrt{k_1}$;
\item $z_{k_1}$ maps $[-\tfrac{1}{4}, \tfrac{1}{4}]+ i T_k$ bijectively to $\{v\in \C : |v|=1,\, \im(v) \geq 0\}$ with $z_{k_1}(\pm\tfrac{1}{4} + i T_k) = \pm1$ and $z_{k_1}(i T_k) = i$;
\item $z_{k_1}$ maps $\pm\tfrac{1}{4} + i[0,T_k]$ bijectively to $\pm[\sqrt{k_1},1]$ with $z_{k_1}(\pm\tfrac{1}{4})=\pm \sqrt{k_1}$ and $z_{k_1}(\pm \tfrac{1}{4}+iT_k) = \pm 1$.
\end{enumerate}
\end{lemma}
\begin{proof}
It is well-known that $\sn(\cdot,k)$ maps the open rectangle $(-K(k),K(k)) + i (0, K(k'))$ biholomorphically onto the upper half plane (see e.g. \cite[\S 47]{akhiezer_elements_1990}) with the boundary $[-K(k),K(k)]$ mapped bijectively to $[-1,1]$ and the boundaries $\pm K(k) + i [0, K(k')]$ to $\pm [1,k]$.
Hence $z_{k_1}$ as defined in \eqref{eq:zkdef} maps $(-\tfrac{1}{4},\tfrac{1}{4}) + i (0,T_{k_1})$ biholomorphically onto the upper half plane with the boundary $[-\tfrac{1}{4},\tfrac{1}{4}]$ mapped bijectively to $[-\sqrt{k_1},\sqrt{k_1}]$ and the boundaries $\pm\tfrac{1}{4} + i [0,T_{k_1}]$ to $\pm[\sqrt{k_1},1/\sqrt{k_1}]$.
Moreover the pseudo-periodicity and oddness of $\sn(\cdot,k)$ (see Appendix \ref{sec:ellipticappendix}) imply that $z_{k_1}(-v) = z_{k_1}(v+1/2) = -z_{k_1}(v)$ as well as $z_{k_1}(0)=0$, showing property (i).

According to \eqref{eq:snshift} and \eqref{eq:kprime} we have the identity $\sn(u,k)\sn(u+iK(k'),k) = 1/k$, from which it follows by setting $u = x - i K(k')/2$ with $x\in\R$ that $|\sn(x + i K(k')/2,k)| = 1/\sqrt{k}$.
Hence, $|z_{k_1}(x + iT_k)| = 1$ and therefore $z_{k_1}$ must map $[-\tfrac{1}{4}, \tfrac{1}{4}]+ i T_k$ bijectively to $\{v\in \C : |v|=1,\, \im(v) \geq 0\}$.
The part of the rectangle $(-\tfrac{1}{4},\tfrac{1}{4}) + i (0,T_{k_1})$ lying below the line $[-\tfrac{1}{4}, \tfrac{1}{4}]+ i T_k$ is then mapped biholomorphically by $z_{k_1}$ to the open upper half disk. 
Together with property (i) and $z_{k_1}(-v) = -z_{k_1}(v)$ this shows that $z_{k_1}$ maps $H_{k_1}$ biholomorphically onto $\disc\setminus\{z\in\R : z^2\geq k_1\}$, providing the inverse of $v_{k_1}$ restricted to the latter set.
Thanks to the pseudo-periodicity $z_{k_1}(v + 1/2) = - z_{k_1}(v)$, the function $z_{k_1}$ maps the infinite strip $\R + i(-T_{k_1}/2,T_{k_1}/2)$ surjectively to $\disc$.
Furthermore, $\sn(u,k)\sn(u+iK(k'),k) = 1/k$ implies that $\sn(iK(k')/2,k)^2 = -1/k$ and therefore $z_{k_1}(iT_k) = i$, proving property (ii). 
Property (iii) is a direct consequence of the first two and the way $z_{k_1}$ maps the boundaries $\pm\tfrac{1}{4} + i [0,T_{k_1}]$.

It remains to show that $z_{k_1}$ maps $H_k$ to $\psi_k(\disc) \setminus \{z\in\R : z^2 \geq k_1\}$.
The descending Landen transformation \eqref{eq:landensn} relates the Jacobi elliptic functions $\sn(\cdot, k)$ and $\sn(\cdot, k_1)$ through
\[
\sn(u,k) = \frac{(1+k_1)\sn(u/(1+k_1),k_1)}{1+ k_1 \sn^2(u/(1+k_1),k_1)}. 
\]
From the arguments above we know that if $u \in (-K(k),K(k)) + i (-K(k')/2,K(k')/2)$ then $|\sn(u,k)| < 1/\sqrt{k}$, in which case we may invert the relation to
\begin{equation}\label{eq:psisnrelation}
\sqrt{k_1}\sn(u/(1+k_1),k_1) = \frac{1-\sqrt{1-k^2 \sn^2(u,k)}}{k\,\sn(u,k)} = \psi_k( \sqrt{k}\sn(u,k) ).
\end{equation}
The sign of the square root is determined by noticing, with the help of \eqref{eq:landenk}, that $\frac{u}{1+k_1} \in (-K(k_1),K(k_1) ) + i (-K'(k_1)/4,K'(k_1)/4)$ and therefore $|\sn(u/(1+k_1),k_1)| < 1/\sqrt{k_1}$ whereas $|\sn(u,k)| < 1/\sqrt{k}$.
Setting $u = 4K(k) v$ and using \eqref{eq:landenk} the relation \eqref{eq:psisnrelation} becomes $z_{k_1}(v) = \psi_k( z_k(v))$ valid for $v\in H_{k_1}$.
From our previous considerations, with $k_1$ replaced by $k$, we know that $z_k(H_k) = \disc \setminus \{z\in\R: z^2 \geq k\}$.
Since $\psi_k(\sqrt{k}) = \sqrt{k_1}$, we conclude that
\begin{equation*}
z_k(H_k) = \psi_k(\disc \setminus \{z\in\R: z^2 \geq k\}) = \psi_k(\disc) \setminus \{z\in\R: z^2 \geq k_1\},
\end{equation*}
which finishes the proof.
\end{proof}

Let $\mathcal{H}_{k_1}$ be the Hilbert space of analytic functions $g$ on the infinite strip $\R + i(-T_k,T_k)$ that satisfy $g(0)=0$ and $g(v+1) = g(v) = g(1/2-v)$ and that have finite norm with respect to the inner product defined via an integral over the rectangle $H_{k_1}$,
\begin{equation}
\langle g,h\rangle_{\mathcal{H}_{k_1}} = \int_{H_{k_1}} \overline{g'(v)}\, h'(v) \rmd A(v). 
\end{equation}

\begin{lemma}\label{thm:isomorphism}
The pullback $f \mapsto f\circ z_{k_1}$ determines an isomorphism $\Dir \to \mathcal{H}_{k_1}$ of Hilbert spaces.
\end{lemma}
\begin{proof}
According to Lemma \ref{thm:mapping}, $z_{k_1}$ maps the infinite strip $\R + i(-T_k,T_k)$ surjectively to the disk $\disc$ and satisfies $z_{k_1}(v+1) = z_{k_1}(v)= z_{k_1}(1/2-v)$.
Any function $f\in\Dir$ is thus mapped to an analytic function $g = f \circ z_{k_1}$ on the strip $\R + i(-T_k,T_k)$ satisfying $g(0)=0$ and $g(v+1) = g(v) = g(1/2-v)$.
Moreover, using the conformal invariance of the Dirichlet inner product, we have that
\begin{equation*}
\|f\|^2_\Dir = \int_{\disc} |f'(z)|^2 \rmd A(z) = \int_{\disc\setminus\{z\in\R:z^2 > k_1\}} |f'(z)|^2 \rmd A(z) = \int_{H_{k_1}} |g'(v)|^2 \rmd A(v) = \| g\|^2_{\mathcal{H}_{k_1}}.
\end{equation*}
Hence $g = f\circ z_{k_1} \in \mathcal{H}_{k_1}$, and more generally $\langle f_1\circ z_{k_1}, f_2 \circ z_{k_1}\rangle_{\mathcal{H}_{k_1}} = \DirProd{f_1}{f_2}$ for $f_1,f_2\in\Dir$.

Finally we need to show that $f\mapsto f\circ z_{k_1}$ is surjective.
Observe that $g \in \mathcal{H}_{k_1}$ is determined (via its periodicities) by its values in $H_{k_1}$.
Therefore $g \circ v_{k_1}$ determines an analytic function in $\disc \setminus \{z\in\R:z^2 > k_1\}$.
Pairs of points just above and below the slits correspond (under the mapping $v_{k_1}$) to complex conjugate pairs of points of real part $\pm 1/4$ (see Lemma \ref{thm:mapping}(iii)). 
Since $g( 1/4 + i y) = g(1/4 - i y)$ and $g( -1/4 + i y) = g(-1/4 - i y)$ for $y\in (0,T_k)$, the function $g \circ v_{k_1}$ is continuous across the slits and therefore extends to an analytic function $f\in \Dir$, which clearly satisfies $f \circ z_{k_1} = g$.
\end{proof}	

Since $\mathcal{H}_{k_1}$ deals with periodic functions, it is natural to seek an orthogonal basis of $\mathcal{H}_{k_1}$ consisting of trigonometric functions.

\begin{lemma}\label{thm:fourierbasis}
	The sequence $(g_m)_{m=1}^\infty$ defined by $g_m(v) = \cos(2\pi m(v+1/4))-\cos(\pi m/2)$ is an orthogonal basis of $\mathcal{H}_{k_1}$ and 
	\begin{equation}
	\| g_m \|_{\mathcal{H}_{k_1}}^2 = \frac{m}{4} (q_k^{-m}-q_k^m).
	\end{equation}
\end{lemma}
\begin{proof}
	It should be clear from the definition that $g_m \in \mathcal{H}_{k_1}$ for any $m\geq 1$. 
	Computing the inner product of $g_m$ and $g_n$ with the help of Stokes' theorem, we find
	\begin{align*}
	\langle g_m, g_n \rangle_{\mathcal{H}_{k_1}} & = \int_{H_{k_1}} \overline{g_m'(v)}\,g_n'(v)\rmd A(v) = \frac{1}{2\pi i}\int_{\partial H_{k_1}} \overline{g_m(v)}\,g_n'(v)\rmd z
	\end{align*}
	where we used that $\rmd A(v)=\rmd A(x+iy)$ is the measure associated to the $2$-form $\frac{1}{\pi}\rmd x \wedge \rmd y = \frac{1}{2\pi i} \rmd \bar{v} \wedge\rmd v$ and the second integral traces the boundary $\partial H_{k_1}$ in counterclockwise direction.
	With the help of the periodicity of the functions $g_m$ and $g_n$ this reduces to
	\begin{align*}
	\langle g_m, g_n \rangle_{\mathcal{H}_{k_1}} &= \frac{1}{2\pi i} \int_{-1/4}^{1/4} (g_m(x+i T_k)g_n'(x-iT_k) - g_m(x-i T_k)g_n'(x+iT_k)) \rmd x\\
	&= \frac{1}{2\pi i} \int_{-1/2}^{1/2} g_m(x+i T_k)g_n'(x-iT_k)  \rmd x\\
	&= n\,i \int_{-1/2}^{1/2} \cos(2\pi m(x+i T_k +\tfrac{1}{4}))\sin(2\pi n(x-i T_k +\tfrac{1}{4})).
	\end{align*}
	This integral vanishes for $m\neq n$, while for	$m=n$ the integral is straightforwardly evaluated to
	\begin{equation}
	\langle g_m,g_m \rangle_{\mathcal{H}_{k_1}} = \frac{m}{2} \sinh( 4 \pi m T_k) = \frac{m}{4} (q_k^{-m}-q_k^m).
	\end{equation}
	It remains to show that the linear span of $(g_m)_{m=1}^\infty$ is dense in $\mathcal{H}_{k_1}$.
	To this end observe that any analytic function $g\in \mathcal{H}_{k_1}$ admits a Fourier series representation
	\begin{equation}
	g(v) = \sum_{m=-\infty}^\infty a_m e^{2\pi m i (v+1/4)},
	\end{equation}
	that is absolutely convergent on the strip $\R + i (-T_k,T_k)$ (and the same is true for its $v$-derivative).
	The identities $g(v)=g(1/2-v)$ and $g(0)=0$ imply that $a_{-m} = a_m$ and $a_0 + 2\sum_{m=1}^\infty a_m \cos(\pi m/2) = 0$.
	Hence we have the absolutely convergent sum
	\begin{equation*}
	g(v) = a_0 + 2\sum_{m=1}^\infty a_m \cos(2\pi m(v+1/4)) = \sum_{m=1}^\infty 2 a_m\, g_m(v).
	\end{equation*}
	From here one may check that $\langle g_m, g \rangle_{\mathcal{H}_{k_1}} = 2 a_m \|g_m\|^2_{\mathcal{H}_{k_1}}$ and therefore if $g\neq 0$ then $\langle g_m, g \rangle_{\mathcal{H}_{k_1}} \neq 0$ for some $m\geq 1$, implying $(g_m)_{m=1}^\infty$ is dense.
\end{proof}

Transferring the basis $(g_m)_{m=1}^\infty$ of $\mathcal{H}_{k_1}$ to $\mathcal{D}$ using the isomorphism of Lemma \ref{thm:isomorphism} gives precisely the basis $(f_m)_{m=1}^\infty$ announced in \eqref{eq:fmbasis} in the introduction.
In the following it will be useful to know that the functions $f_m$ can be analytically continued beyond the unit disc $\disc$:

\begin{lemma}\label{thm:fmradius}
For each $m\geq 1$, $f_m(z) = g_m \circ v_{k_1}(z)$ has radius of convergence equal to $1/\sqrt{k_1}$.
\end{lemma}
\begin{proof}
	From the initial remarks in the proof of Lemma \ref{thm:mapping} it follows that $z_{k_1}$ maps $(-1/4,1/4) + i (-2T_k,2 T_k)$ biholomorphically to the double-slit plane $\C \setminus \{ z\in \R : z^2 \geq \sqrt{k_1}\}$.
	Since $g_m$ is analytic on $\C$, the function $f_m$ can be analytically continued to the double-slit plane.
	It was also shown in Lemma \ref{thm:mapping} that the boundary segments $1/4 \pm i[0,2T_k]$ of the rectangle are both mapped by $z_{k_1}$ to the slit $[\sqrt{k_1},1/\sqrt{k_1}]$.
	Since $g_m(1/4+ i y)=g_m(1/4 - i y)$ for $y\in\R$, we find that $f_m$ is continuous across the slit $[\sqrt{k_1},1/\sqrt{k_1}]$, and similar arguments apply to the other slit $[-1/\sqrt{k_1},-\sqrt{k_1}]$.
	Hence, $f_m$ can be analytically continued to $\C \setminus \{ z\in \R : z^2 \geq 1/\sqrt{k_1}\}$.
	To finish the proof we observe that $f_m(z)$ necessarily has a branch cut at $z=1/\sqrt{k_1}$ because $z_{k_1}(v+2iT_k)=1/z_{k_1}(v)$ and $z'_{k_1}(1/4) = 0$ imply that $z_{k_1}'(1/4 + 2i T_k) = 0$ while $g_m'(1/4+2i T_k) \neq 0$. 
\end{proof}

We are now ready to perform the diagonalization of $\mathbf{J}_k$.

\begin{proposition}\label{thm:Jdiagonalization}
The family $(f_m)_{m=1}^\infty$ forms an orthogonal basis of $\Dir$ consisting of eigenvectors of $\mathbf{J}_k$ satisfying
\[
\mathbf{J}_k f_m = \frac{1}{q_k^{m/2} + q_k^{-m/2}} \,f_m, \qquad \|f_m\|^2_\Dir = \frac{m}{4}\left(q_k^{-m}-q_k^{m}\right).
\]
\end{proposition}
\begin{proof}
	The functions $(g_m)_{m=1}^\infty$ of Lemma \ref{thm:fourierbasis} satisfy $g_m = f_m \circ z_{k_1}$.
	Therefore Lemma \ref{thm:isomorphism} and Lemma \ref{thm:fourierbasis} together imply that $(f_m)_{m=1}^\infty$ is an orthogonal basis of $\Dir$ and that 
	\begin{equation*}
	\langle f_m,f_n\rangle_{\Dir} = \langle g_m,g_n\rangle_{\mathcal{H}_{k_1}} = \frac{1}{2}m \sinh( 2 m \pi T_{k_1} )\one_{\{m=n\}} = \frac{m}{4}\left(q_k^{-m}-q_k^{m}\right)\one_{\{m=n\}}.
	\end{equation*}
By Lemma \ref{thm:mapping}, $\langle f_m,f_n\rangle_{\Dir(\psi_k(\disc))} = \langle g_m,g_n\rangle_{\mathcal{H}_{k}}$, so by replacing $k_1 \to k$ in the previous expression we find
\[
\langle f_m,\mathbf{J}_kf_n\rangle_{\Dir} = \langle f_m,f_n\rangle_{\Dir(\psi_k(\disc))} = \frac{1}{2}m \sinh( 2 m \pi T_{k} )\one_{\{m=n\}}. 
\]
Hence $f_m$ is an eigenvector of $\mathbf{J}_k$ with eigenvalue given by
\[
\frac{\langle f_m,\mathbf{J}_kf_m\rangle_{\Dir}}{\langle f_m,f_m\rangle_{\Dir}}=\frac{\sinh(2m\pi T_{k})}{\sinh(4m\pi T_k)} = \frac{1}{2 \cosh( 2m\pi T_k)} = \frac{1}{q_k^{m/2} + q_k^{-m/2}}, 
\]
where we made use of \eqref{eq:landenk} and \eqref{eq:nomedef}.
\end{proof}

Notice that this verifies Theorem \ref{thm:mainresult}(ii) for $\alpha=\pi/2$ and $I=(-\pi/2,\pi/2)$.

\subsection{Analytic structure of the eigenvectors}\label{sec:analytic}

\begin{proposition}\label{thm:fhatanalytic}
For any $p,m\geq 1$, the coefficient $[z^p]\hat{f}_m(z) = \frac{1}{p}\DirProd{e_p}{\hat{f}_m}$ of the normalized eigenvector $\hat{f}_m = f_m / \|f_m\|_{\Dir}$ is analytic in $k$ around $0$.
\end{proposition}
\begin{proof}
From the definition \eqref{eq:vkmap} of $v_k$ and the fact that $K(k) = \frac{\pi}{2} + \mathcal{O}(k^2)$ is analytic in $k$ at $0$, it follows that $[z^p]v_k(\sqrt{k}\,z)$ is analytic in $k$ at $k=0$ for any $p\geq 1$. 
Since $k_1 = \frac{1}{4}k^2 + \mathcal{O}(k^4)$, the same is true for $[z^p]v_{k_1}(\sqrt{k_1}\,z)$ as well as $[z^p]f_m(\sqrt{k_1}z) = [z^p]g_m(v_{k_1}(\sqrt{k_1}\,z))$ for any $p,m\geq 1$.
Since $\sqrt{k_1}$ is analytic in $k$ around $k=0$, $[z^p]f_m(z)$ has a Laurent series expansion around $k=0$ with a pole of order at most $p$ at $k=0$. 
The inverse norm $1/\|f_m\|_{\Dir} = 1/\sqrt{\frac{m}{4}(q_k^{-m}-q_k^m)}$ as computed in Proposition \ref{thm:Jdiagonalization} is analytic in $k$ for any $m$ (see \eqref{eq:nomeexpansion}), implying that $[z^p]\hat{f}_m(z)$ has a Laurent series expansion around $k=0$ too. 
Recalling from Lemma \ref{thm:fmradius} that $f_m(z)$ has radius of convergence larger than $1$, we deduce from Cauchy's inequality and Lemma \ref{thm:fhatbound} below that
\begin{equation*}
\left| [z^p] \hat{f}_m(z)\right| \leq \sup_{|z|=1} |\hat{f}_m(z)| \leq 2\qquad \text{for any }p,m\geq 1\text{ and }k\in(0,1/2).
\end{equation*}
We conclude that $[z^p]\hat{f}_m(z)$ has no pole at $k=0$, thus finishing the proof.
\end{proof}

\begin{lemma}\label{thm:fhatbound}
We have the bound $|\hat{f}_m(z)| < 2$ for all $k\in(0,1/2)$, $|z|\leq 1$ and $m\geq 1$.
\end{lemma}
\begin{proof}
The function $g_m(v) = \cos(2\pi m(v+1/4))-\cos(\pi m/2)$ of Lemma \ref{thm:fourierbasis} evaluated at $v= x+i y$ satisfies
\begin{align*}
|g_m(v)|^2 &= \left(\cos(\tfrac{1}{2}m \pi)-\cos(2m \pi(x+1/4))\cosh(2m\pi y)\right)^2 + \sin^2(2m \pi(x+1/4))\sinh^2(2m\pi y) \\
&= \cos^2(\tfrac{1}{2}m \pi) -2\cos(\tfrac{1}{2}m \pi)\cos(2m\pi(x+1/4)) \cosh(2m\pi y)\\
&\quad + \frac{1}{2} \cos(4m\pi(x+1/4)) + \frac{1}{2} \cosh(4m\pi y) \\
&\leq \frac{3}{2} + 2 \cosh(2m\pi y) + \frac{1}{2} \cosh(4m\pi y) = 4 \cosh^4(m\pi y).
\end{align*}
For $v\in H_{k_1} = (-1/4,1/4)+i(-T_k,T_k)$ we therefore have
\begin{equation}
|g_m(v)| \leq 2\cosh^2(\pi m T_k) = \frac{1}{2}(q_k^{m/4}+q_k^{-m/4})^2,
\end{equation} 
which translates into the bound
\begin{equation*}
|f_m(z)| \leq \frac{1}{2}(q_k^{m/4}+q_k^{-m/4})^2\qquad\text{for all }|z|\leq 1\text{, }m\geq 1\text{ and }k\in(0,1).
\end{equation*}
According to \eqref{eq:nomebound2}, we have $q_k \in (0,\tfrac{1}{20})$ when $k\in (0,1/2)$ and therefore $\frac{1}{2}(q_k^{m/4}+q_k^{-m/4})^2 \leq \sqrt{q_k^{-m}-q_k^m}$.
The claim then follows from $|\hat{f}_m(z)| = |f_m(z)| / \sqrt{\frac{m}{4}(q_k^{-m}-q_k^m)}$.
\end{proof}

It follows from the proof of Proposition \ref{thm:fhatanalytic} that one can compute the series expansion of $\DirProd{e_p}{\hat{f}_m} = p[z^p]\hat{f}_m(z)$ to arbitrary order in $k$, given the series expansions of $v_{k_1}(\sqrt{k_1} z)$ (which follows from the definition \eqref{eq:vkmap}), $K(k)$ and $q_k$ (which are described in Appendix \ref{sec:ellipticappendix}). 
The table below shows the first few terms of $\frac{1}{\sqrt{m}}\DirProd{e_p}{\hat{f}_m}$ in $t=k/4$ for $p\leq 3$ and $m\leq 6$.
\begin{equation*}
\begin{array}{c|cccc}
\tfrac{1}{\sqrt{m}}\DirProd{e_p}{\hat{f}_m} & p=1 & p=2 & p=3 \\\hline
m=1& -1+\frac{3 t^4}{2}+24 t^6+\frac{2765 t^8}{8}+\cdots & 0 & -3 t^2-24 t^4-\frac{477 t^6}{2}-2652 t^8+\cdots \\
m=2 & 0 &\mkern-55mu 1-4 t^4-64 t^6-\frac{1835 t^8}{2}+\cdots \mkern-55mu& 0 \\
m=3 & t^2+8 t^4+82 t^6+944 t^8+\cdots & 0 & -1+9 t^4+144 t^6+2052 t^8+\cdots \\
m=4 & 0 &\mkern-55mu -2 t^2-16 t^4-160 t^6-1792 t^8+\cdots \mkern-55mu& 0 \\
m=5 & -t^4-16 t^6-230 t^8+\cdots & 0 & 3 t^2+24 t^4+231 t^6+2472 t^8+\cdots \\
m=6 & 0 &\mkern-55mu 3 t^4+48 t^6+684 t^8+\cdots \mkern-55mu& 0 \\
\end{array}
\end{equation*}

\subsection{The operators $\mathbf{B}_k$ and $\mathbf{A}_k$} \label{sec:operatorBA}

\begin{figure}
		\centering
		\includegraphics[width=.35\linewidth]{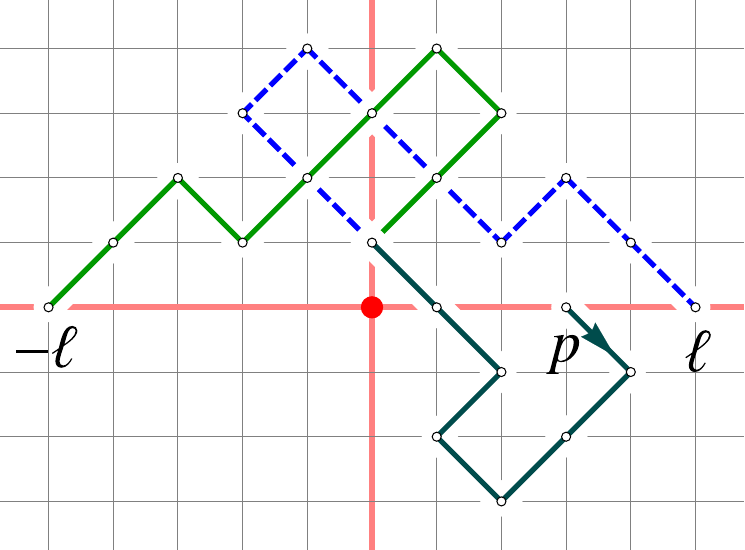}	
	\caption{The reflection principle in the vertical axis yields a bijection between walks from $(p,0)$ to $(-\ell,0)$ and walks $(p,0)$ to $(\ell,0)$ that visit the vertical axis at least once. }\label{fig:reflectionb}
\end{figure}

Recall that $B_{\ell,p}(t) = \DirProd{e_\ell}{\mathbf{B}_ke_p}$ is the generating function for the set $\mathcal{B}_{\ell,p}$ of simple diagonal walks from $(p,0)$ to $(\ell,0)$ that maintain strictly positive first coordinate.
A simple reflection principle (see Figure \ref{fig:reflectionb}) teaches us that $B_{\ell,p}(t)$ is given by
\begin{equation}\label{eq:Breflection}
B_{\ell,p}(t) = B_{\ell-p}(t) - B_{-\ell-p}(t),
\end{equation}
where $B_m(t)$, $m\in\Z$, is the generating function of simple diagonal walks from $(0,0)$ to $(m,0)$ without further restrictions.

\begin{lemma}\label{thm:bilateral}
	For fixed $t\in(0,1/4)$, $B_m(t)$ can be expressed as a contour integral as
	\begin{equation}\label{eq:bilateral}
	B_m(t) = \frac{1}{2\pi i} \int_{\gamma} \frac{\rmd z}{z^{m+1}\sqrt{1-4t^2(z+1/z)^2}}, 
	\end{equation}
	where $\gamma$ traces the unit circle in counterclockwise direction.
\end{lemma}
\begin{proof}
	The contribution of walks from $(0,0)$ to $(m,0)$ of length $2n$ and $m$ even is
	\[
	  \binom{2n}{n} \binom{2n}{n+m/2} t^{2n}= \frac{1}{2\pi i} \int_{\gamma} \frac{\rmd z}{z^{m+1}} \binom{2n}{n}(z+1/z)^{2n}t^{2n}.
	\]
	The result then follows from summing over $n\geq 0$ and relying on absolute convergence to interchange the summation and integral.
\end{proof}

Based on the similarity between the integrand of \eqref{eq:bilateral} and the one in the definition \eqref{eq:vkmap} of $v_{k_1}$, we find the following useful representation of $\mathbf{B}_k$.

\begin{lemma}\label{thm:Boperatorintegral}
	If $f,g\in \Dir$ are analytic in a neighbourhood of the closed unit disk, then
	\begin{equation}\label{eq:Boperatorexpr}
	\langle f,\mathbf{B}_k g\rangle_\Dir = \frac{2K(k)}{\pi} \int_{\gamma'} (f(z)-f(z^{-1}))(g(z)-g(z^{-1}))v_{k_1}'(z)\rmd z,
	\end{equation}
	where $\gamma'$ traces the upper half of the unit circle starting at $1$ and ending at $-1$.
\end{lemma}
\begin{proof}
From the definition \eqref{eq:vkmap} we see that the integrand of \eqref{eq:Boperatorexpr} is continuous and bounded on the upper-half circle, and therefore the right-hand side of \eqref{eq:Boperatorexpr} converges absolutely.
Hence, it suffices to check the identity for $f=e_\ell$ and $g=e_p$, $p,\ell\geq 1$.
Combining \eqref{eq:Breflection} and Lemma \ref{thm:bilateral} we find
\begin{align}
\langle e_\ell, \mathbf{B} e_p \rangle_\Dir &= B_{\ell-p}(t) - B_{-\ell-p}(t) = \frac{1}{2}\left( B_{\ell-p}(t) - B_{-\ell-p}(t) + B_{p-\ell}(t) - B_{\ell+p}(t)\right) \nonumber\\
&= \frac{-1}{4\pi i} \int_{\gamma} \frac{(z^\ell-z^{-\ell})(z^p-z^{-p})\rmd z}{z \sqrt{1-4t^2(z+1/z)^2}} = \frac{-1}{2\pi i} \int_{\gamma'} \frac{(z^\ell-z^{-\ell})(z^p-z^{-p})\rmd z}{z \sqrt{1-4t^2(z+1/z)^2}},\label{eq:Bintegral}
\end{align}
where in the last equality we used that both sides vanish for $p+\ell$ odd, while for $p+\ell$ even the upper and lower half circles contribute equally.
Note that
\[
(k_1-z^2)(1-k_1 z^2) = - z^2( (1+k_1)^2 - k_1 (z+1/z)^2 ) = -z^2(1+k_1)^2 \left( 1-\frac{k^2}{4}(z+1/z)^2\right).
\]
Hence, for $z$ on the upper-half circle \eqref{eq:vkmap} implies that
\[
v_{k_1}'(z) = \frac{1}{4K(k_1)} \frac{1}{\sqrt{(k_1-z^2)(1-k_1 z^2)}} = \frac{i}{4K(k)}\frac{1}{z\sqrt{1-\frac{k^2}{4}(z+1/z)^2}},
\]
where we used \eqref{eq:landenk} and the sign on the right-hand is determined by the fact that $v_{k_1}'(z)$ has positive imaginary part when $z$ is on the upper-half circle close to $1$. 
Combining with \eqref{eq:Bintegral} we indeed reproduce the right-hand side of \eqref{eq:Boperatorexpr}.
\end{proof}

With the contour integral representation in hand it is now straightforward to evaluate $\mathbf{B}_k$ (and $\mathbf{A}_k$ subsequently) with respect to the basis $(f_m)_{m=1}^\infty$.

\begin{proposition}\label{thm:Bdiagonalization}
The linear operators $\mathbf{B}_k$ and $\mathbf{A}_k$ are compact and have the same eigenvectors $(f_m)_{m=1}^\infty$ as $\mathbf{J}_k$ satisfying
\begin{align*}
\mathbf{B}_k f_m &= \frac{2K(k)}{\pi} \frac{1}{m} \frac{1-q_k^m}{1+q_k^m}\, f_m, \\
\mathbf{A}_k f_m &= \frac{\pi}{2K(k)} \frac{m}{q_k^{-m/2}-q_k^{m/2}} \,f_m.
\end{align*}
\end{proposition}
\begin{proof}
According to Lemma \ref{thm:fmradius} the functions $f_m(z)$ are analytic for $|z| < 1/\sqrt{k_1}$, and therefore we may apply Lemma \ref{thm:Boperatorintegral} to obtain
\begin{align*}
\langle f_n, \mathbf{B}_k f_m \rangle_{\Dir} &= \frac{2K(k)}{\pi} \int_{\gamma'} (f_n(z)-f_n(z^{-1}))(f_m(z)-f_m(z^{-1}))v_{k_1}'(z)\rmd z \\
& = -\frac{2K(k)}{\pi} \int_{-1/4}^{1/4} (g_n(v+iT_k) - g_n( v-iT_k))(g_m(v+iT_k) - g_m( v-iT_k))\rmd v,
\end{align*}
where we changed integration variables $z = z_{k_1}(v)$, using that $z_{k_1}(v-iT_k)=1/z_{k_1}(v+iT_k)$ and that $z_{k_1}$ maps $(-1/4,1/4) + i T_k$ to $\gamma'$ but with opposite orientation. 
Substituting the expression for $g_m$ yields
\begin{align*}
\langle f_n, \mathbf{B}_k f_m \rangle_{\Dir}& = -\frac{2K(k)}{\pi} \int_{0}^{1/2} (\cos(2\pi n (v+iT_k)) - \cos(2\pi n( v-iT_k)))\\
& \qquad\qquad\qquad\qquad \times(\cos(2\pi m (v+iT_k)) - \cos(2\pi m( v-iT_k)))\rmd v\\
&= \frac{8K(k)}{\pi} \sinh( 2 \pi n T_k)\sinh(2\pi m T_k) \int_{0}^{1/2} \sin(2\pi n v)\sin(2\pi m v) \rmd v\\
&= \frac{2K(k)}{\pi} \sinh^2(2\pi m T_k) \one_{\{m=n\}}.
\end{align*}
Together with \eqref{eq:fnorm} we conclude that $f_m$ is an eigenvector of $\mathbf{B}_k$ (which we already knew from Lemma \ref{thm:simdiag} and Proposition \ref{thm:Jdiagonalization}) with eigenvalue
\[
\frac{\langle f_m,\mathbf{B}_k f_m\rangle_\Dir}{\langle f_m,f_m\rangle_\Dir} = 
\frac{4K(k)}{\pi} \frac{\sinh^2(2m\pi T_k)}{m\sinh(4m\pi T_k)} = \frac{2K(k)}{\pi} \frac{\tanh(2m\pi T_k)}{m} = \frac{2K(k)}{\pi} \frac{1}{m} \frac{1-q_k^m}{1+q_k^m}.
\]
Since $\mathbf{J}_k = \mathbf{A}_k\mathbf{B}_k$, we find using Proposition \ref{thm:Jdiagonalization} that
\[
\mathbf{A}_k f_m = \frac{\pi}{2K(k)} m \frac{1+q_k^m}{1-q_k^m} \frac{1}{q_k^{m/2}+q_k^{-m/2}} f_m = \frac{\pi}{2K(k)} \frac{m}{q_k^{-m/2}-q_k^{m/2}} f_m.
\]
The eigenvalues of $\mathbf{B}_k$ and $\mathbf{A}_k$ both approach $0$ as $m\to\infty$, implying compactness (see e.g. \cite[Proposition 1.19]{zhu_operator_2007}). 
Note that we already established compactness of $\mathbf{A}_k$ in Proposition \ref{thm:JAB}.
\end{proof}

\subsection{Proof of Theorem \ref{thm:mainresult}}\label{sec:proofmainresult}

We start with part (i) with $\alpha\in\frac{\pi}{2}\Z$.
Given a walk $w\in\mathcal{W}_{\ell,p}^{(\alpha)}$, let $(s_i)_{i=1}^N$ be the sequence of times at which $w$ \emph{alternates between the axes}, i.e. $s_0 = 0$ and for each $j\geq 0$ we set $s_{j+1} = \inf\{s>s_j : |\theta^w_s - \theta^w_{s_j}| = \pi/2\}$ provided it exists (otherwise $s_j=s_N$ is the last entry in the sequence).
Let $(\alpha_j)_{j=0}^N$ and $(\ell_j)_{j=0}^{N}$ be the sequences of winding angles and distances to the origin defined by $\alpha_j = \theta_{s_j}^w$ respectively $\ell_j = |w_{s_j}|$ for $0\leq j\leq N$. 
It is now easy to see that for $0\leq j< N$ the part of the walk between time $s_{j}$ and $s_{j+1}$ is (up to a unique rotation around the origin and/or reflection in the horizontal axis) of the form of a walk $w^{(j)}\in \mathcal{J}_{\ell_{j+1},\ell_{j}}$.
Similarly, the last part of the walk between time $s_{N}$ and $|w|$ is (up to rotation) of the form of a walk $w^{(N)}\in\mathcal{B}_{\ell,\ell_N}$. See Figure \ref{fig:walkdecomposition} for an example.
\begin{figure}
	\begin{subfigure}[b]{0.275\textwidth}
		\centering
		\includegraphics[width=.85\linewidth]{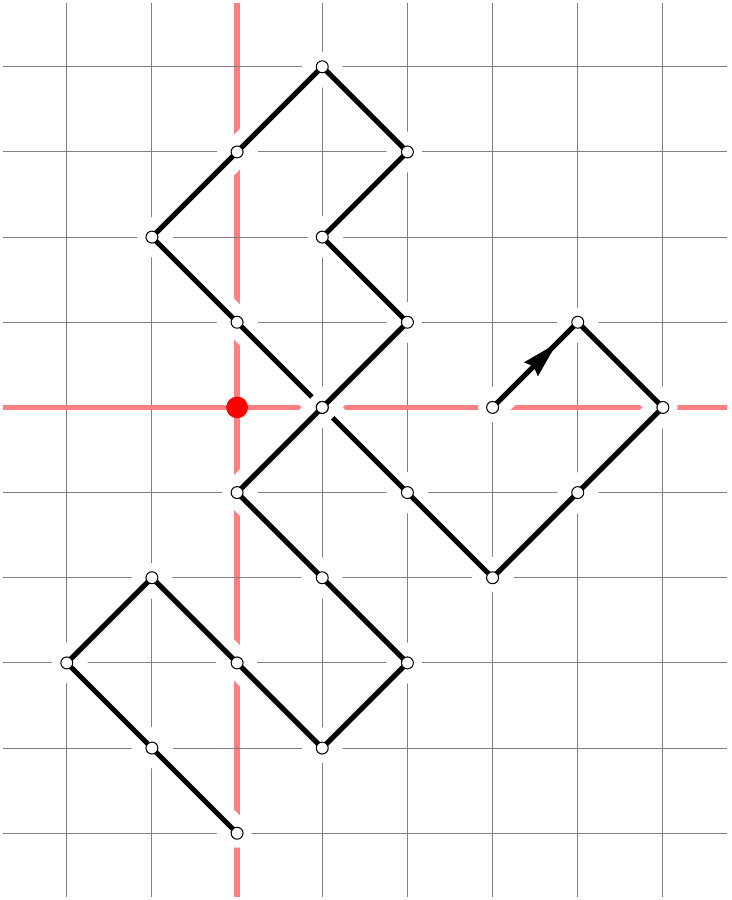}
		\caption{}\label{fig:decomposition1}
	\end{subfigure}%
	\begin{subfigure}[b]{0.275\textwidth}
		\centering
		\includegraphics[width=.85\linewidth]{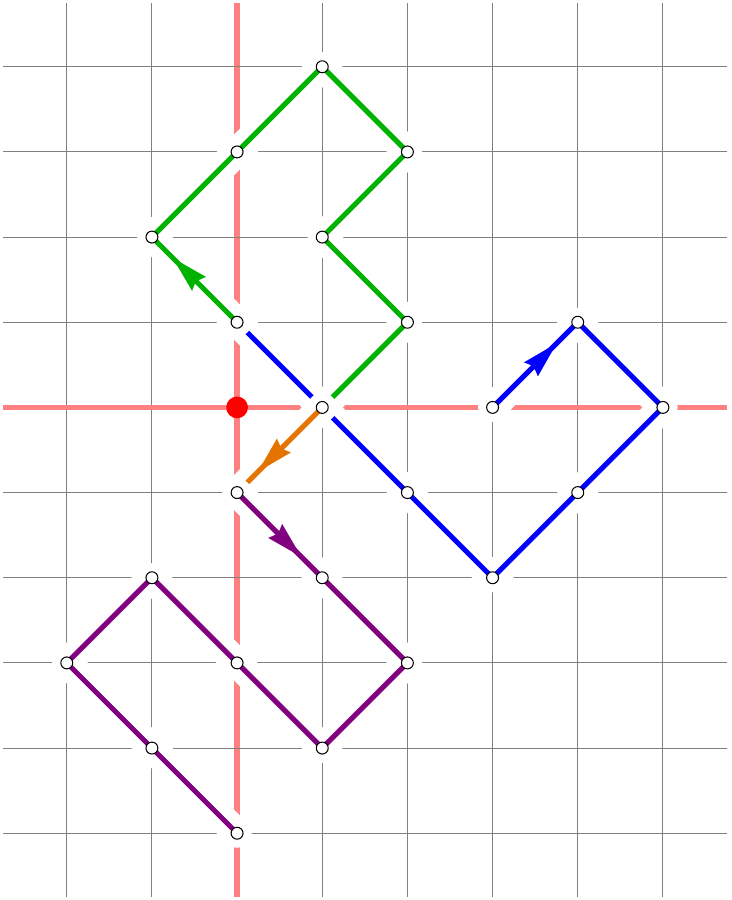}
		\caption{}\label{fig:decomposition2}
	\end{subfigure}%
	\begin{subfigure}[b]{0.45\textwidth}
		\centering
		\includegraphics[width=.99\linewidth]{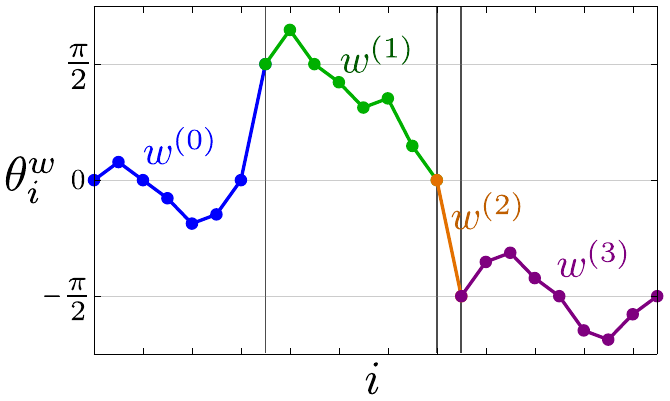}
		\caption{}\label{fig:decompositiongraph}
	\end{subfigure}%
	\caption{(a) A walk $w\in\mathcal{W}_{5,3}^{(-\pi/2)}$ with $N=3$ axis alternations; (b) its decomposition into $w^{(0)} \in \mathcal{J}_{1,3}$, $w^{(1)},w^{(2)}\in\mathcal{J}_{1,1}$, $w^{(3)}\in \mathcal{B}_{5,1}$; (c) its winding angle sequence $(\theta_i^w)$.}\label{fig:walkdecomposition}
\end{figure}

In fact, this construction is seen to yield a bijection between $\mathcal{W}_{\ell,p}^{(\alpha)}$ and the set of tuples 
\[
\left(N, (\ell_j)_{j=0}^{N}, (\alpha_j)_{j=0}^{N}, (w^{(j)})_{j=0}^N\right)
\]
where $N\geq 0$, $\ell_0=p$, $\ell_j\geq 1$, $(\alpha_j)_{j=0}^N$ is a simple walk on $\frac{\pi}{2}\Z$ from $\alpha_0=0$ to $\alpha_N=\alpha$, $w^{(j)}\in \mathcal{J}_{\ell_{j+1},\ell_{j}}$ for $0\leq j <N$ and $w^{(N)}\in\mathcal{B}_{\ell,\ell_N}$.
If we denote by 
\[
a^{(\alpha)}_N = \binom{N}{\frac{N-\frac{2\alpha}{\pi}}{2}}\one_{\{N-\frac{2\alpha}{\pi}\text{ even and non-negative}\}}
\] 
the number of simple walks on $\frac{\pi}{2}\Z$ from $0$ to $\alpha$ of length $N\geq 0$, then we may identify the generating function of $\mathcal{W}_{\ell,p}^{(\alpha)}$ as
\begin{align*}
W_{\ell,p}^{(\alpha)}(t) &= \sum_{N=0}^\infty a^{(\alpha)}_N \sum_{\ell_1,\ldots,\ell_N=1}^\infty B_{\ell,\ell_N}(t) \prod_{i=0}^{N-1}J_{\ell_{i+1},\ell_i}(t) \\
&=\sum_{N=0}^\infty a^{(\alpha)}_N \sum_{\ell_1,\ldots,\ell_N=1}^\infty \DirProd{e_\ell}{\mathbf{B}_ke_{\ell_N}} \prod_{i=0}^{N-1}\frac{1}{\ell_{i+1}}\DirProd{e_{\ell_{i+1}}}{\mathbf{J}_ke_{\ell_i}}\\
&=\sum_{N=0}^\infty a^{(\alpha)}_N \DirProd{e_\ell}{\mathbf{B}_k\,\mathbf{J}_k^N\,e_p}.\label{eq:BJcomposition}
\end{align*}
Since the eigenvalues of $\mathbf{J}_k$ are all strictly smaller than $1/2$ and $a^{(\alpha)}_N \leq 2^N$, the operator $\sum_{N=0}^M a^{(\alpha)}_N \mathbf{B}_k\,\mathbf{J}_k^N$ is compact for any $M\geq 0$ and converges as $M\to\infty$ (in the operator norm topology) to a compact self-adjoint operator $\mathbf{Y}^{(\alpha)}_k$ satisfying
\begin{equation}
W_{\ell,p}^{(\alpha)}(t) = \DirProd{e_\ell}{\mathbf{Y}_k^{(\alpha)}\,e_p}.
\end{equation}
With a little help of \eqref{eq:lagrangeinversion}, we find the (formal) generating function
\[
a^{(\alpha)}(x) = \sum_{N=0}^\infty a^{(\alpha)}_N x^N = \frac{1}{\sqrt{1-4 x^2}}\left(\frac{1-\sqrt{1-4x^2}}{2x}\right)^{2|\alpha|/\pi}. 
\]
Then one may deduce after some simplification that
\begin{align*}
\mathbf{Y}^{(\alpha)}_k f_m = \mathbf{B}_k\,a^{(\alpha)}\left(\mathbf{J}_k\right) f_m &= \frac{2K(k)}{\pi} \frac{1}{m} \frac{1-q_k^m}{1+q_k^m} a^{(\alpha)}\left( \frac{1}{q_k^{m/2}+q_k^{-m/2}}\right) f_m\\
&= \frac{2K(k)}{\pi} \frac{1}{m}\,q_k^{|\alpha| m/\pi} f_m,
\end{align*} 
in agreement with part (i) of Theorem \ref{thm:mainresult}.

We could easily extend this result to the case $I=(\beta_-,\beta_+)$ with $\beta_\pm \in \frac{\pi}{2}\Z\cup\{\pm\infty\}$ such that $0\in I$ and $\alpha\in I\cap\frac{\pi}{2}\Z$, by replacing $a^{(\alpha)}(x)$ by the generating function of simple walks on $\frac{\pi}{2}\Z$ confined to an interval.
Instead, we choose to discuss a reflection principle at the level of the simple diagonal walks, which allows us to directly generalize to the case of $\beta_\pm \in \frac{\pi}{4}\Z\cup\{\pm\infty\}$ of part (iii).
\begin{lemma}\label{thm:reflection}
Suppose $\ell$ and $p$ are both odd and $\beta_\pm \in \frac{\pi}{2}\Z$, or $\ell$ and $p$ are both even and $\beta_\pm \in \frac{\pi}{4}\Z$, such that $0\in I$. 
If in addition $\alpha \in \frac{\pi}{2}\Z \cap I$, then the generating function for $\mathcal{W}_{\ell,p}^{(\alpha,I)}$ is given by
\[
W_{\ell,p}^{(\alpha,I)}(t) = \sum_{n=-\infty}^\infty \left(W_{\ell,p}^{(\alpha+n\delta)}(t) - W_{\ell,p}^{(2\beta_+-\alpha+n\delta)}(t)\right),  \qquad \delta \coloneqq 2(\beta_+-\beta_-).
\]
\end{lemma}
\begin{proof}
Consider any walk $w\in\bigcup_{n=-\infty}^\infty \mathcal{W}_{\ell,p}^{(2\beta_+-\alpha+n\delta)}$ and let $s = \inf\{ j\geq 0 : \theta^w_j \notin I\}$ be the first time $w$ leaves $I$, which is well-defined since $\theta^w \notin I$.
It is not hard to see that under the stated conditions on $\ell,p,\beta_\pm$ the winding angle sequence $(\theta_i^w)_{i=0}^{|w|}$ cannot cross $\beta_\pm$ without visiting $\beta_\pm$, and therefore $\theta_s^w = \beta_\pm$ and $w_s$ lies on an axis or a diagonal of $\Z^2$.
Then we let $w'$ be the walk obtained from $w$ by reflecting the portion of $w$ after time $s$ in this axis or diagonal (see Figure \ref{fig:reflectionexcursions}).
\begin{figure}
		\begin{subfigure}[b]{0.5\textwidth}
			\centering
			\includegraphics[height=.5\linewidth]{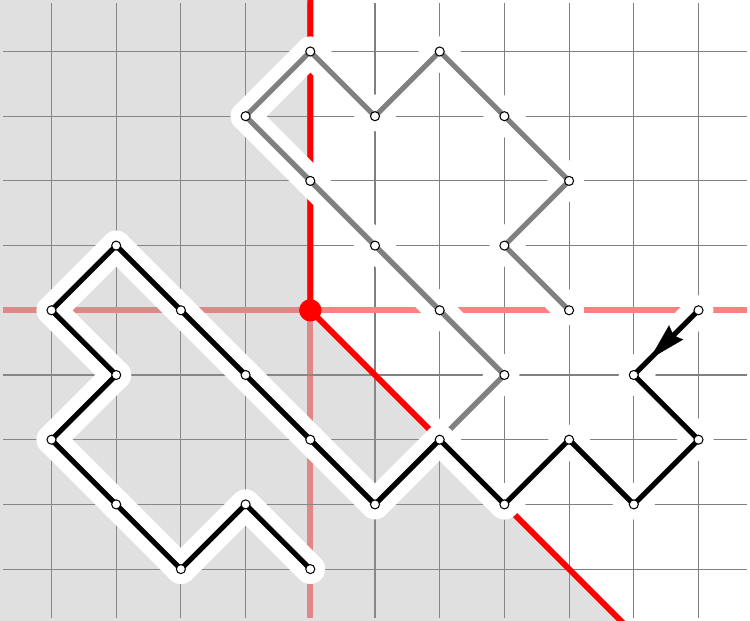}
			\caption{}\label{fig:reflectionexcursion}
		\end{subfigure}%
		\begin{subfigure}[b]{0.5\textwidth}
			\centering
			\includegraphics[height=.53\linewidth]{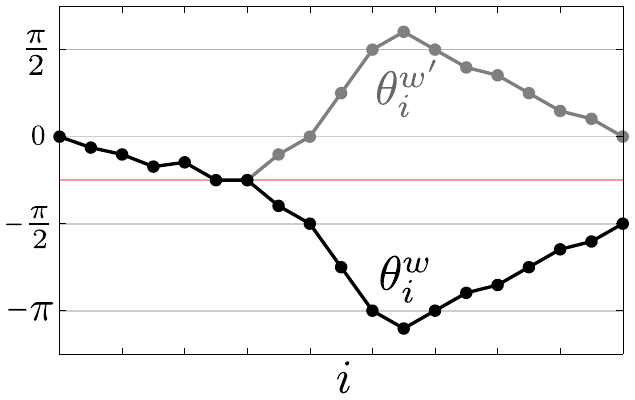}
			\caption{}\label{fig:reflectiongraph}
		\end{subfigure}%
	\caption{Example with $\alpha=0$, $I =(-\pi/4,\pi/2)$. (a) A walk $w\in\mathcal{W}_{4,6}^{(-\pi/2)}$ is depicted in black, and its reflection $w'\in \mathcal{W}_{4,6}^{(0)}$ in gray. (b) The corresponding winding angle sequences $(\theta^{w}_i)$ and $(\theta^{w'}_i)$. }\label{fig:reflectionexcursions}
\end{figure}
Then $\theta^{w'} = 2\beta_+ - \theta^w$ or $\theta^{w'} = 2\beta_- - \theta^w$.
Hence $w' \in \bigcup_{n=-\infty}^\infty \mathcal{W}_{\ell,p}^{(\alpha+n\delta)}$.
It is not hard to see that this mapping $w\mapsto w'$ is injective (the inverse $w'\mapsto w$ is given by the exact same reflection operation).
Moreover, any walk $w' \in \bigcup_{n=-\infty}^\infty \mathcal{W}_{\ell,p}^{(\alpha+n\delta)}$ is obtained in such way provided $(\theta_i^{w'})_{i=0}^{|w'|}$ visits $\beta_\pm$ at least once.
Clearly, the only walks $w'$ not satisfying the latter condition are the ones in $\mathcal{W}_{\ell,p}^{(\alpha,I)}$.
The claimed result for the generating function $W_{\ell,p}^{(\alpha,I)}(t)$ readily follows (absolute convergence is granted because $\sum_{\alpha'\in\frac{\pi}{2}\Z} W_{\ell,p}^{(\alpha')}(t) < \infty$).
\end{proof}

Inspired by this result let us introduce for $\beta_\pm\in \frac{\pi}{4}\Z$ such that $0\in I$ and $\alpha\in\frac{\pi}{2}\Z\cap I$, the operator $\mathbf{B}_k^{(\alpha,\beta_-,\beta_+)}$ on $\Dir$ defined by
\begin{equation}\label{eq:BfromY}
\mathbf{B}_k^{(\alpha,\beta_-,\beta_+)} \coloneqq \sum_{n=-\infty}^\infty \left(\mathbf{Y}_k^{(\alpha+n \delta)}-\mathbf{Y}_k^{(2\beta_+-\alpha+n \delta)}\right),  \qquad \delta \coloneqq 2(\beta_+-\beta_-),
\end{equation}
which by construction satisfies $\mathbf{B}_k^{(\alpha,\beta_-,\beta_+)} = \mathbf{B}_k^{(-\alpha,-\beta_+,-\beta_-)}$.
By Theorem \ref{thm:mainresult}(i) it is well-defined, compact and self-adjoint and has eigenvalues
\begin{equation}\label{eq:Bspectrumgeneral}
\frac{2K(k)}{\pi}\frac{1}{m} \sum_{n=-\infty}^\infty \left(q_k^{|\alpha+n \delta|m/\pi}-q_k^{|2\beta_+-\alpha+n \delta|m/\pi}\right) = \frac{2K(k)}{\pi}\, \frac{q_k^{2m\beta_-/\pi}-1}{m\,q_k^{m\alpha/\pi}}\, \frac{q_k^{2m\alpha/\pi}-q_k^{2m\beta_+/\pi}}{q_k^{2m\beta_-/\pi}-q_k^{2m\beta_+/\pi}}
\end{equation}
for $\alpha \geq 0$, while the eigenvalues for $\alpha < 0$ are obtained by the substitution $(\alpha,\beta_-,\beta_+) \to (-\alpha,-\beta_+,-\beta_-)$.
Lemma \ref{thm:reflection} then tells us that 
\begin{equation}\label{eq:genfunfromB}
W_{\ell,p}^{(\alpha,I)}(t) = \DirProd{e_\ell}{\mathbf{B}_k^{(\alpha,\beta_-,\beta_+)}e_p}
\end{equation}
holds under the conditions stated in the lemma, which exactly verifies part (ii) and (iii) for $0,\alpha\in I$ and $\beta_\pm$ finite.

Similarly when $\beta_-=-\infty$ or $\beta_+=\infty$ one may introduce the operators $\mathbf{B}_k^{(\alpha,-\infty,\beta_+)} = \mathbf{Y}_k^{(\alpha)} - \mathbf{Y}_k^{(2\beta_+-\alpha)}$ and $\mathbf{B}_k^{(\alpha,\beta_-,\infty)} = \mathbf{Y}_k^{(\alpha)} - \mathbf{Y}_k^{(2\beta_--\alpha)}$.
It is straightforward to check that then \eqref{eq:genfunfromB} still holds, and that the eigenvalues are given by \eqref{eq:Bspectrumgeneral} in the appropriate limit $\beta_-\to-\infty$ or $\beta_+ \to\infty$.

Next, let us consider the case $0<\alpha \in\frac{\pi}{2}\Z$ and $I = (0,\alpha)$. 
The case $\alpha=\pi/2$ with the corresponding operator $\mathbf{A}_k^{(\pi/2)} = \mathbf{A}_k$ has already been settled in Proposition \ref{thm:Bdiagonalization}, so let us assume $\alpha \geq \pi$.
Any such walk $w\in \mathcal{W}_{\ell,p}^{(\alpha,(0,\alpha))}$ is naturally encoded in a triple $w^{(1)},w^{(2)},w^{(3)}$ of walks with $w^{(1)} \in \mathcal{A}_{\ell_1,p}$, $w^{(2)} \in \mathcal{W}_{\ell_2,\ell_1}^{(\alpha-\pi,(-\pi/2,\alpha-\pi/2))}$ and $w^{(3)} \in \mathcal{A}_{\ell,\ell_2}$ for some $\ell_1,\ell_2\geq 1$.
Hence
\begin{align*}
W_{\ell,p}^{(\alpha,(0,\alpha))}(t) &= \sum_{\ell_1,\ell_2=1}^\infty \frac{1}{\ell\ell_2}\DirProd{e_\ell}{\mathbf{A}_k e_{\ell_2}} \DirProd{e_{\ell_2}}{\mathbf{B}_k^{(\alpha-\pi,-\pi/2,\alpha-\pi/2)} e_{\ell_1}}  \frac{1}{\ell_1p}\DirProd{e_{\ell_1}}{\mathbf{A}_k e_{p}}\\
&= \frac{1}{\ell p}\DirProd{e_\ell}{\mathbf{A}_k\mathbf{B}_k^{(\alpha-\pi,-\pi/2,\alpha-\pi/2)}\mathbf{A}_ke_p}.
\end{align*}
One may easily verify the claimed eigenvalues of $\mathbf{A}_k^{(\alpha)} = \mathbf{A}_k\mathbf{B}_k^{(\alpha-\pi,-\pi/2,\alpha-\pi/2)}\mathbf{A}_k$ and its compactness, thus settling Theorem \ref{thm:mainresult}(ii) for the operator $\mathbf{A}^{(\alpha)}$.

Finally, a similar argument shows that for $0<\alpha\in\frac{\pi}{2}\Z$, $0>\beta_-\in\frac{\pi}{4}\Z$, 
\[
W_{\ell,p}^{(\alpha,(\beta_-,\alpha))}(t) = \sum_{\ell_1=1}^\infty \frac{1}{\ell\ell_1}\DirProd{e_\ell}{\mathbf{A}_k e_{\ell_1}} \DirProd{e_{\ell_1}}{\mathbf{B}_k^{(\alpha-\pi/2,\beta_-,\alpha)} e_{p}} = \frac{1}{\ell} \DirProd{e_\ell}{\mathbf{A}_k\mathbf{B}_k^{(\alpha-\pi/2,\beta_-,\alpha)} e_{p}},
\]
and once again one may directly verify the eigenvalues of $\mathbf{J}_k^{(\alpha,\beta_-)}=\mathbf{A}_k\mathbf{B}_k^{(\alpha-\pi/2,\beta_-,\alpha)}$.
This verifies Theorem \ref{thm:mainresult}(ii) and (iii) for the operator $\mathbf{J}_k^{(\alpha,\beta_-)}$, thereby finishing the proof of Theorem \ref{thm:mainresult}.

\section{Excursions}

\subsection{Counting excursions with fixed winding angle}

Recall from the introduction the set of excursions $\mathcal{E}$ consisting of (non-empty) simple diagonal walks starting and ending at the origin with no intermediate returns.
For such an excursion $w\in\mathcal{E}$ we have a well-defined winding angle sequence $(\theta^w_i)_{i=0}^{|w|}$ with $\theta_1^w=\theta_0^w=0$ and $\theta^w = \theta_{|w|}^w = \theta_{|w|-1}^w$.
Our first goal is to compute the generating function of excursions with winding angle equal to $\alpha\in\frac{\pi}{2}\Z$,
\begin{equation}
F^{(\alpha)}(t) \coloneqq \sum_{w\in\mathcal{E}} t^{|w|} \one_{\{\theta^w = \alpha\}}.\label{eq:Ffixeddef}
\end{equation}
To this end we cannot directly apply Theorem \ref{thm:mainresult}(i) because the excursions start and end at the origin.
Nevertheless, a combinatorial trick allows us to relate $F^{(\alpha)}(t)$ to the generating functions $W_{\ell,p}^{(\alpha')}$ with $\alpha' > |\alpha|$. 
In order for this trick to work we first have to establish a bound on $W_{\ell,p}^{(\alpha)}$ as $\alpha$ gets large.

\begin{lemma}\label{thm:wsumbound}
	For $t\in(0,1/4)$, there exists a $C>0$ such that $\sum_{\ell,p\geq 1} W_{\ell,p}^{(\alpha)}(t) < C q_k^{\frac{|\alpha|}{\pi}-\frac{1}{4}}$ for all $\alpha\in\frac{\pi}{2}\Z\setminus\{0\}$.
\end{lemma}
\begin{proof}
	Let $r\in(1,1/\sqrt{k_1})$.  
	Since $\hat{f}_m$ has radius of convergence larger than $r$ (Lemma \ref{thm:fmradius}), Cauchy's inequality implies that $\left|[z^p] \hat{f}_m(z)\right| \leq r^{-p} \sup_{|z|=r} |\hat{f}_m(z)|$.
	Therefore
	\begin{equation}\label{eq:fcoeffabsbound}
	\sum_{p=1}^\infty \left| \DirProd{e_p}{\hat{f}_m}\right| = \sum_{p=1}^\infty p \left|[z^p]\hat{f}_m(z)\right| \leq \frac{r}{(r-1)^2} \sup_{|z|=r} |\hat{f}_m(z)| = \frac{r}{(r-1)^2} \frac{1}{\|f_m\|_\Dir} \sup_{|z|=r} |f_m(z)|.
	\end{equation}
	Given that $v_{k_1}$ maps the double-slit disc $\disc\setminus \{z\in\R : z^2 \geq k_1\}$ onto $H_{k_1}$ (Lemma \ref{thm:mapping}), we may choose $r\in(1,1/\sqrt{k_1})$ such that $v_{k_1}$ maps $\{z\in\C: |z| \leq r \}\setminus \{z\in\R : z^2 \geq k_1\}$ into the slightly larger rectangle $(-\frac{1}{4},\frac{1}{4}) + i(-\frac{5}{4}T_k,\frac{5}{4}T_k)$.
	The proof of Lemma \ref{thm:fhatbound} shows that within this triangle $|g_m(v)|$ satisfies
	\begin{equation}
	|g_m(v)| \leq 2 \cosh^2(\tfrac{5}{4} m \pi T_k) = \frac{1}{2} \left(q_k^{-\frac{5}{16}m} + q_k^{\frac{5}{16}m}\right)^2 \leq 2 q_k^{-\frac{5}{8}m}.
	\end{equation}
	The maximum of $|g_m(v)|$ on the latter is attained in the corners $\pm1/4\pm i \frac{5}{4}T_k$ of the rectangle (see also the proof of Lemma \ref{thm:fhatbound}), where it is bounded by
	\begin{equation*}
	\cosh( 2 \pi m \tfrac{5}{4}T_k) \leq q_k^{-5m/8}.
	\end{equation*}
	By \eqref{eq:fnorm}, there exists a $c'>0$ (depending on $t$) such that $\|f_m\|_\Dir \geq c' q_k^{-m/2}$ for all $m\geq 1$.
	It then follows from \eqref{eq:fcoeffabsbound} that we can find a $c>0$ such that
	\begin{equation}\label{eq:fcoeffabsbound2}
	\sum_{p=1}^\infty \left| \DirProd{e_p}{\hat{f}_m}\right| \leq c \frac{q_k^{-5m/8}}{\|f_m\|_\Dir} \leq \frac{c}{c'}\, q_k^{-m/8}\qquad \text{for all }m\geq 1.
	\end{equation}
	
	With the help of Theorem \ref{thm:mainresult}(i) we may evaluate for $\alpha\in\frac{\pi}{2}\Z\setminus\{0\}$,
	\begin{align*}
	\sum_{\ell,p=1}^\infty W_{\ell,p}^{(\alpha)}(t) &= \sum_{\ell,p= 1}^\infty \DirProd{e_\ell}{\mathbf{Y}_k^{(\alpha)} e_p} = \sum_{\ell,p\geq 1} \sum_{m=1}^\infty \frac{2K(k)}{\pi}\frac{1}{m} q_k^{m|\alpha|/\pi} \DirProd{e_\ell}{\hat{f}_m}\DirProd{e_p}{\hat{f}_m},
	\end{align*}
	which according to our bound \eqref{eq:fcoeffabsbound2} is seen to be absolutely convergent.
	Interchanging the summations and performing the sum over $\ell$ and $p$, we find that there exists a $c''>0$ such that
	\begin{equation*}
	\sum_{\ell,p=1}^\infty W_{\ell,p}^{(\alpha)}(t) \leq c'' \sum_{m=1}^\infty \frac{1}{m} q_k^{m(\frac{|\alpha|}{\pi} - \frac{1}{4})} = c'' q_k^{\frac{|\alpha|}{\pi} - \frac{1}{4}} \sum_{m=1}^\infty \frac{1}{m} q_k^{(m-1)(\frac{|\alpha|}{\pi} - \frac{1}{4})}\leq c'' q_k^{\frac{|\alpha|}{\pi} - \frac{1}{4}} \sum_{m=1}^\infty \frac{1}{m} q_k^{\frac{1}{4}(m-1)}. 
	\end{equation*}
	Since the latter sum converges, we have obtained the desired estimate.
\end{proof}

\begin{lemma}\label{thm:Falphacombinatorics}
	For $\alpha\in\frac{\pi}{2}\Z$, the series $F^{(\alpha)}(t)$ may be expressed as the absolutely convergent sum
	\begin{equation}\label{eq:Falternatingsum}
	F^{(\alpha)}(t) = 4\sum_{m,\ell,p=1}^\infty (-1)^{\ell+p+m+1} m\, W_{2\ell,2p}^{(|\alpha|+m\pi/2)}(t).
	\end{equation}
\end{lemma}
\begin{proof}
	By Lemma \ref{thm:wsumbound}, $\sum_{\ell,p=1}^\infty W_{2\ell,2p}^{(|\alpha|+m\pi/2)}(t) < C q_k^{\frac{m}{2}+\frac{|\alpha|}{\pi} - \frac{1}{4}}$ for some $C>0$ and all $\alpha\in\frac{\pi}{2}\Z$ and $m\geq 1$, from which it follows that the right-hand side of \eqref{eq:Falternatingsum} is absolutely convergent.

	Next we use that for $\alpha'\in\frac{\pi}{2}\Z_{>0}$, the sets $\bigcup_{p\in 4\Z_{>0}} \mathcal{W}_{\ell,p}^{(\alpha')}$ and $\bigcup_{p\in 4\Z_{>0}-2} \mathcal{W}_{\ell,p}^{(\alpha')}$ are nearly in bijection.
	Indeed, a walk in the former is mapped to a unique walk in the latter by moving its starting point by $(\pm2,0)$ depending on the direction of the first step (see Figure \ref{fig:mappingfirststep}), while keeping the other sites fixed.
	\begin{figure}
		\centering
		\includegraphics[width=.32\linewidth]{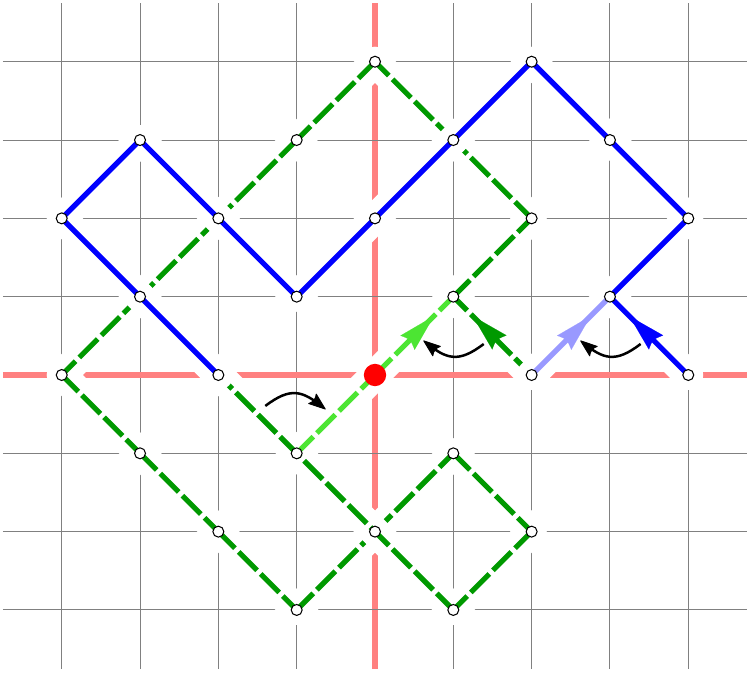}
		\caption{By changing the starting point the dark blue (solid) walk in $\mathcal{W}_{2,4}^{(\pi)}$ is mapped to a walk in $\mathcal{W}_{2,2}^{(\pi)}$. The  dark green (dashed) walk $w\in\mathcal{W}_{2,2}^{(\pi)}$ satisfies $|w_1| = |w_{|w|-1}| = \sqrt{2}$, such that moving its starting point and endpoint to the origin gives an excursion. }\label{fig:mappingfirststep}
	\end{figure}
	It is not hard to see that any walk $\bigcup_{p\in 4\Z_{>0}-2} \mathcal{W}_{\ell,p}^{(\alpha')}$ is obtained in such way except for those walks in $\mathcal{W}_{\ell,2}^{(\alpha')}$ that have $w_1=(1,\pm 1)$.
	The generating function of such walks is therefore given by
	\[
	 \sum_{w\in \mathcal{W}_{\ell,2}^{(\alpha')}} t^{|w|} \one_{\{|w_1|=\sqrt{2}\}} = \sum_{p=1}^\infty (-1)^{p+1} W^{(\alpha')}_{\ell,2p}(t). 
	\]
	An analogous argument for the endpoint then yields
	\begin{equation}\label{eq:Salphadef}
	S^{(\alpha')}(t)\coloneqq\sum_{w\in \mathcal{W}_{2,2}^{(\alpha')}} t^{|w|} \one_{\{|w_1|=|w_{|w|-1}|=\sqrt{2}\}} = \sum_{p,\ell=1}^\infty (-1)^{p+\ell} W^{(\alpha')}_{2\ell,2p}(t). 
	\end{equation}
	Let $X^{(\alpha')}$ be the set of four possible tuples $(w_1,w_{|w|-1},\theta^w_{|w|-1}-\theta^w_1)$ when $w\in \mathcal{W}_{2,2}^{(\alpha')}$ and $|w_1|=|w_{|w|-1}|=\sqrt{2}$.
	For instance, if $\alpha'=3\pi/2$ then 
	\begin{equation*}
	X^{(3\pi/2)} = \left\{ ((1,1),(-1,-1),\pi),\, ((1,1),(1,-1),3\pi/2),\, ((1,-1),(-1,-1),3\pi/2),\, ((1,-1),(1,-1),2\pi)\right\}.
	\end{equation*} 
	Then we may write
	\[
	S^{(\alpha')}(t) = \sum_{(x,y,\alpha)\in X^{(\alpha')}} \sum_{w\in \mathcal{W}_{2,2}^{(\alpha')}} t^{|w|} \one_{\{w_1=x,w_{|w|-1}=y\}} = \sum_{(x,y,\alpha)\in X^{(\alpha')}} \sum_{w\in\mathcal{E}}  t^{|w|} \one_{\{w_1=x,w_{|w|-1}=y\}}\one_{\{\theta^w = \alpha\}},
	\]
	where we used the obvious mapping to excursions by merely moving the starting point and endpoint to the origin (see Figure \ref{fig:mappingfirststep}).   
	By the rotational symmetry of the set of excursions we may replace the $\one_{\{w_1=x,w_{|w|-1}=y\}}$ by $1/4$ in the last sum.
	By observing that $\alpha$ takes exactly the four values $\alpha'-\pi/2$, $\alpha'$, $\alpha'$, $\alpha'+\pi /2$, one finds that
	\[
	S^{(\alpha')}(t) = \frac{1}{4}  \sum_{(x,y,\alpha)\in X^{(\alpha')}} \sum_{w\in\mathcal{E}}  t^{|w|}\one_{\{\theta^w = \alpha\}} = \frac{1}{4}F^{(\alpha'-\pi/2)}(t) + \frac{1}{2} F^{(\alpha')}(t)+\frac{1}{4}F^{(\alpha'+\pi/2)}(t).
	\]  
	From here it is an easy check by substitution that for $\alpha \in\frac{\pi}{2}\Z_{\geq0}$,
	\[
	4\sum_{m=1}^\infty (-1)^{m+1}\,m\,S^{(\alpha+m \pi/2)}(t) = F^{(\alpha)}(t) = F^{(-\alpha)}(t),
	\]
	which together with \eqref{eq:Salphadef} yields the claimed identity.
\end{proof}

We are now in the position to apply Theorem \ref{thm:mainresult}(i) and explicitly evaluate $F^{(\alpha)}(t)$, as well as its ``characteristic function''
\begin{equation}\label{eq:charfun}
F(t,b) \coloneqq \sum_{w\in\mathcal{E}} t^{|w|} e^{i b\theta^w} = \sum_{\alpha\in\frac{\pi}{2}\Z}  F^{(\alpha)}(t)\, e^{i b\alpha}.
\end{equation}
Notice that $b\mapsto F(t,b)$ is periodic with period $4$, and that by symmetry $F^{(\alpha)}(t) = F^{(-\alpha)}(t)$ and therefore $F(t,b)$ is real and $F(t,-b) = F(t,b)$.

\begin{proposition}\label{thm:excursiongenfun}
	The excursion generating functions are given by
	\begin{align}
		F^{(\alpha)}(t) &= \frac{2\pi}{K(k)} \sum_{n=1}^\infty \frac{(1-q_k^{n})^2}{1-q_k^{4n}}q_k^{n(\frac{2}{\pi}|\alpha|+1)}, \label{eq:excursionfixed}\\	
	F(t,b)&= \begin{cases} \displaystyle
	\frac{1}{\cos\left(\frac{\pi b}{2}\right)}\left[ 1-\frac{\pi \tan\left(\frac{\pi b}{4}\right)}{2K(k)} \frac{\theta_1'\left(\frac{\pi b}{4},\sqrt{q_k}\right)}{\theta_1\left(\frac{\pi b}{4},\sqrt{q_k}\right)}\right] & \text{for } b\in \R\setminus\Z, \\
	1-\frac{\pi}{2K(k)} & \text{for }b\in 4\Z,\\
	1 - \frac{2E(k)}{\pi} & \text{for }b\in 2\Z+1,\\
	-1 + \frac{4 E(k)}{\pi} - (1-k^2) \frac{2K(k)}{\pi} & \text{for }b\in 4\Z+2,
	\end{cases}\label{eq:Fexpression}
	\end{align}
	where $K(k)$ and $E(k)$ are the complete elliptic integrals of the first and second kind and $\theta_1(z,q)$ is a Jacobi theta function (see Appendix \ref{sec:ellipticappendix} for definitions).
\end{proposition}
\begin{proof}
	Combining Lemma \ref{thm:Falphacombinatorics} with Theorem \ref{thm:mainresult}(i) we find
	\begin{align}
	F^{(\alpha)}(t) &= 4 \sum_{m,\ell,p=1}^\infty (-1)^{\ell+p+m+1}m\,\DirProd{e_{2\ell}}{\mathbf{Y}_k^{(|\alpha|+m\pi/2)}e_{2p}}\nonumber\\
	&= \frac{4K(k)}{\pi}\sum_{m,n=1}^\infty (-1)^{m+1}\frac{m}{n} q_k^{n(\frac{2}{\pi}|\alpha|+m)}\frac{1}{\|f_{2n}\|^{2}} \left(\sum_{p=1}^\infty(-1)^p\DirProd{f_{2n}}{e_{2p}}\right)^2\nonumber\\
	&= \frac{4K(k)}{\pi}\sum_{n=1}^\infty \frac{1}{n} \frac{q_k^{n(\frac{2}{\pi}|\alpha|+1)}}{(1+q_k^{n})^2}\frac{1}{\|f_{2n}\|^{2}} \left(\sum_{p=1}^\infty(-1)^p\DirProd{f_{2n}}{e_{2p}}\right)^2,\label{eq:Falphasum}
	\end{align}
	where we used that $\DirProd{f_n}{e_{2p}}=0$ for $n$ odd.
	Since $f_{2m}(z)$ has radius of convergence larger than one (Lemma \ref{thm:fmradius}) and is even, we have
	\begin{align*}
	\sum_{p=1}^\infty (-1)^p\, \DirProd{f_{2n}}{e_{2p}} &= \sum_{p=1}^\infty (-1)^p\,2p\,  [z^{2p}] f_{2n}(z) = i f_{2n}'(i).
	\end{align*}
	From the definition \eqref{eq:vkmap} of $v_{k_1}$ one may read off that $v_{k_1}'(i) = 1/( 4 K(k_1) (1+k_1)) = 1/(4K(k))$ using \eqref{eq:landenk}, while $v_{k_1}(i) = i T_k$ using Lemma \ref{thm:mapping}(ii).
	Together with \eqref{eq:fmbasis} we then find
	\begin{align}
	\sum_{p=1}^\infty (-1)^p\, \langle e_{2p}, f_{2n} \rangle &= i\frac{g_{2n}'(iT_k)}{4 K(k)} = 
	(-1)^n \frac{\pi}{K(k)} n\sinh(4\pi n T_k) = (-1)^n \frac{\pi}{2K(k)}n (q_k^{-n}-q_k^n).\label{eq:falternating}
	\end{align}
	Combining with \eqref{eq:Falphasum} and \eqref{eq:fnorm} we arrive at 
	\[
	F^{(\alpha)}(t) = \frac{2\pi}{K(k)} \sum_{n=1}^\infty \frac{(q_k^{-n}-q_k^n)^2}{(1+q_k^n)^2}\frac{q_k^{n(\frac{2}{\pi}|\alpha|+1)}}{q_k^{-2n}-q_k^{2n}} = \frac{2\pi}{K(k)} \sum_{n=1}^\infty \frac{(1-q_k^{n})^2}{1-q_k^{4n}}q_k^{n(\frac{2}{\pi}|\alpha|+1)},
	\]
	establishing the first identity \eqref{eq:excursionfixed} of the Proposition.
	
	Using the identity 
	\[
	\sum_{\alpha\in\frac{\pi}{2}\Z} x^{\frac{2}{\pi}|\alpha|} e^{i b\alpha} = \frac{x^{-1}-x}{x+x^{-1}-2\cos(\pi b/2)}\qquad \text{for }0<x<1
	\]
	we obtain with some algebraic manipulation
	\begin{align}
	F(t,b) &= \frac{2\pi}{K(k)} \sum_{n=1}^\infty \frac{(1-q_k^{n})^2}{1-q_k^{4n}} \frac{1-q_k^{2n}}{q_k^{n}+q_k^{-n}-2\cos(\pi b/2)}\\
	&=\frac{2\pi}{K(k)} \sum_{n=1}^\infty \left(1-\frac{2}{q_k^n+q_k^{-n}}\right)\frac{1}{q_k^{n}+q_k^{-n}-2\cos(\pi b/2)}\label{eq:Fcomputestep}\\
	&= \frac{2\pi}{K(k)} \sum_{n=1}^\infty \left[ \left(1-\frac{1}{\cos(\pi b/2)}\right) \frac{1}{q_k^n+q_k^{-n}-2\cos(\pi b/2)} + \frac{1}{\cos(\pi b/2)} \frac{1}{q_k^n+q_k^{-n}}\right],\label{eq:Fcomputestep2} 
	\end{align} 
	where the last equality only holds for $b\in\R\setminus(2\Z+1)$.
	The first term in the sum can be handled for $b\in\R\setminus\Z$ by recognizing that
	\begin{align*}
	\sum_{n=1}^\infty \frac{1}{q_k^n+q_k^{-n}-2\cos(\pi b/2)} &= \sum_{n=1}^\infty \sum_{m=1}^\infty q_k^{m n} \frac{\sin(\pi b m/2)}{\sin(\pi b/2)} = \sum_{m=1}^\infty \frac{q_k^m}{1-q_k^m} \frac{\sin(\pi b m/2)}{\sin(\pi b/2)}\\
	&=\frac{1}{4\sin(\pi b/2)} \left[\frac{\theta_1'(\pi b/4,\sqrt{q_k})}{\theta_1(\pi b/4,\sqrt{q_k})}-\cot(\pi b/4)\right],
	\end{align*}
	where in the last equality we used \cite[16.29.1]{abramowitz_handbook_1964}.
	The second term follows from \cite[17.3.22]{abramowitz_handbook_1964},
	\[\sum_{n=1}^\infty \frac{1}{q_k^{n}+q_k^{-n}} = -\frac{1}{4} + \frac{K(k)}{2\pi}.
	\]
	Hence, for $b\in\R\setminus\Z$ we have
	\[
		F(t,b)=\frac{1}{\cos\left(\frac{\pi b}{2}\right)}\left[ 1-\frac{\pi \tan\left(\frac{\pi b}{4}\right)}{2K(k)} \frac{\theta_1'\left(\frac{\pi b}{4},\sqrt{q_k}\right)}{\theta_1\left(\frac{\pi b}{4},\sqrt{q_k}\right)}\right],
	\]
	as claimed.
	
	It remains to check the value of $F(t,b)$ at $b=0,1,2$, since $F(t,b+4) = F(t,b)$ and $F(t,-b)=F(t,b)$ as observed below \eqref{eq:charfun}.
	To this end we require two identities that follow from \cite[(23) and (26)]{bruckman_evaluation_1977}, namely
	\begin{align*}
	\sum_{n=1}^\infty \frac{1}{(q_k^n+q_k^{-n})^2} &= -\frac{1}{8} + \frac{1}{2}\sum_{n=-\infty}^\infty \frac{1}{(q_k^n+q_k^{-n})^2} = -\frac{1}{8} +\frac{1}{2\pi^2} K(k)E(k),\\
	\sum_{n=1}^\infty \frac{1}{(q_k^{n-1/2}+q_k^{1/2-n})^2} &= \frac{1}{2\pi^2}(K(k)E(k) - (1-k^2) K^2(k)). 
	\end{align*}
		Starting from \eqref{eq:Fcomputestep} or \eqref{eq:Fcomputestep2} and using these identities we obtain
	\begin{align*}
	F(t,0) &= \frac{2\pi}{K(k)} \sum_{n=1}^\infty \frac{1}{q_k^n+q_k^{-n}} = 1 - \frac{\pi}{2K(k)},\\
	F(t,1) &= \frac{4\pi}{K(k)} \left[ \frac{1}{2}\sum_{n=1}^\infty  \frac{1}{q_k^n+q_k^{-n}}- \sum_{n=1}^\infty  \frac{1}{(q_k^n+q_k^{-n})^2}\right] = \frac{4\pi}{K(k)} \left[\frac{K(k)}{4\pi} - \frac{E(k)K(k)}{2\pi^2}\right],\\
	F(t,2) &= \frac{2\pi}{K(k)} \left[2\sum_{n=1}^\infty \frac{1}{(q_{k}^{n/2}+q_{k}^{-n/2})^2} -\sum_{n=1}^\infty \frac{1}{q_k^n+q_k^{-n}}\right]\\
	&=\frac{2\pi}{K(k)} \left[2\sum_{n=1}^\infty \frac{1}{(q_{k}^{n}+q_{k}^{-n})^2} + 2\sum_{n=1}^\infty \frac{1}{(q_{k}^{n-1/2}+q_{k}^{1/2-n})^2} -\sum_{n=1}^\infty \frac{1}{q_k^n+q_k^{-n}}\right]\\
	&= \frac{2\pi}{K(k)}\left[ 2\frac{E(k)K(k)}{\pi^2} - (1-k^2) \frac{K^2(k)}{\pi^2} - \frac{K(k)}{2\pi}\right],
	\end{align*}
	which reduce to formulas in \eqref{eq:Fexpression}.
\end{proof}

There are more elementary ways to compute the quantity $F(t,b)$ at integer values of $b$, by directly considering the contributions of the individual walks to \eqref{eq:charfun}.
Indeed, $F(t,0)$ just counts all excursions from the origin without intermediate returns, and therefore is related to \eqref{eq:standardexcursiongenfun} via
\begin{equation*}
\frac{2}{\pi}K(4t) = \sum_{n=0}^\infty F(t,0)^n = \frac{1}{1-F(t,0)}.
\end{equation*}
Similarly, $F(t,1)$ precisely counts excursions that stay strictly within a single diagonal half-plane, while $F(t,2)$ counts excursions that stay strictly within a single quadrant.
This follows from a reflection principle that is a special case of Theorem \ref{thm:conegenfun} below, and indeed one may check that they are related to the cone generating functions $F^{(0,(-\pi/2,\pi/2))}(t) = F(t,1)/4$ and $F^{(0,(-\pi/4,\pi/4))}(t) = F(t,2)/4$ of Theorem \ref{thm:conegenfun}.
From elementary considerations we could therefore have determined that
\begin{align}
F(t,1) &= 4\sum_{n=0}^\infty  t^{2n+2} \sum_{p=0}^n \frac{1}{p+1} \binom{2p}{p} \binom{2(n-p)}{n-p} \binom{2n}{2p} = 1 - {_2F_1}(-\tfrac{1}{2},\tfrac{1}{2};1;(4t)^2), \label{eq:f1hyper}\\
\qquad 
F(t,2) &= 4\sum_{n=0}^\infty t^{2n+2}\left(\frac{1}{n+1}\binom{2n}{n}\right)^2 = -1 + {_2F_1}(-\tfrac{1}{2},-\tfrac{1}{2};1;(4t)^2), \label{eq:f2hyper}
\end{align}
in terms of the hypergeometric function ${_2F_1}(a,b;c;z) = \sum_{n=0}^\infty \frac{(a)_n(b)_n}{(c)_n} \frac{z^n}{n!}$.
With the help of \cite[17.3.10]{abramowitz_handbook_1964} and by noting that \eqref{eq:standardexcursiongenfun} implies $\frac{1}{16t} \frac{\partial}{\partial t} t \frac{\partial}{\partial t} F(t,2) = \frac{2}{\pi} K(4t)$ these can be shown to reproduce the explicit expressions in Proposition \ref{thm:excursiongenfun}.

By \eqref{eq:f1hyper} and \eqref{eq:f2hyper} and the fact \cite[17.3.9]{abramowitz_handbook_1964} that $\frac{2}{\pi}K(k) = {_2F_1}(\tfrac{1}{2},\tfrac{1}{2};1;k^2)$, all of the functions $F(t,b)$ at integer values $b$ can be expressed as (a rational function of) a hypergeometric function ${_2F_1}(\pm\tfrac{1}{2},\pm\tfrac{1}{2};1;(4t)^2)$ with appropriate signs.
It is a classical result \cite{schwartz_uber_1873} that all of these, hence all of $F(t,b)$ for $b\in\Z$, are transcendental in $t$.
Actually something stronger is true: $\frac{2}{\pi}K(k)$ and $\frac{2}{\pi}E(k)$ are algebraically independent in $k$.
Although we are certain this is a classical result, we were unable to find a reference.

\begin{lemma}\label{thm:EKindependent}
The series $\frac{2}{\pi}K(k)$ and $\frac{2}{\pi}E(k)$ are algebraically independent in $k$.
\end{lemma}
\begin{proof}
At $k=1/\sqrt{2}$ the complete elliptic integrals $K(k)$ and $E(k)$ take the value \cite[Theorem 1.7]{borwein_pi_1987}
\begin{equation*}
\frac{2K(1/\sqrt{2})}{\pi} = \frac{\Gamma^2(\tfrac{1}{4})}{2\pi^{3/2}},\qquad
\frac{2E(1/\sqrt{2})}{\pi} = \frac{8\pi^2 + \Gamma^4(\tfrac{1}{4})}{4\pi^{3/2}\Gamma^2(\tfrac{1}{4})}.
\end{equation*}
However, $\pi$ and $\Gamma(\tfrac{1}{4})$ (and $q_{1/2} =e^\pi$) are famously \cite{nesterenko_modular_1996} known to be algebraically independent over $\Q$.
Hence, $k$, $\frac{2}{\pi}K(k)$ and $\frac{2}{\pi}E(k)$ cannot satisfy a nontrivial polynomial equation with rational coefficients. 
\end{proof}

The situation is surprisingly different at other rational values of $b$:

\begin{corollary}\label{thm:Falgebraic}
	The power series $t\mapsto F(t,b)$ is algebraic if $b \in \Q \setminus\Z$ and transcendental if $b\in\Z$.
\end{corollary}
\begin{proof}
	We have already dealt with the case $b\in\Z$, so let us assume $b \in \Q \setminus \Z$.
	Using the Landen transformation $\theta_1'(u,q)/\theta_1(u,q) + \theta_4'(u,q)/\theta_4(u,q) = \theta_1'(u,\sqrt{q})/\theta_1(u,\sqrt{q})$, which follows e.g. from the series representations \cite[16.29.1 \& 16.29.4]{abramowitz_handbook_1964}, we may rewrite $F(t,b)$ as
	\begin{equation}\label{eq:Finthetaq}
	F(t,b) = \frac{1}{\cos\left(\frac{\pi b}{2}\right)}\left[ 1-\frac{\pi \tan\left(\frac{\pi b}{4}\right)}{2K(k)} \left( \frac{\theta_1'\left(\frac{\pi b}{4},q_k\right)}{\theta_1\left(\frac{\pi b}{4},q_k\right)}+\frac{\theta_4'\left(\frac{\pi b}{4},q_k\right)}{\theta_4\left(\frac{\pi b}{4},q_k\right)}\right)\right],
	\end{equation}
	which in turn can be expressed using Jacobi's zeta function $Z(u,k)$ (see \eqref{eq:zetadef} in Appendix \ref{sec:ellipticappendix}) as\footnote{Thanks to Kilian Raschel for pointing out this relation!}
	\begin{equation}\label{eq:FfromJacobiZ}
	F(t,b) = \frac{1}{\cos\left(\frac{\pi b}{2}\right)}\left[ 1-\tan\left(\frac{\pi b}{4}\right)\left(2Z\left(u,k\right) + \frac{\cn\left(u,k\right)\dn\left(u,k\right)}{\sn\left(u,k\right)}\right)\right] , \qquad u = K(k)\, b/2. 
	\end{equation}
	The numbers $\cos(\pi b/2)$, $\tan(\pi b/4)$ are algebraic because $\cos(\pi b/2)$ and $\sin(\pi b/2)$ occur as roots of Chebyshev polynomials.
	To show that $F(t,b)$ is algebraic, we thus have to check that $\sn(u,k)$, $\cn(u,k)$, $\dn(u,k)$, and $Z(u,k)$ are algebraic in $k$.
	
	The Jacobi elliptic functions $\sn$, $\cn$, $\dn$ satisfy addition formulas \cite[16.17.1-3]{abramowitz_handbook_1964}, e.g. 
	\begin{equation*}
	\sn(x+y) = \frac{\sn x\cn y\dn y+\sn y\cn x\dn x}{1-k^2 \sn^2 x\sn^2 y},
	\end{equation*}
	where all elliptic functions are understood to have modulus $k$.
	These addition formulas allow one to express $\sn(n x, k)$ as a rational function of $k$, $\sn(x,k)$, $\cn(x,k)$ and $\dn(x,k)$. 
	Setting $x = K(k) b/2$ such that $\sn(n x,k)=0$ and eliminating $\cn(x,k)$ and $\dn(x,k)$ using the quadratic relations \eqref{eq:quadraticrelations}, one obtains a polynomial equation for $\sn(x,k)$ with coefficients that are polynomials in $k$, implying that $\sn(K(k)b/2,k)$ is algebraic in $k$ and similarly for  $\cn(K(k)b/2,k)$ and  $\dn(K(k)b/2,k)$. 
	
	The Jacobi zeta function $Z$ also satisfies an addition formula \eqref{eq:zetaaddition}, that can be used together with the addition formulas for $\sn$, $\cn$ and $\dn$ to express $Z(nx,k) - n Z(x,k)$ as a rational function of $k$, $\sn(x,k)$, $\cn(x,k)$ and $\dn(x,k)$.
	Setting $x=K(k) b/2$, we have that $Z(nx,k) = 0$ \cite[17.4.29-30]{abramowitz_handbook_1964} and $\sn(x,k)$, $\cn(x,k)$ and $\dn(x,k)$ are algebraic in $k$, so the same is true for $Z(K(k) b/2,k)$.
\end{proof}

\subsection{Asymptotic behaviour}

As an intermezzo let us look at a probabilistic consequence of Proposition \ref{thm:excursiongenfun} in the limit $t\to 1/4$.
For a simple random walk on $\Z^2$ started at the origin we consider the total winding angle around the origin up to its first return to the origin, ignoring the contributions from the first and last step (as we have been doing all along for excursions).

\begin{corollary}\label{thm:returnangle}
Consider a simple random walk on $\Z^2$ started at the origin.
The probability that its winding angle around the origin upon its first return equals $\pi m/2$ for $m\in\Z$ is
\[
\frac{1}{\pi} \left(- \psi\left(\frac{|m|+1}{4}\right)+2\psi\left(\frac{|m|+2}{4}\right)  - \psi\left(\frac{|m|+3}{4}\right)\right),
\]	
where $\psi(x) = \Gamma'(x)/\Gamma(x)$ is the digamma function.
\end{corollary}
\begin{proof}
	Let $m\in \Z$ be fixed.
	First we rewrite the sum in \eqref{eq:excursionfixed} with $\alpha = m\pi/2$ as 
	\begin{align*}
\sum_{n=1}^\infty \frac{(1-q_k^{n})^2}{1-q_k^{4n}}q_k^{n(|m|+1)} &=  \sum_{n=1}^\infty \sum_{p=0}^\infty q_k^{n(|m|+4p+1)} (1-q_k^n)^2\\
& =  \sum_{p=0}^\infty \left[\frac{1}{1-q_k^{|m|+4p+1}} - \frac{2}{1-q_k^{|m|+4p+2}} + \frac{1}{1-q_k^{|m|+4p+3}}\right]
	\end{align*}
	As $k\to 1$, i.e. $q_k\to 1$, the latter is asymptotically equal to
	\begin{align*}
	\sum_{p=0}^\infty&\left[  \frac{1}{|m|+4p+1} - \frac{2}{|m|+4p+2} + \frac{1}{|m|+4p+3}\right] \frac{1}{1-q_k}(1 + o(1))\\
	&= \frac{1}{4}\left(- \psi\left(\frac{|m|+1}{4}\right)+2\psi\left(\frac{|m|+2}{4}\right)  - \psi\left(\frac{|m|+3}{4}\right)\right)\frac{1}{1-q_k}(1 + o(1)),
	\end{align*}
	where we used the series representation $\psi(x+1)-\psi(1) = \sum_{p=1}^\infty \left(\frac{1}{p}-\frac{1}{x+p}\right)$ of the digamma function.
	From the definition \eqref{eq:nomedef} of the nome it follows that $K(k)\log q_k = -\pi K(k')$, so $K(k)(1-q_k) \to \pi K(0) = \pi^2/2$ as $k\to 1$. 
	Hence, we find from \eqref{eq:excursionfixed} that as $t\to 1/4$,
	\[
	F^{(m\pi/2)}(t) \to \frac{1}{\pi} \left(- \psi\left(\frac{|m|+1}{4}\right)+2\psi\left(\frac{|m|+2}{4}\right)  - \psi\left(\frac{|m|+3}{4}\right)\right).
	\]
	From the definition \eqref{eq:Ffixeddef} it is easily seen that this is precisely the desired probability.
\end{proof}

Next we look at the growth of the coefficients $[t^{2\ell}]F(t,b)$ as $\ell\to\infty$.

\begin{lemma}\label{thm:Fasymptotics}
	For fixed $b \in [0,2]$ the coefficients of $t\mapsto F(t,b)$ satisfy the estimate 
	\[
	[t^{2\ell}] F(t,b) \sim \begin{cases}
\sin^2\left(\frac{\pi b}{4}\right)\frac{\Gamma(1+b)}{\pi}  \frac{4^{2(\ell+1-b)}}{\ell^{b+1}} & \text{for } b\in (0,2)\\
	\frac{\pi}{\ell\log^2 \ell}4^{2\ell} &\text{for } b=0\\
	\frac{1}{4\pi \ell^3}4^{2\ell}&\text{for } b=2.
	\end{cases}
	\]
	The analogous results for other $b\in\R$ follow from $F(t,b)=F(t,-b)=F(t,b+4)$. 
\end{lemma}
\begin{proof}
	We first consider the case $b\in(0,1)\cup(1,2)$.
	For fixed $z\in \R \setminus \pi \Z$, by \eqref{eq:nomebound} and \eqref{eq:theta1def} the functions $q_k^{-1/4} \theta_1(z,q_k)$ and $q_k^{-1/4} \theta_1'(z,q_k)$ are analytic in $k \in \C \setminus \{ w\in\R : w^2 \geq 1\}$ and both have only a single simple zero at $k=0$ \cite[16.36.2]{abramowitz_handbook_1964}. 
	Similarly $\theta_4(z,q_k)$ and $\theta_4'(z,q_k)$ are analytic in $k \in \C \setminus \{ w\in\R : w^2 \geq 1\}$ and $\theta_4(z,q_k)$ has no zeroes \cite[16.36.2]{abramowitz_handbook_1964}.
	It follows that 
	\begin{equation*}
	\frac{\theta_1'\left(z,q_k\right)}{\theta_1\left(z,q_k\right)}+\frac{\theta_4'\left(z,q_k\right)}{\theta_4\left(z,q_k\right)}
	\end{equation*}
	is analytic in $k \in \C \setminus \{ w\in\R : w^2 \geq 1\}$ for $z\in \R \setminus \pi \Z$.
	Since the same is true for $K(k)$ and $K(k)\neq 0$ (see \eqref{eq:nomebound}), we deduce from the expression \eqref{eq:Finthetaq} that $k\mapsto F(k/4,b)$ with $b\in(0,1)\cup(1,2)$ is analytic in $k \in \C \setminus \{ z\in\R : z^2 \geq 1\}$.
	Hence, we should focus on the behaviour of $F(k/4,b)$ as $k^2 \to 1$, which corresponds to $q_k \to 1$, $k' \to 0$ and $q_{k'} \to 0$.
	
	For $b$ fixed we may use Jacobi's identity (see e.g. \cite[\S22 (6)]{akhiezer_elements_1990})
	\begin{align*}
	\theta_1\left(\frac{\pi b}{4},\sqrt{q_k}\right) = \frac{i}{\sqrt{2T_k}} e^{-\frac{\pi b^2}{32 T_k}} \theta_1\left( -i\frac{\pi b}{8T_k}, q_{k'}^2\right). 
	\end{align*}
	Taking logarithmic derivatives in $b$ on both sides, we find
	\begin{align*}
	\frac{\pi }{4K(k)}\frac{\theta_1'\left(\frac{\pi b}{4},\sqrt{q_k}\right)}{\theta_1\left(\frac{\pi b}{4},\sqrt{q_k}\right)} &= - \frac{1}{K(k)}\left( \frac{i\pi}{8T_k} \frac{\theta_1'\left(-i\frac{\pi b}{8T_k},q_{k'}^2\right)}{\theta_1\left(-i\frac{\pi b}{8T_k},q_{k'}^2\right)} + \frac{\pi b}{16 T_k}\right)\\
	&= -\frac{\pi}{2K(k')} \left( i\frac{\theta_1'\left(-2i\pi b T_{k'},q_{k'}^2\right)}{\theta_1\left(-2i\pi b T_{k'},q_{k'}^2\right)} + \frac{b}{2}\right),
	\end{align*}
	where we used that $T_k = K(k')/(4K(k)) = 1/(16 T_{k'})$, see \eqref{eq:nomedef}.
	
	From the definition \eqref{eq:theta1def} with $\sin((2n+1)(-2i\pi bT_{k'})) = \frac{i}{2} ( q_{k'}^{(n+1/2)b} - q_{k'}^{-(n+1/2)b})$ we find the series representations
	\begin{align*}
	i \,\theta_1\left(-2i\pi b T_{k'},q_{k'}^2\right) &= \sum_{n=0}^\infty (-1)^n \left( q_{k'}^{(n+1/2)(2n-b+1)}- q_{k'}^{(n+1/2)(2n+b+1)}\right)\\
	&= q_{k'}^{\frac{1}{2}(1-b)} - q_{k'}^{\frac{1}{2}(1+b)} + O\left( q_{k'}^{\frac{3}{2}(3-b)} \right),\\
	\theta_1'\left(-2i\pi b T_{k'},q_{k'}^2\right) &= \sum_{n=0}^\infty (-1)^n (2n+1)\left( q_{k'}^{(n+1/2)(2n-b+1)}+ q_{k'}^{(n+1/2)(2n+b+1)}\right)\\
	&= q_{k'}^{\frac{1}{2}(1-b)} + q_{k'}^{\frac{1}{2}(1+b)} + O\left( q_{k'}^{\frac{3}{2}(3-b)} \right),
	\end{align*}
	where the asymptotics as $q_{k'}\to 0$ are valid because $0< b < 2$.
	Together with $K(k') = \frac{\pi}{2} + 2\pi q_{k'} + O(q_{k'}^2)$ (see \eqref{eq:Kfromq}), we obtain 
	\begin{align*}
	\frac{\pi }{4K(k)}\frac{\theta_1'\left(\frac{\pi b}{4},\sqrt{q_k}\right)}{\theta_1\left(\frac{\pi b}{4},\sqrt{q_k}\right)} & 
	= \left(1 - \tfrac{b}{2}\right) (1-4 q_{k'}) + 2 q_{k'}^b + O\left(q_{k'}^{2b} + q_{k'}^2\right).
	\end{align*}
	 Hence, for $b\in(0,1)\cup(1,2)$ as $k\to \pm 1$ we find from \eqref{eq:Fexpression} that \footnote{The constant term is exactly the characteristic function of the probability distribution in Corollary \ref{thm:returnangle}.} 
	 \begin{align*}
	 F(t,b) &= \frac{1+(b-2)\tan\left(\frac{\pi b}{4}\right)}{\cos\left(\frac{\pi b}{2}\right)} - 4 \frac{\tan\left(\frac{\pi b}{4}\right)}{\cos\left(\frac{\pi b}{2}\right)} \left( q_{k'}^b + (b-2)q_{k'}\right)+ O\left(q_{k'}^{2b}+ q_{k'}^2\right).
	 \end{align*}
	 Since $q_{k'} = (1-k^2)/16 + O((1-k^2)^2)$, the dominant contribution at the singularity is proportional to $(1-k^2)^b$ and then a transfer theorem \cite[Theorem 3A]{flajolet_singularity_1990} tells us that
	 \begin{equation}\label{eq:Fasympproof}
	 [t^{2\ell}] F(t,b) \sim - 4^{1-2b} \frac{\tan\left(\frac{\pi b}{4}\right)}{\cos\left(\frac{\pi b}{2}\right)\Gamma(-b)}\frac{4^{2\ell}}{\ell^{b+1}}  = \sin^2\left(\frac{\pi b}{4}\right)\frac{\Gamma(1+b)}{\pi}  \frac{4^{2(\ell+1-b)}}{\ell^{b+1}},
	 \end{equation}
	 where we used the reflection formula $\Gamma(1-z)\Gamma(z) = \pi/\sin(\pi z)$.
	 This proves the claimed asymptotics for $b\in (0,1)\cup(1,2)$.
	 
	 Next we look at $F(t,b)$ for $b=0,1,2$.
	 Using \eqref{eq:nomedef} and Legendre's relation \eqref{eq:legendre} each can be expressed  in terms of $K'(k)$, $E'(k)$ and $q_{k'}$, which in turn are analytic in $1-k^2$,
	 \begin{align*}
	 F(t,0) &= 1 - \frac{\pi}{2K(k)} = 1 + \frac{\pi^2}{2K'(k)\log q_{k'}} = 1 + \frac{\pi}{ \log\frac{1-k^2}{16}} + O\left( 1-k^2\right),\\
	 F(t,1) &= 1 - 2 \frac{E(k)}{\pi} = 1 + \frac{2}{\pi^2} (K'(k) - E'(k)) \log q_{k'} - \frac{1}{K'(k)} \\
	 &= 1 - \frac{2}{\pi} + \frac{1}{2\pi} (1-k^2)\log \frac{1-k^2}{16}  + O\left( 1-k^2 \right),\\
	 F(t,2) &= -1 + \frac{2}{K'(k)} + \frac{4}{\pi^2} \left( E'(k) - \frac{1+k^2}{2}K'(k)\right)\log q_{k'} \\
	 &= \frac{4}{\pi} - 1 - \frac{1}{\pi}(1-k^2) - \frac{1}{8\pi} (1-k^2)^2 \log \frac{1-k^2}{16} + O\left( (1-k^2)^2\right).
	 \end{align*}
	 In each case the term containing $\log \frac{1-k^2}{16}$ produces the main contribution to the asymptotics $[t^{2\ell}]F(t,b)$.
	 Using transfer theorems \cite[Theorem 3A \& remarks at the end of Section 3]{flajolet_singularity_1990} we establish that
	 \begin{equation*}
	 [t^{2\ell}]F(t,0) \sim \frac{\pi}{\ell\log^2 \ell}\,4^{2\ell}, \qquad [t^{2\ell}]F(t,1) \sim \frac{1}{2\pi \ell^2}\, 4^{2\ell}, \qquad [t^{2\ell}]F(t,2) \sim \frac{1}{4\pi \ell^3}\, 4^{2\ell}.
	 \end{equation*}
	 Finally, observe that the formula for $b=1$ agrees with \eqref{eq:Fasympproof}, while this is not the case for $b=0$ and $b=2$. 
\end{proof}

\subsection{Excursions restricted to an angular interval}

We turn to the problem of enumerating excursions $w$ with winding angle sequence $(\theta_i^w)_{i=0}^{|w|}$ restricted to lie fully within an interval $I\subset\R$.
For convenience we let $\mathcal{E}' = \{ w\in\mathcal{E} : w_1 = (1,1) \}$ be the set of excursions that leave the origin in a fixed direction.
Then we let $F^{(\alpha,I)}(t)$ be the generating function
\[
F^{(\alpha,I)}(t) = \sum_{w\in\mathcal{E}'} t^{|w|} \one_{\{\theta^w = \alpha\text{ and }\theta_i^w\in I\text{ for }0\leq i\leq |w|\}}.
\]
Notice that with this definition $F^{(\alpha,\R)}(t) = F^{(\alpha)}(t)/4$. 

\begin{theorem}\label{thm:conegenfun}
For $I=(\beta_-,\beta_+)$, $\beta_- \in \frac{\pi}{4}\Z_{<0}$ and $\beta_+ \in \frac{\pi}{4}\Z_{>0}$, and $\alpha\in I\cap \frac{\pi}{2}\Z$ the generating function $F^{(\alpha,I)}(t)$ is given by the finite sum
\begin{equation}\label{eq:conegenfun}
F^{(\alpha,I)}(t)=\frac{\pi}{8\delta} \sum_{\sigma\in(0,\delta)\cap \frac{\pi}{2}\Z}\left( \cos\left( \frac{4\sigma\alpha}{\delta} \right) - \cos\left( \frac{4\sigma(2\beta_+ - \alpha)}{\delta} \right)\right) F\left(t,\frac{4\sigma}{\delta}\right), \qquad \delta \coloneqq 2(\beta_+-\beta_-).
\end{equation}
It is transcendental in $t$ if one of the following conditions is satisfied:
\begin{itemize}
	\item $\beta_+,\beta_-\in \frac{\pi}{2}\Z +\frac{\pi}{4}$,
	\item $\beta_+,\beta_-\in \pi\Z +\frac{\pi}{2}$ and $\alpha \in\pi\Z$.
\end{itemize}
Otherwise it is algebraic in $t$. 
Moreover, its coefficients satisfy the asymptotic estimate
\[
[t^{2\ell}] F^{(\alpha,I)}(t) \sim \left( \cos\left( \frac{2\pi\alpha}{\delta} \right) - \cos\left( \frac{2\pi(2\beta_+ - \alpha)}{\delta} \right)\right) \sin^2\left(\frac{\pi^2}{2\delta}\right) \frac{\Gamma\left(1+\frac{2\pi}{\delta}\right)}{4\delta} \frac{4^{2(\ell+1-2\pi/\delta)}}{\ell^{1+2\pi/\delta}}.
\]
\end{theorem}

\begin{proof}
	By a reflection principle that is completely analogous to that used in Lemma \ref{thm:reflection} we observe that $F^{(\alpha,I)}(t)$ is given by the sum
	\[
		F^{(\alpha,I)}(t) = \frac{1}{4} \sum_{n=-\infty}^\infty \left(F^{(\alpha+n\delta)}(t) - F^{(2\beta_+-\alpha+n\delta)}(t)\right),  \qquad \delta \coloneqq 2(\beta_+-\beta_-),
	\]
	where the factor $1/4$ is due to the fourfold difference between $\mathcal{E}$ and $\mathcal{E}'$.
	
	Then the discrete Fourier transform 
	\[
	\frac{\pi}{2\delta} \sum_{\sigma\in(0,\delta)\cap \frac{\pi}{2}\Z} \left( e^{-4i\sigma \alpha/\delta} - e^{-4i\sigma (2\beta_+-\alpha)/\delta}\right) e^{4i \sigma \alpha'/\delta} = \one_{\{\alpha-\alpha'\in\delta\Z\}}-\one_{\{2\beta_+-\alpha-\alpha'\in\delta\Z\}}
	\]
	with $\alpha'\in\frac{\pi}{2}\Z$ leads to
	\[
	F^{(\alpha,I)}(t) = \frac{\pi}{8\delta}\sum_{\sigma\in(0,\delta)\cap \frac{\pi}{2}\Z} \left( e^{-4i\sigma \alpha/\delta} - e^{-4i\sigma (2\beta_+-\alpha)/\delta}\right)\sum_{\alpha'\in\frac{\pi}{2}\Z}  F^{(\alpha')}(t)e^{4i\sigma\alpha'/\delta}.
	\]
	Using that the rightmost sum, by definition \eqref{eq:charfun}, equals $F(t,4\sigma/\delta) = F(t,4(\delta-\sigma)/\delta)$, the latter is seen to agree exactly with \eqref{eq:conegenfun}.

	Since $F^{(\alpha,I)}(t)$ is expressed in terms of trigonometric functions at rational angles (meaning rational multiples of $\pi$) and $F(t,b)$ at rational values of $b$, it follows from Corollary \ref{thm:Falgebraic} that $F^{(\alpha,I)}(t)$ is algebraic unless the sum includes one or more $F(t,b)$ at integer values of $b$.
	Note that $F(t,0)$ and $F(t,4)$ cannot occur, while $F(t,3)=F(t,1)$ and both occur with the same prefactor, because the summand in \eqref{eq:conegenfun} is invariant under the replacement $\sigma \to \delta -\sigma$.
	By Lemma \ref{thm:EKindependent} and \eqref{eq:Fexpression}, the series $F(t,1)$ and $F(t,2)$ are algebraically independent, so $F^{(\alpha,I)}(t)$ is transcendental if and only if at least one of $F(t,1)$ and $F(t,2)$ is present with a nonzero prefactor.
	We check this case by case:
	\begin{itemize}
	\item $\beta_+,\beta_-\in \frac{\pi}{2}\Z +\frac{\pi}{4}$: $F^{(\alpha,I)}(t)$ is transcendental because it contains a term with $\sigma = \delta/2$ equal to $(\cos(2\alpha) - \cos(4\beta_+-2\alpha))F(t,2) = 2\cos(2\alpha) F(t,2) = \pm 2F(t,2)$;
	\item $\beta_+-\beta_-\in \frac{\pi}{2}\Z + \frac{\pi}{4}$: $F^{(\alpha,I)}(t)$ is algebraic because it contains no term with $\sigma = \delta/2$ or $\sigma = \delta/4$. 
	\end{itemize} 
	If $\beta_+,\beta_-\in \frac{\pi}{2}\Z$ then the term with $\sigma=\delta/2$ equals $(\cos(2\alpha) - \cos(4\beta_+-2\alpha))F(t,2) = 0$, so we should focus on $\sigma=\delta/4$.
	If it exists, its prefactor equals $\cos(\alpha) - \cos(2\beta_+ - \alpha) = \cos(\alpha)(1-\cos(2\beta_+))$.
	\begin{itemize}
	\item $\beta_+,\beta_-\in\frac{\pi}{2}\Z$, $\beta_+-\beta_- \in \pi\Z + \frac{\pi}{2}$: $F^{(\alpha,I)}(t)$ is algebraic, because it contains no term with $\sigma = \delta/4$;
	\item $\beta_+,\beta_-\in\pi\Z$: $F^{(\alpha,I)}(t)$ is algebraic, because $1-\cos(2\beta_+)=0$;
	\item $\beta_+,\beta_-\in\pi\Z+\frac{\pi}{2}$, $\alpha \in \pi\Z+\frac{\pi}{2}$: $F^{(\alpha,I)}(t)$ is algebraic, because $\cos(\alpha)=0$;
	\item $\beta_+,\beta_-\in\pi\Z+\frac{\pi}{2}$, $\alpha \in \pi\Z$: $F^{(\alpha,I)}(t)$ is transcendental, because $\cos(\alpha)(1-\cos(2\beta_+))=\pm 2$. 
	\end{itemize}
	Since these six cases exhaust all choices of $\beta_-$, $\beta_+$, and $\alpha$, we have proven the criterion in the Theorem statement. 
	
	According to Lemma \ref{thm:Fasymptotics} the asymptotics of $F^{(\alpha,I)}(t)$ is determined by the terms $\sigma = \pi/2$ and $\sigma = \delta - \pi/2$.
	If $\delta > \pi$ then $\pi/2 \neq \delta - \pi/2$ and the corresponding terms are equal and involve $F(t,b)$ at $b = 2\pi/\delta \in (0,2)$.
	If $\delta = \pi$, then there is a single term $\sigma = \pi/2$ that involves $F(t,2)$, but its asymptotic growth is twice larger than what one gets for $F(t,b)$ with $b\in(0,2)$ as $b\to 2$ (see Lemma \ref{thm:Fasymptotics}). 
	Hence, in either case Lemma \ref{thm:Fasymptotics} implies that as $\ell\to\infty$,
	\[
	[t^{2\ell}]F^{(\alpha,I)}(t) \sim \frac{\pi}{4\delta} \left( \cos\left( \frac{2\pi\alpha}{\delta} \right) - \cos\left( \frac{2\pi(2\beta_+ - \alpha)}{\delta} \right)\right) \sin^2\left(\frac{\pi^2}{2\delta}\right) \frac{\Gamma\left(1+\frac{2\pi}{\delta}\right)}{\pi} \frac{4^{2(\ell+1-2\pi/\delta)}}{\ell^{1+2\pi/\delta}},
	\]
	in accordance with the claimed result.
\end{proof}

As a special case we look at excursions that stay in the angular interval $(-\pi/4,\pi/2)$, see Figure \ref{fig:gesselwalk}.

\begin{corollary}[Gessel's lattice path conjecture]\label{thm:gessel}
	The generating function of excursions that stay in the angular interval $(-\pi/4,\pi/2)$ with winding angle $\alpha=0$ is
	\begin{align}
F^{(0,(-\pi/4,\pi/2))}(t) = \frac{1}{4} F\left(t,\frac{4}{3}\right) = \frac{1}{2} \left[ \frac{\sqrt{3}\pi}{2K(4t)}  \frac{\theta_1'\left(\frac{\pi}{3},\sqrt{q_k}\right)}{\theta_1\left(\frac{\pi}{3},\sqrt{q_k}\right)}-1\right]=\sum_{n=0}^\infty t^{2n+2}\,16^n \frac{(5/6)_n(1/2)_n}{(2)_n(5/3)_n}.
	\end{align}
\end{corollary}
\begin{proof}
	The first two equalities follow from Theorem \ref{thm:conegenfun} and Proposition \ref{thm:excursiongenfun} respectively (using that $F(t,4-b) = F(t,b)$). 
	It remains to show that our generating function reproduces the known formula \cite{kauers_proof_2009,bostan_complete_2010}
	\[
	\sum_{n=0}^\infty t^{2n+2}\,16^n \frac{(5/6)_n(1/2)_n}{(2)_n(5/3)_n} = \frac{1}{2}\left[ {_2F_1}\left(-\frac12,-\frac16;\frac23;(4t)^2\right)-1\right],
	\]
	which we will do by showing they both solve the same algebraic equation (and checking that the first few terms in the expansion agree).
	
	Denoting $y=\frac12 F(t,4/3)+1$ and using \eqref{eq:FfromJacobiZ} we get (with $k=4t$)
	\[
	 y = \sqrt{3} \left( 2Z(u,k) + \frac{\cn\left(u,k\right)\dn\left(u,k\right)}{\sn\left(u,k\right)} \right), \qquad u = \frac{2 K(k)}{3}. 
	\]
	Applying the addition formula \eqref{eq:zetaaddition} to $Z(u+u+u,k)$ twice one finds
	\begin{align*}
	0=Z(3 u,k) &= 3 Z(u,k) - k^2\left( \sn(2u,k) \sn^2(u,k)+\sn(2u,k)\sn(3u,k)\sn(u,k)\right)\\
	&= 3Z(u,k) - k^2 \sn(2u,k) \sn^2(u,k),
	\end{align*}
	where we used that $Z(2K(k),k) = \sn( 2K(k),k)=0$.
	Using the double argument formula \cite[16.18.1]{abramowitz_handbook_1964}
	\[
	\sn(2u,k) = \frac{2\sn(u,k)\cn(u,k)\dn(u,k)}{1-k^2 \sn^4(u,k)}
	\]
	as well as the quadratic relations \eqref{eq:quadraticrelations}, one may express $y$ in terms of $x\coloneqq \dn(u,k)$ as
	\begin{equation}\label{eq:gesselfexpr}
	y = -x\,\frac{(1-x^2)^2 +3k^2}{(1-x^2)^2-k^2}\, \sqrt{\frac{x^2+k^2-1}{3(1-x^2)}}. 
	\end{equation}
	Using the various addition theorems applied to $\sn(u+u+u,k)=0$ and rewriting in terms of $x=\dn(u,k)$ one finds after a slightly tedious calculation that $x$ solves
	\[
	k^2 = \frac{(1-x)(1+x)^3}{1+2x}.
	\]
	Eliminating $k$ from \eqref{eq:gesselfexpr}, $y$ is then seen to be given by
	\[
	y=\frac{3-(1-x)^2}{\sqrt{3(1+2x)}}.
	\]
	Finally one may check that the last two displays (with $k^2=16t^2$) parametrize the curve
	\begin{align*}
	27 y^8&-\left(4608 t^4+4032 t^2+18\right)y^4+\left(-32768
	t^6+67584 t^4+4224 t^2-8\right) y^2\\
	&	-65536 t^8-114688 t^6-50688 t^4-448 t^2-1=0,
	\end{align*}
	which is equivalent to the polynomial in \cite[Corollary 2]{bostan_complete_2010} (after substituting $\mathtt{t}\to t^2$, $\mathtt{T} \to (y-1)/(2t^2)$).
\end{proof}

\section{Winding angle distribution of an unconstrained random walk}\label{sec:unconstrained}

Let $(W_i)_i$ be the simple random walk on $\Z^2$ started at the origin.
In this section we use the results of Theorem \ref{thm:mainresult} to determine statistics of the winding angle $\theta_j^z$ of $(W_i)_i$ up to time $j$ around a lattice point $z\in\Z^2$ or a dual lattice point $z\in(\Z+1/2)^2$.
In the case that $(W_i)_i$ hits the lattice point $z$ at time $i$, we set $\theta_i^z = \theta_{i-1}^z$ and $\theta_j^z = \infty$ for $j > i$.
From a physicist's point of view, we may regard the $z\in\Z^2$ as an obstacle with \emph{absorbing} boundary conditions and $z\in(\Z+1/2)^2$ as one with \emph{reflecting} boundary conditions.

As alluded to in the introduction, the asymptotics of $\theta_j^z$ in the limit $j\to\infty$ has been well-studied in the literature. 
If $z\in (\Z+1/2)^2$ then one has the convergence in distribution to the \emph{hyperbolic secant distribution} \cite{rudnick_winding_1987,belisle_windings_1989,belisle_winding_1991,shi_windings_1998}
\[
\prob\left[
\frac{2\theta_j^z}{\log j} \in (a,b) \right] \xrightarrow{j\to\infty}  \int_a^b \frac{1}{2}\sech\left(\frac{\pi x}{2}\right) \rmd x. 
\]
Otherwise if $z\in \Z^2$, then conditionally on $\theta_j^z<\infty$ one has the convergence \cite{rudnick_winding_1987} (see also \cite{drossel_winding_1996})
\[
\prob\left[
\frac{2\theta_j^z}{\log j} \in (a,b) \,\middle| \,\theta_j^z < \infty \right] \xrightarrow{j\to\infty}  \int_a^b \frac{\pi}{4}\sech^2\left(\frac{\pi x}{2}\right) \rmd x. 
\]

As we will see now, these asymptotic laws have discrete counterparts if the combinatorics is setup appropriately.
First of all we restrict to winding angles around points that are close to the origin.
To be precise, we consider the lattice point $z^\bullet\coloneqq(0,0)$ and the dual lattice point $z^\square \coloneqq(-1/2,-1/2)$ (or any of the other dual lattice points at distance $1/\sqrt{2}$ from the origin) and we denote the corresponding winding angles by $\theta^\bullet_{j}$ and $\theta^\square_{j}$ respectively.
Moreover, instead of considering the winding angle at integer time $j$, we look at angles $\theta^\bullet_{j+1/2}$ and $\theta^\square_{j+1/2}$ at half-integer time $j+1/2$ (see Figure \ref{fig:freewinding}), which prevents them from taking certain unwanted exact multiples of $\pi/4$.
Finally, we replace the fixed time by a geometric random variable.

\begin{theorem}[Hyperbolic secant laws]\label{thm:hypsecant}
If $\zeta_k$ is a geometric random variable with parameter $k\in(0,1)$ independent of the random walk, i.e. $\prob[ \zeta_k = j ] = k^j(1-k)$ for $j\geq 0$,
then 
\begin{align*}
\prob\big[ \theta^\square_{\zeta_k+1/2} \in (\alpha-{\textstyle\frac{\pi}{2}}, \alpha+{\textstyle\frac{\pi}{2}}) \big] &= c_k \sech\left(4T_k\alpha\right) + {\textstyle\frac{k-1}{k}}\one_{\{\alpha=0\}}&&\text{for }\alpha\in{\textstyle\frac{\pi}{2}}\Z,\\
\prob\big[ \theta^{\bullet}_{\zeta_k+1/2} \in (\alpha-{\textstyle\frac{\pi}{2}}, \alpha+{\textstyle\frac{\pi}{2}})\big] &= C_k  \sech\left(4T_k(\alpha-{\textstyle\frac{\pi}{4}})\right)\sech\left(4T_k(\alpha+{\textstyle\frac{\pi}{4}})\right) &&\text{for }\alpha\in{\textstyle\frac{\pi}{2}\Z + \frac{\pi}{4}}, 
\end{align*} 
where $c_k = \pi/(2kK(k))$, $C_k = \pi\sinh\left(2T_k\pi\right)/(2K(k))$ and as usual $T_k = K(k')/(4K(k))$.
\end{theorem}

A consequence of these hyperbolic secant laws is the following probabilistic interpretation of the Jacobi elliptic functions $\cn$ and $\dn$.

\begin{corollary}[Jacobi elliptic functions as characteristic functions]\label{thm:charfuns}
Let $\theta^\square_{j+1/2}$ be the winding angle around $(-1/2,-1/2)$ as in Theorem \ref{thm:hypsecant} with the convention $\theta^\square_{-1/2}=0$.
If we denote by $\{\cdot\}_{A} : \R \to A$ rounding to the closest element of $A\subset\R$ (conflicts do not arise here) then we have the characteristic functions
\begin{alignat}{2}
\expec \exp\Big[i b \{ &\theta^\square_{\zeta_k-1/2} \}_{\pi\Z}&&\Big] = \dn(K(k) b, k), \\
\expec \exp\Big[i b \{ &\theta^\square_{\zeta_k+1/2} \}_{\pi\Z+\frac{\pi}{2}}&&\Big] = \cn(K(k) b, k),
\end{alignat} 
where $\cn(\cdot,k)$ and $\dn(\cdot, k)$ are Jacobi elliptic functions with modulus $k$.
\end{corollary}

Even though Theorem \ref{thm:hypsecant} and Corollary \ref{thm:charfuns} are stated for simple walks on the square lattice, the proofs that follow deal exclusively with simple diagonal walks. 
In order to approach the problem we require a new building block, in the sense of Section \ref{sec:walksaxis}, that involves walks that start on an axis but end at a general point.
In particular we will consider the set $\mathcal{C}_n$ of non-empty diagonal walks $w$ starting at $(n,0)$ and staying strictly inside the positive quadrant, i.e. $w_i \in \Z_{>0}^2$ for $1\leq i\leq |w|$, with generating function $C_n(t)$.

\begin{lemma}\label{thm:arbitraryendpointsum}
	For any $m>0$, $C_n(t)$ satisfies
	\[
	\sum_{n=1}^\infty \left(\frac{1}{2} + C_n(t)\right) \DirProd{e_n}{f_m} = (-1)^m\frac{\pi}{8 (1-k)K(k)} \, \frac{m(q_k^{-m}-q_k^m)}{(q_k^{m/4}+q_k^{-m/4})^2}.
	\]
\end{lemma}
\begin{proof}
	Let us consider the set of all (possibly empty) diagonal walks starting at $(p,0)$ for $p>0$, which has generating function $1/(1-k)$.
	Among these the walks that visit the vertical axis for the first time at $(\pm \ell, 0)$ for $\ell>0$ have generating function $\frac{2}{1-k} J_{\ell,p}$ (see Figure \ref{fig:cdecomposition1}).
	On the other hand, the walks that first visit the origin have generating function that can be expressed as the alternating sum $\frac{2}{1-k}\sum_{\ell=1}^\infty (-1)^{\ell+1} J_{2\ell,p}$ using an argument analogous to the one in the proof of Lemma \ref{thm:Falphacombinatorics}. 
	Hence, the remaining walks, those starting at $(p,0)$ and avoiding the vertical axis, have generating function
	\[
	\frac{1}{1-k} - \frac{2}{1-k} \sum_{\ell=1}^\infty J_{\ell,p}(t) - \frac{2}{1-k}\sum_{\ell=1}^\infty (-1)^{\ell+1} J_{2\ell,p} = \frac{1}{1-k} \sum_{\ell=1}^\infty\frac{1}{\ell}\DirProd{e_\ell}{\left[I - 2\left(1-\cos\frac{\pi \ell}{2}\right)\mathbf{J}_k\right]e_p}\!.
	\]	 
	\begin{figure}
	\centering
	\begin{subfigure}[c]{0.35\textwidth}
		\centering
		\includegraphics[height=.75\linewidth]{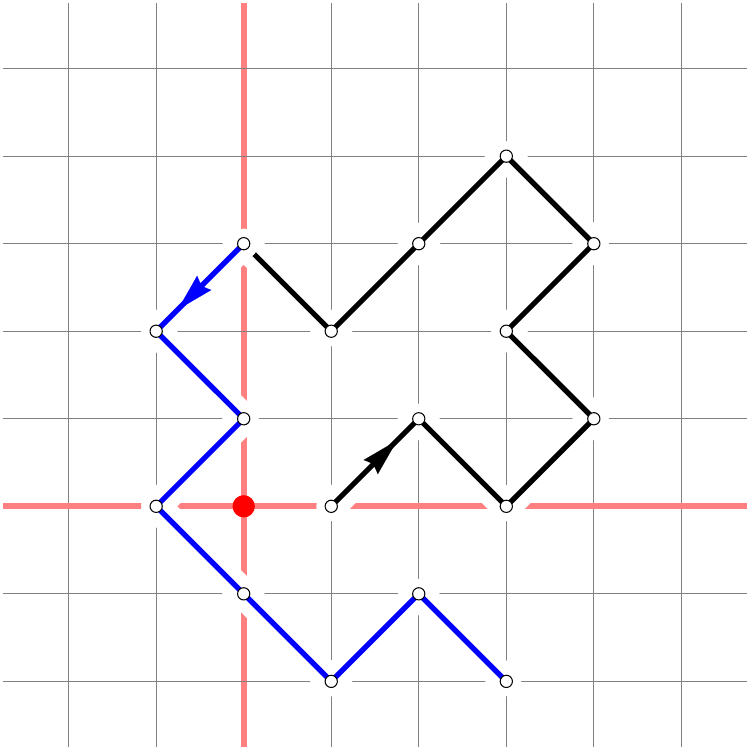}
		\caption{}\label{fig:cdecomposition1}
	\end{subfigure}%
	\begin{subfigure}[c]{0.35\textwidth}
		\centering
		\includegraphics[height=.75\linewidth]{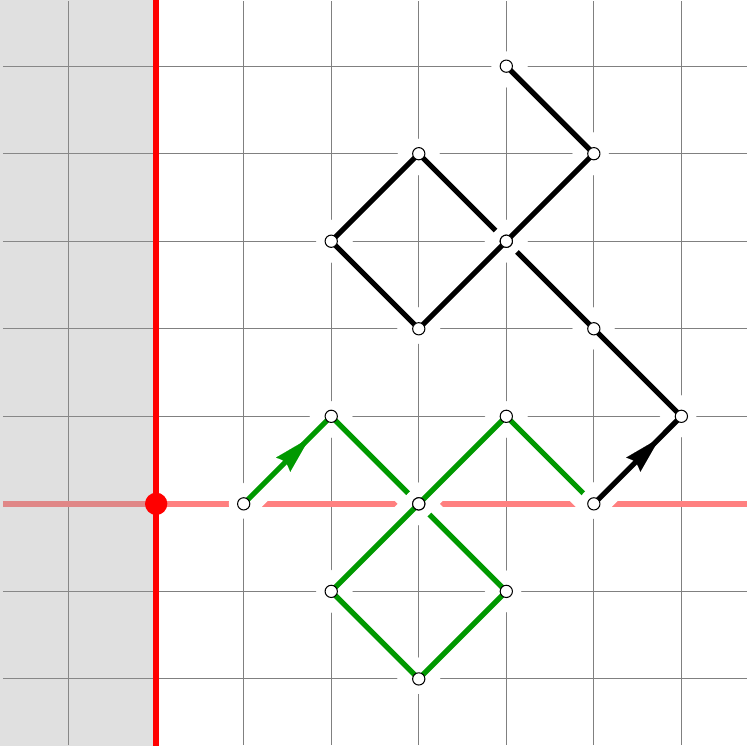}
		\caption{}\label{fig:cdecomposition2}
	\end{subfigure}%
	\caption{(a) An unconstrained walk starting at $(1,0)$ that visits the vertical axis decomposes into a walk in $\mathcal{J}_{3,1}$ (in black) and another unconstrained walk (in blue). (b) A walk starting at $(1,0)$ and avoiding the vertical axis decomposes into a walk in $\mathcal{B}_{5,1}$ (in green) and a walk in $\mathcal{C}_5$ (in black). }\label{fig:cdecomposition}
	\end{figure}%
	By further decomposing such walks at their last intersection with the horizontal axis (see Figure \ref{fig:cdecomposition2}), we find that the latter generating function equals
	\[
	\sum_{n=1}^\infty (1+2C_n(k)) \DirProd{e_n}{\mathbf{B}_ke_p},
	\]
	where the addition of $1$ in parentheses takes into account the walks ending on the horizontal axis.
	The equality of the last two displays together with Propositions \ref{thm:Bdiagonalization} and \ref{thm:Jdiagonalization} implies that
	\begin{align*}
	\sum_{n=1}^\infty \left(\frac{1}{2}+C_n(k)\right) \DirProd{e_n}{f_m} &= \frac{\pi}{2K(k)} \frac{m(1+q_k^m)}{1-q_k^m} \sum_{n=1}^\infty \left(\frac{1}{2}+C_n(k)\right) \DirProd{e_n}{\mathbf{B}_kf_m} \\
	&= \frac{\pi}{4(1-k)K(k)} \frac{m(1+q_k^m)}{1-q_k^m} \sum_{\ell=1}^\infty\frac{1}{\ell}\DirProd{e_\ell}{\left[I - 2\left(1-\cos\frac{\pi \ell}{2}\right)\mathbf{J}_k\right]f_m}\\
	&=\frac{\pi}{4(1-k)K(k)} \frac{m(1+q_k^m)}{1-q_k^m} \sum_{\ell=1}^\infty\left[1 -2 \frac{1-\cos\frac{\pi \ell}{2}}{q_k^{m/2}+q_k^{-m/2}}\right]\frac{1}{\ell}\DirProd{e_\ell}{f_m}.
	\end{align*}
	Since $f_m$ has a radius of convergence larger than one (Lemma \ref{thm:fmradius}) and $z_{k_1}( \frac{1}{4} + i T_k) = 1$ and $z_{k_1}(\pm iT_k) = \pm i$ (Lemma \ref{thm:mapping}(ii)), we have the special values $f_m(\pm i) = g_m( \pm iT_k)$ and $f_m(1) = g_m(\frac{1}{4} + i T_k)$, with 
	$g_m$ as defined in Lemma \ref{thm:fourierbasis}.
	Therefore we may evaluate
	\begin{align*}
	\sum_{\ell=1}^\infty \frac{1}{\ell} \DirProd{e_\ell}{f_m} &= \sum_{\ell=1}^\infty [z^\ell]f_m(z) = f_m(1) = (-1)^{m}\cosh(2\pi m T_k)-\cos\frac{\pi m}{2}, \\
	\sum_{\ell=1}^\infty \frac{\cos\frac{\pi \ell}{2}}{\ell} \DirProd{e_\ell}{f_m} &= \sum_{\ell=1}^\infty \frac{i^\ell+(-i)^\ell}{2}[z^\ell]f_m(z) = \frac{f_m(i)+f_m(-i)}{2} = \cos\frac{\pi m}{2}\left(\cosh(2\pi m T_k)-1\right).
	\end{align*}
	Combining the last two displays yields
	\[
	\sum_{\ell=1}^\infty\left[1 - 2\frac{1-\cos\frac{\pi \ell}{2}}{q_k^{m/2}+q_k^{-m/2}}\right]\frac{1}{\ell}\DirProd{e_\ell}{f_m} = (-1)^m (\cosh(2\pi mT_k)-1) = \frac{(-1)^m}{2}(q_k^{m/4}-q_k^{-m/4})^2,
	\]
	which leads to the stated result.
\end{proof}

\begin{proof}[Proof of Theorem \ref{thm:hypsecant}]
	For $\alpha\in\frac{\pi}{2}\Z+\frac{\pi}{4}$ and $p\geq 1$ we introduce the set $\mathcal{G}_p^{(\alpha)}$ of non-empty simple diagonal walks $w$ starting at $(p,0)$ and ending at an arbitrary location that have no intermediate visits to the origin, i.e. $w_i \neq (0,0)$ for $i \leq |w|-1$, and that have winding angle $\theta^w_{|w|-1/2}  \in (\alpha-\frac{\pi}{4},\alpha+\frac{\pi}{4})$. 
	See Figure \ref{fig:freedecomposition1} for an example with $p=1$ and $\alpha = -5\pi/4$.
	
		\begin{figure}[t]
			\centering
			\begin{subfigure}[c]{0.35\textwidth}
				\centering
				\includegraphics[height=.8\linewidth]{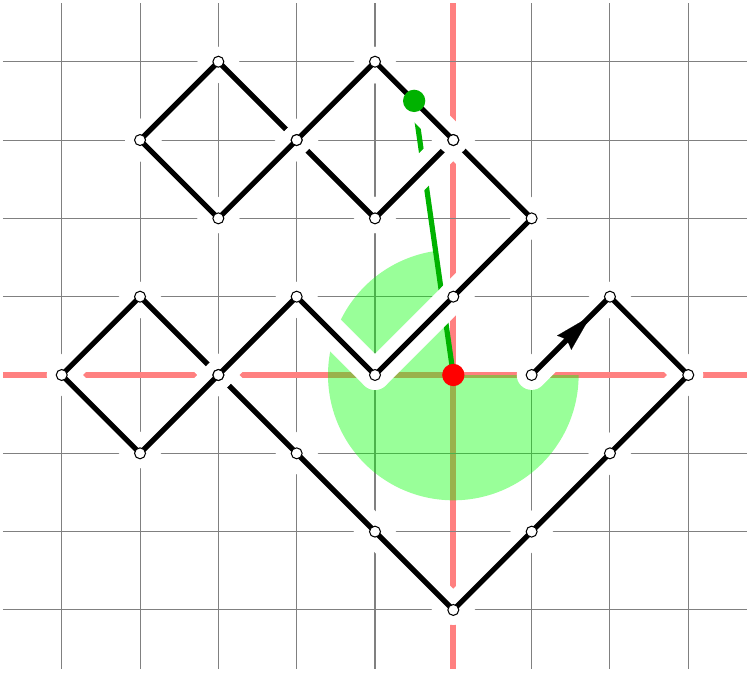}
				\caption{}\label{fig:freedecomposition1}
			\end{subfigure}%
			\begin{subfigure}[c]{0.35\textwidth}
				\centering
				\includegraphics[height=.8\linewidth]{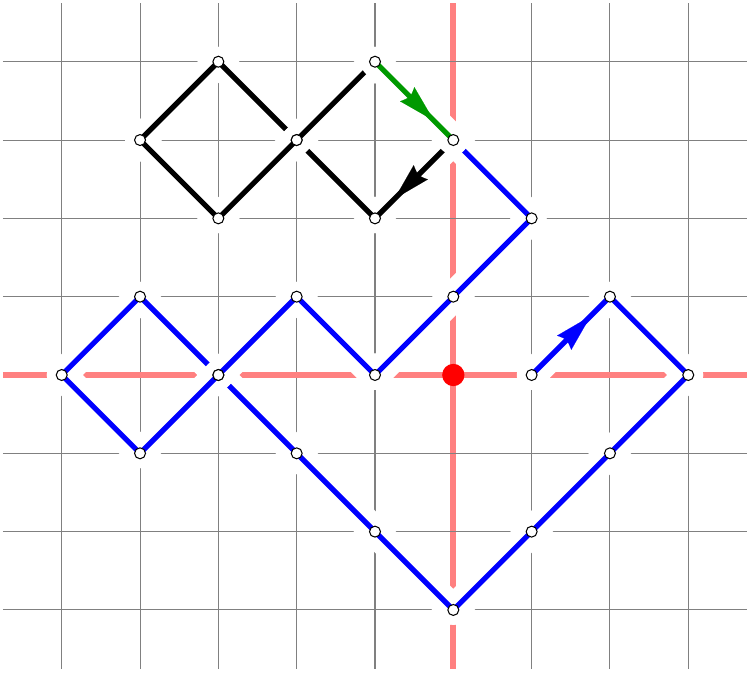}
				\caption{}\label{fig:freedecomposition2}
			\end{subfigure}%
			\caption{(a) An unconstrained walk $w$ with $\theta^w_{|w|-1/2} \in (-3\pi/2,-\pi)$. (b) Its decomposition into $w'\in \mathcal{W}_{3,1}^{(-3\pi/2)}$ (blue), $w''\in \mathcal{C}_3$ (black) and $s = (1,-1)$ (green). }\label{fig:freedecomposition}
		\end{figure}%
	It is not hard to see that such a walk can be uniquely encoded in a triple $(w',w'',s)$, where $w' \in \mathcal{W}_{n,p}^{(\alpha-\pi/4)}\cup \mathcal{W}_{n,p}^{(\alpha+\pi/4)}$ for some $n\geq 1$, $w''$ is empty or $w'' \in \mathcal{C}_n$ and $s \in \{ (1,1),(1,-1),(-1,-1),(-1,1)\}$ a single step (see Figure \ref{fig:freedecomposition2}). 
	Any such triple gives rise to a walk in $\mathcal{G}_p^{(\alpha)}$ unless $w''$ is empty, in which case $s$ is only allowed to take two of the four possible steps.
	Consequently the generating function $G_p^{(\alpha)}(t)$ of $\mathcal{G}_p^{(\alpha)}$ satisfies the relation
	\begin{align*}
	G_p^{(\alpha)}(t) &= 4t\sum_{n=1}^\infty \left(\frac12 + C_n(t)\right) \left( W_{n,p}^{(\alpha-\pi/4)}(t) + W_{n,p}^{(\alpha+\pi/4)}(t)\right).
	\end{align*}
	With the help of Theorem \ref{thm:mainresult}(i) and Lemma \ref{thm:arbitraryendpointsum} this evaluates to
	\begin{align}
	G_p^{(\alpha)}(t)&= k\sum_{n=1}^\infty \left(\frac12 + C_n(t)\right) \DirProd{e_n}{(\mathbf{Y}_k^{(\alpha-\pi/4)} + \mathbf{Y}_k^{(\alpha+\pi/4)})e_p} \nonumber\\
	&= k\sum_{m,n=1}^\infty  \left(\frac12 + C_n(t)\right)\DirProd{e_n}{f_m} \frac{\DirProd{e_p}{(\mathbf{Y}_k^{(\alpha-\pi/4)} + \mathbf{Y}_k^{(\alpha+\pi/4)})f_m}}{\|f_m\|^2_{\Dir}} \nonumber\\
	&= k\sum_{m,n=1}^\infty \left(\frac12 + C_n(t)\right)\DirProd{e_n}{f_m} \frac{8K(k)}{\pi} \frac{q_k^{m |\alpha|/\pi}(q_k^{m/4}+q_k^{-m/4})}{m^2(q_k^{-m}-q_k^{m})}  \DirProd{e_p}{f_m}\nonumber\\
	&= \frac{k}{1-k} \sum_{m=1}^\infty \frac{(-1)^m}{m}\frac{q_k^{m|\alpha|/\pi}}{q_k^{m/4}+q_k^{-m/4}}\DirProd{e_p}{f_m}. \label{eq:Galphaexpression} 
	\end{align}

	\begin{figure}
		\centering
		\includegraphics[width=.6\linewidth]{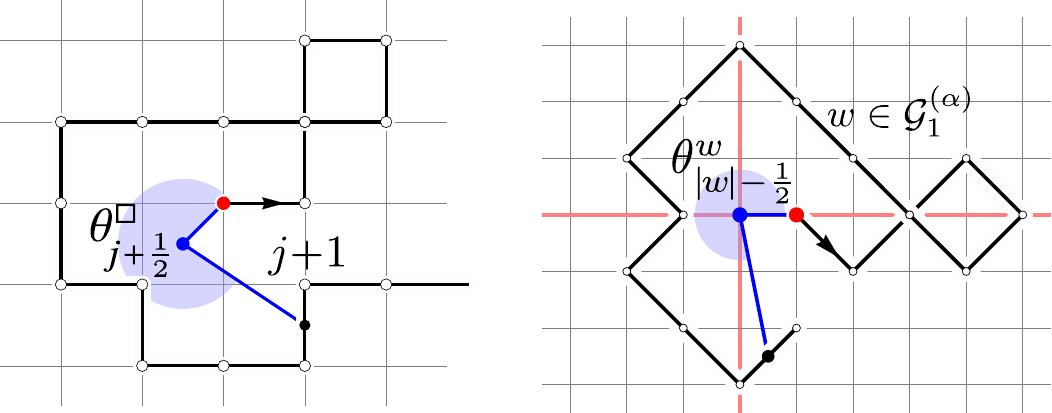}
		\caption{The probability that the winding angle $\theta^\square_{j+\frac{1}{2}}$ around $(-\frac12,-\frac12)$ of the simple random walk $(W_i)_i$ up to time $j+\frac{1}{2}$ takes values in $(\alpha-\frac{\pi}{4},\alpha+\frac{\pi}{4})$ is related to the counting of walks $w\in\mathcal{G}_1^{(\alpha)}$. }\label{fig:freewindingdiagonal}
	\end{figure}
	We turn to the distribution of the winding angle $\theta_{\zeta_k+1/2}^{\square}$ around $(-1/2,-1/2)$.
	By a suitable rotation and translation of the walk (see Figure \ref{fig:freewindingdiagonal}) we may relate the generating function $G_p^{(\alpha)}(t)$ at $p=1$ with the probability that $\theta_{\zeta_k+1/2}^{\square}$ takes values in $(\alpha - \frac{\pi}{4} ,\alpha + \frac{\pi}{4})$ with $\alpha\in\frac{\pi}{2}\Z+\frac{\pi}{4}$, 
	\[
	\prob\big[\theta^{\square}_{\zeta_k+1/2} \in ({\textstyle \alpha - \frac{\pi}{4} ,\alpha + \frac{\pi}{4}})\big] = \frac{1-k}{k} G_1^{(\alpha)}(k/4).
	\]
	Note that $\DirProd{e_1}{f_m}=0$ for $m$ even, while  
	\[
	\DirProd{e_1}{f_m} = f_m'(0) = (-1)^{(m+1)/2} 2\pi m \,v_{k_1}'(0) = \frac{(-1)^{(m+1)/2} \pi m}{2\sqrt{k_1}K(k_1)} =  \frac{(-1)^{(m+1)/2} \pi m}{k K(k)}
	\]
	for $m$ odd, where we used that $2\sqrt{k_1} K(k_1) = 2/(\sqrt{k_1}+1/\sqrt{k_1}) K(k) = k K(k)$ according to \eqref{eq:landenk} and \eqref{eq:kone}.
	With this we can evaluate \eqref{eq:Galphaexpression} at $p=1$ to
	\[
	G_1^{(\alpha)}(t) = \frac{1}{1-k} \frac{\pi}{K(k)} \sum_{\ell=0}^\infty (-1)^\ell \frac{q_k^{(2\ell+1)|\alpha|/\pi}}{q_k^{(2\ell+1)/4}+q_k^{-(2\ell+1)/4}}.
	\]
	In particular, for $\alpha \in \frac{\pi}{2}\Z$ and $\alpha \neq 0$ we have
	\begin{align*}
	\prob\big[ &\theta^\square_{\zeta_k+1/2} \in (\alpha -{\textstyle\frac{\pi}{2}}, \alpha +{\textstyle\frac{\pi}{2}})\big] = \frac{1-k}{k}\left( 
	G_1^{(\alpha-\frac{\pi}{4})}(k/4) + G_1^{(\alpha+\frac{\pi}{4})}(k/4)\right) \\
	&=  \frac{\pi}{k K(k)} \sum_{\ell=0}^\infty (-1)^\ell q_k^{(2\ell+1)|\alpha|/\pi}= \frac{\pi}{k K(k)} \frac{1}{q_k^{\alpha/\pi}+q_k^{-\alpha/\pi}}= \frac{\pi}{2k K(k)}\sech\left(4T_k \alpha\right).
	\end{align*}
	Note that the second equality does not hold for $\alpha = 0$, but using that these probabilities with $\alpha\in\pi\Z$ must sum to one, we deduce with the help of $\sum_{\alpha\in\pi\Z} \sech(4T_k \alpha) = 2K(k)/\pi$ \cite[Equation (19)]{bruckman_evaluation_1977} that
	\[
	\prob\big[\theta^\square_{\zeta_k+1/2} \in (-{\textstyle\frac{\pi}{2}}, {\textstyle\frac{\pi}{2}})\big] = 1-\sum_{\substack{\alpha\in\pi\Z\\ \alpha\neq 0}} \prob\big[\theta^\square_{\zeta_k+1/2} \in (\alpha-{\textstyle\frac{\pi}{2}},\alpha+ {\textstyle\frac{\pi}{2}})\big] = 1 - \frac{1}{k} + \frac{\pi}{2k K(k)}. 
	\]
	
	In the case of the winding angle around the origin and $\alpha\in\frac{\pi}{2}\Z+\frac{\pi}{4}$, we may identify
	\begin{equation}\label{eq:windingoriginexpr}
	\prob\big[\theta^{\bullet}_{\zeta_k+1/2} \in {\textstyle(\alpha - \frac{\pi}{2}, \alpha + \frac{\pi}{2})}\big] = \frac{1-k}{k} \left( \hat{G}^{(\alpha-\frac{\pi}{4})}(k/4) + \hat{G}^{(\alpha+\frac{\pi}{4})}(k/4) \right),
	\end{equation}
	where $\hat{G}^{(\alpha')}(t)$, $\alpha'\in \frac{\pi}{2}\Z$, is the generating function for the set $\hat{\mathcal{G}}^{(\alpha')}$ of non-empty simple diagonal walks $w$ starting at $(0,0)$ and ending at arbitrary location with no intermediate visits to the origin, i.e. $w_i\neq (0,0)$ for $1 \leq i \leq |w|-1$, and $\theta^w_{|w|-1/2} \in (\alpha'-\frac{\pi}{4},\alpha'+\frac{\pi}{4})$.
	We claim that \eqref{eq:windingoriginexpr} for $\alpha \in (\frac{\pi}{2}\Z+\frac{\pi}{4}) \setminus \{\pm \frac{\pi}{4}\}$ can be computed via the absolutely convergent sum
	\begin{equation}\label{eq:GGhatrelation}
	\hat{G}^{(\alpha-\frac{\pi}{4})}(t)+\hat{G}^{(\alpha+\frac{\pi}{4})}(t) = 4\sum_{p=1}^\infty (-1)^{p-1} G_{2p}^{(\alpha)}(t).
	\end{equation}
	Note that this sum fails to be absolutely convergence for $\alpha = \pm \pi/4$, since $G_{2p}^{(\pm\pi/4)}(t) = 2t + O(t^2)$ for any $p \geq 1$.
	By an argument similar to that used in Lemma \ref{thm:Falphacombinatorics}, the alternating sum counts walks $w'\in\mathcal{G}_{2}^{(\alpha)}$ starting at $(2,0)$ and having $w'_1 = (1,\pm 1)$.
	By moving the starting point of $w'$ to the origin we obtain a walk $w$ with winding angle $\theta^{w}_{|w|-1/2} = \theta^{w'}_{|w'|-1/2} \mp \pi/4$ if $w_1=(1,\pm1)$.
	Since $\theta^{w'}_{|w'|-1/2}\in (\alpha-\frac{\pi}{4},\alpha+\frac{\pi}{4})$, these walks correspond precisely to $\{ w\in \hat{\mathcal{G}}^{(\alpha-\frac{\pi}{4})} : w_1 = (1,1) \}$ and $\{ w\in \hat{\mathcal{G}}^{(\alpha+\frac{\pi}{4})} : w_1 = (1,-1) \}$, and include therefore precisely $1/4$ of the desired walks $\hat{\mathcal{G}}^{(\alpha-\frac{\pi}{4})} \cup\hat{\mathcal{G}}^{(\alpha+\frac{\pi}{4})}$.
	This verifies the expression \eqref{eq:GGhatrelation}.
	
	With the help of \eqref{eq:Galphaexpression} and \eqref{eq:falternating} (and absolute convergence) this evaluates to
	\begin{align}
	\hat{G}^{(\alpha-\frac{\pi}{4})}(t)+\hat{G}^{(\alpha+\frac{\pi}{4})}(t) &= \frac{4k}{1-k}\sum_{p,m=1}^\infty \frac{(-1)^{m+p-1}}{m} \frac{q_k^{m |\alpha|/\pi}}{q_k^{m/4}+q_k^{-m/4}} \DirProd{e_{2p}}{f_m} \nonumber\\
	& = \frac{k}{1-k} \frac{\pi}{K(k)} \sum_{n=1}^\infty (-1)^{n-1} (q_k^{-n/2}-q_k^{n/2})q_k^{2n|\alpha|/\pi}\nonumber\\
	& = \frac{k}{1-k} \frac{\pi}{K(k)} \left( \frac{1}{1+q^{\frac{1}{2}+\frac{2\alpha}{\pi}}} -\frac{1}{1+q^{-\frac{1}{2}+\frac{2\alpha}{\pi}}} \right),\label{eq:Ghatexpression}
	\end{align}
	where we dropped the absolute value $|\cdot|$ around $\alpha$ because the expression is invariant under $\alpha \to -\alpha$.
	Combining with \eqref{eq:windingoriginexpr} this leads to the claimed probability for $\alpha \neq \pm \frac{\pi}{4}$.
	
	Finally we show that \eqref{eq:Ghatexpression} is valid also for $\alpha = \pm \frac{\pi}{4}$.
	To this end it is sufficient to verify that the resulting expression for $\sum_{\alpha'\in\frac{\pi}{2}\Z} \hat{G}^{(\alpha')}(t)$ corresponds to the generating function of walks from the origin with no intermediate returns.
	The latter is related to the excursion generating function $F(t,b)$ of Proposition \ref{thm:excursiongenfun} by
	\[
	\sum_{\alpha'\in\frac{\pi}{2}\Z} \hat{G}^{(\alpha')}(t) = \frac{k}{1-k}(1-F(t,0)) = \frac{k}{1-k}\frac{\pi}{2 K(k)}.
	\]
	The sum of \eqref{eq:Ghatexpression} over $\alpha \in \frac{\pi}{2}\Z + \frac{\pi}{4}$, thus double counting each $\hat{G}^{(\alpha')}(t)$ for $\alpha'\in\frac{\pi}{2}\Z$, indeed gives twice this expression.
\end{proof}

\begin{proof}[Proof of Corollary \ref{thm:charfuns}]
First we note that for $\alpha\in\pi\Z$,
\begin{align*}
\prob\big[ \theta^\square_{\zeta_k-1/2} \in {\textstyle(\alpha - \frac{\pi}{2}, \alpha + \frac{\pi}{2})}\big] &= (1-k)\one_{\{\alpha=0\}} + k\, \prob\big[  \theta^\square_{\zeta_k+1/2} \in {\textstyle(\alpha - \frac{\pi}{2}, \alpha + \frac{\pi}{2})}\big]\\
&=\frac{\pi}{2K(k)}\sech\left(\alpha \frac{K(k')}{K(k)}\right).
\end{align*}
The desired characteristic functions then follow with the help of \cite[16.23.2 \& 16.23.3]{abramowitz_handbook_1964}, 
\begin{align*}
\expec \exp\left(ib\{\theta^\square_{\zeta_k-1/2}\}_{\pi\Z}\right) &= \frac{\pi}{K(k)}\sum_{n=-\infty}^\infty \frac{e^{ib n\pi}}{q_k^{n}+q_k^{-n}} = \dn(K(k)b,k),\\
\expec \exp\left(ib\{\theta^\square_{\zeta_k+1/2}\}_{\pi\Z+\frac{\pi}{2}}\right) &= \frac{\pi}{k K(k)}\sum_{n=-\infty}^\infty \frac{e^{ib (n+1/2)\pi}}{q_k^{n+1/2}+q_k^{-n-1/2}} = \cn(K(k)b,k).
\end{align*}
\end{proof}

\section{Winding angle of loops}\label{sec:loops}

Another, rather interesting application of Theorem \ref{thm:mainresult} is the counting of \emph{loops} on $\Z^2$.
To be precise, for integer $n\neq 0$, let the set $\mathcal{L}_n$ of \emph{rooted loops of index $n$} be the set of simple diagonal walks $w$ on $\Z^2\setminus\{(0,0)\}$ that start and end at the same (arbitrary) point and have winding angle $\theta^w = 2\pi n$.
The set $\mathcal{L}_n=\mathcal{L}_n^{\text{even}} \cup \mathcal{L}_n^{\text{odd}}$ naturally partitions into the \emph{even loops} $\mathcal{L}_n^{\text{even}}$ supported on $\{ (x,y)\in \Z^2 : x+y\text{ even, }(x,y)\neq (0,0)\}$ and the \emph{odd loops} $\mathcal{L}_n^{\text{odd}}$ on $\{ (x,y)\in \Z^2 : x+y\text{ odd}\}$.

\begin{theorem}\label{thm:loopgenfun}
The (``inverse-size biased'') generating functions for $\mathcal{L}_n^{\text{even}}$ and $\mathcal{L}_n^{\text{odd}}$ are given by
\begin{align*}
L^{\text{even}}_n(t)\coloneqq\sum_{w\in\mathcal{L}^{\text{even}}_n} \frac{t^{|w|}}{|w|} &= 
\frac{1}{|n|} \tr_\Dir \mathbf{P}^{\text{even}}\mathbf{J}_k^{(2\pi |n|,-\infty)} = \frac{1}{|n|} \frac{q_k^{4|n|}}{1-q_k^{4|n|}},\\
L^{\text{odd}}_n(t)\coloneqq\sum_{w\in\mathcal{L}^{\text{odd}}_n} \frac{t^{|w|}}{|w|} &= 
\frac{1}{|n|} \tr_\Dir \mathbf{P}^{\text{odd}}\mathbf{J}_k^{(2\pi |n|,-\infty)} = \frac{1}{|n|} \frac{q_k^{2|n|}}{1-q_k^{4|n|}},
\end{align*}
where $\mathbf{P}^{\text{even}}$ (respectively $\mathbf{P}^{\text{odd}}$) is the projection operator onto the even (respectively odd) functions in $\Dir$. 
\end{theorem}

\begin{proof}
Without loss of generality we will take $n>0$, since the case of negative $n$ then follows from symmetry.
Let us consider the subset 
\[
\hat{\mathcal{L}}_n \coloneqq \{ w\in \mathcal{L}_n: w_0 \in\{ (x,0): x>0\}\text{ and }\theta^w_i < \theta^w\text{ for }0\leq i<|w|\}
\] 
of rooted loops of index $n>0$ that start on the positive $x$-axis and that attain the winding angle $2n\pi$ only at the very end.
Clearly $\hat{\mathcal{L}}_n = \bigcup_{p\geq 1} \mathcal{W}_{p,p}^{(2\pi n,(-\infty,2\pi n))}$ in the notation of Theorem \ref{thm:mainresult}. 
Similarly $\hat{\mathcal{L}}_n^{\text{even/odd}} \coloneqq \hat{\mathcal{L}}_n \cap \mathcal{L}^{\text{even/odd}}=\bigcup_{p\text{ even/odd}} \mathcal{W}_{p,p}^{(2\pi n,(-\infty,2\pi n))}$.
The generating functions of $\hat{\mathcal{L}}_n^{\text{even}}$ and $\hat{\mathcal{L}}_n^{\text{odd}}$ are therefore given by
\[
\sum_{w\in\hat{\mathcal{L}}^{\text{even/odd}}_n} t^{|w|} = \sum_{p\text{ even/odd} }\frac{1}{p} \DirProd{e_{p}}{\mathbf{J}_k^{(2\pi n,-\infty)}e_{p}} = \tr_{\Dir} \mathbf{P}^{\text{even/odd}} \mathbf{J}_k^{(2n\pi,-\infty)} = \sum_{m\text{ even/odd}} q_k^{2 m n},
\]
where we used that (according to Theorem \ref{thm:mainresult}(ii)) $\mathbf{J}_k^{(2n\pi,-\infty)}$ has eigenvalues $(q_k^{2nm})_{m\geq 1}$ and that the even and odd subspaces of $\Dir$ are spanned by the even respectively odd elements of the basis $(f_m)_{m\geq 1}$.
Hence,
\begin{equation}\label{eq:Lhatgenfun}
\sum_{w\in\hat{\mathcal{L}}^{\text{even}}_n} t^{|w|} = \frac{q_k^{4n}}{1-q_k^{4n}} \qquad\text{and}\qquad \sum_{w\in\hat{\mathcal{L}}^{\text{odd}}_n} t^{|w|} = \frac{q_k^{2n}}{1-q_k^{4n}}.
\end{equation}

Now suppose we take a general loop $w\in \mathcal{L}_n$.
We denote by $w^{(j)}\in\mathcal{L}_n$, $1\leq j \leq |w|$, the cyclic permutation of $w$ given by the walk $w^{(j)} \coloneqq (w_j, w_{j+1}, \ldots , w_{|w|}, w_1, \ldots, w_{j})$.
We claim that among these $|w|$ cyclic permutations are exactly $n$ elements of $\hat{\mathcal{L}}_n$.

To see this, let $(i_\ell)_{\ell=1}^m$ be the sequence of increasing times (in $\{1,2,\ldots,|w|\}$) at which $w$ intersects the positive $x$-axis. 
Then $w^{(i_\ell)}$, $1\leq \ell \leq m$, are potential candidates for walks in $\hat{\mathcal{L}}_n$, since they start on the positive $x$-axis. 
For each such walk $w'=w^{(i_\ell)}$ we may consider the winding angle sequence $(\theta^{w'}_i)_{i=0}^{|w'|}$ as well as the subsequence $(\alpha^{(\ell)}_j)_{j=0}^m$ of $(\theta^{w'}_i)_{i=0}^{|w'|}$ containing just those angles in $2\pi\Z$.
Then $(\alpha^{(\ell)}_j)_{j=0}^m$ describes a walk on $2\pi\Z$ from $0$ to $2\pi n$ with steps in $\{-2\pi,0,2\pi\}$, and $w^{(i_\ell)}\in\hat{\mathcal{L}}_n$ precisely when this walk stays strictly below $2\pi n$ until the very end.
Since the walks $(\alpha^{(\ell)}_j)_{j=0}^m$, $1\leq \ell\leq m$, correspond precisely to an equivalence class under cyclic permutation of the increments, a well-known cycle lemma (or ballot theorem) tells us that the latter condition, hence $w^{(i_\ell)}\in\hat{\mathcal{L}}_n$,  is satisfied for exactly $n$ values of $\ell$.

The claim implies that 
\[
\sum_{w\in \mathcal{L}_n} \frac{t^{|w|}}{|w|} = \frac{1}{n}\sum_{w\in \mathcal{L}_n} \frac{t^{|w|}}{|w|}\sum_{j=1}^{|w|} \one_{\{w^{(j)}\in \hat{\mathcal{L}}_n \}}  = \frac{1}{n} \sum_{w\in\hat{\mathcal{L}}_n} \frac{t^{|w|}}{|w|} \sum_{j=1}^{|w|} \one_{\{w^{(j)}\in\mathcal{L}_n\}} = \frac{1}{n} \sum_{w\in\hat{\mathcal{L}}_n} t^{|w|}.
\]
This identity restricted to the even and odd subspaces together with \eqref{eq:Lhatgenfun} then gives the desired expressions. 
\end{proof}

For the rest of this section we switch to simple rectilinear rooted loops on $\Z^2$.
For such a loop $w$ we let the \emph{index} $I^w:\R^2\to\Z$ be defined by setting $I^w(z) = 0$ when $z$ lies on the trajectory of $w$ and otherwise $2\pi I^w(z)$ is the winding angle of $w$ around the point $z$.
By a suitable affine transformation we may now equally think of $\mathcal{L}_n^{\text{odd}}$, $n\neq 0$, as the set of simple rooted loops $w$ on $\Z^2$ with index $I^w(z)=n$ with respect to some fixed off-lattice point, say, $z=(1/2,1/2)$.
Similarly, $\mathcal{L}_n^{\text{even}}$, $n\neq 0$, is in $1$-to-$1$ correspondence with such loops that have index $n$ with respect to a fixed lattice point, say, the origin.
The following probabilistic result takes advantage of this point of view.

\begin{corollary}
For $\ell\geq 1$, let $W=(W_i)_{i=0}^{2\ell}$ be a simple random walk on $\Z^2$ conditioned to return to the origin after $2\ell$ steps.
For $n\neq 0$, let $C_n$ be the set of connected components of $(I^W)^{-1}(n)$.
Then
\begin{align*}
\expec\left[ \sum_{c\in C_n} |c|\right] &= \frac{4^{2\ell}}{\binom{2\ell}{\ell}^2}\frac{2\ell}{n} [k^{2\ell}]\frac{q_k^{2n}}{1-q_k^{4n}} \sim \frac{\ell}{2\pi n^2},\\
\expec\left[ \sum_{c\in C_n} (|\partial c|-2)\right] &= \frac{4^{2\ell}}{\binom{2\ell}{\ell}^2}\frac{4\ell}{n} [k^{2\ell}]\frac{q_k^{2n}}{1+q_k^{2n}} \sim \frac{2\pi^3 \ell}{\log^2 \ell},
\end{align*}
where $|c|$ is the area of $c\in C_n$ and $|\partial c|$ the boundary length of $c$ (see Figure \ref{fig:loopcomponents}). 
The asymptotic formulas hold as $\ell\to\infty$ and $n$ fixed.
\end{corollary}
\begin{proof}
The left and right sides of the identities are all seen to be invariant under $n\to-n$, so we may restrict to the case $n>0$.
A simple counting exercise shows that for any $c\in C_n$ the area and boundary length of $c$ can be expressed in terms of the cardinalities $|c \cap \Z^2|$ and $|c \cap (\Z+1/2)^2|$ as
\[ |c| = |c \cap (\Z+1/2)^2| \qquad\text{and}\qquad |\partial c| = 2|c \cap (\Z+1/2)^2| - 2|c \cap \Z^2| + 2. \]
Hence, we have that  
\begin{align*}
\sum_{c\in C_n} |c| = \sum_{z\in(\Z+1/2)^2} \one_{\{I^W(z)=n\}} &= \sum_w \one_{\{I^w(1/2,1/2) = n\}},\\
\sum_{c\in C_n} (|c|+1-|\partial c|/2) = \sum_{z\in\Z^2} \one_{\{I^W(z)=n\}} &= \sum_{w} \one_{\{I^w(0,0) = n\}},
\end{align*}
where the last sum on both lines is over all possible translations $w$ of $W$ by a vector in $\Z^2$.
The collection of translations that have index $I^w(1/2,1/2)=n$ (respectively $I^w(0,0)=n$) indexed by all possible $W$ precisely determines a partition of the loops of length $2\ell$ in $\mathcal{L}_n^{\text{odd}}$ (respectively $\mathcal{L}_n^{\text{even}}$).
Since the probability of any particular walk $W$ is $\binom{2\ell}{\ell}^{-2}$ we find with the help of Theorem \ref{thm:loopgenfun} that
\begin{align*}
\expec\left[ \sum_{c\in C_n} |c|\right] = \frac{1}{\binom{2\ell}{\ell}^2} [t^{2\ell}] \sum_{w\in\mathcal{L}_n^{\text{odd}}} t^{|w|} = \frac{1}{\binom{2\ell}{\ell}^2} 2\ell\, [t^{2\ell}] L_n^{\text{odd}}(t) = \frac{4^{2\ell}}{\binom{2\ell}{\ell}^2}\frac{2\ell}{n} [k^{2\ell}]\frac{q_k^{2n}}{1-q_k^{4n}},\\
\expec\left[ \sum_{c\in C_n} (|c|+1-|\partial c|/2)\right] = \frac{1}{\binom{2\ell}{\ell}^2} [t^{2\ell}] \sum_{w\in\mathcal{L}_n^{\text{even}}} t^{|w|} = \frac{1}{\binom{2\ell}{\ell}^2} 2\ell\, [t^{2\ell}] L_n^{\text{even}}(t) = \frac{4^{2\ell}}{\binom{2\ell}{\ell}^2}\frac{2\ell}{n} [k^{2\ell}]\frac{q_k^{4n}}{1-q_k^{4n}}.
\end{align*}
The first line and the difference between the two lines agree with the claimed formulas.

Since $k\mapsto q_k$ is analytic in $\C \setminus \{ k\in\R : |k| \geq 1\}$ and $|q_k| < 1$ (see \eqref{eq:nomebound} in Appendix \ref{sec:ellipticappendix}), it follows that both $k\mapsto q_k^{2n}/(1-q_k^{4n})$ and $k\mapsto q_k^{2n}/(1+q_k^{2n})$ are $\Delta$-analytic \cite{flajolet_singularity_1990} with singularities at $k=\pm1$.
Since $q_k = \exp\left(\frac{\pi^2}{\log q_{k'}}\right) = 1 + \frac{\pi^2}{\log(1-k^2)} + O( \log^{-2}(1-k^2) )$ as $k\to\pm1$, we find 
\[
\frac{q_k^{2n}}{1-q_k^{4n}} = -\frac{1}{4n} \frac{\log(1-k^2)}{\pi^2} + O(1) \qquad\text{and}\qquad \frac{q_k^{2n}}{1+q_k^{2n}} = \frac{1}{2}+\frac{n}{2} \frac{\pi^2}{\log(1-k^2)} + O(\log^{-2}(1-k^2)).
\]
Standard transfer theorems (see \cite{flajolet_singularity_1990}) then imply
\[
[k^{2\ell}] \frac{q_k^{2n}}{1-q_k^{4n}} \sim \frac{1}{4n \pi^2 \ell} \qquad\text{and}\qquad [k^{2\ell}] \frac{q_k^{2n}}{1+q_k^{2n}} \sim \frac{\pi^2 n}{2\ell\,\log^2\ell}\qquad\text{as }\ell\to\infty.
\]
Together with $4^{2\ell}\binom{2\ell}{\ell}^{-2} \sim \pi \ell$ these give rise to the stated asymptotics. 
\end{proof}

\appendix

\section{Elliptic functions}\label{sec:ellipticappendix}

This work depends heavily on elliptic functions and their properties. 
Given that the required properties are scattered in the literature and that different sources often use inconsistent notation, we provide here a summary of the definitions and required material.
We refer to \cite{abramowitz_handbook_1964,akhiezer_elements_1990,borwein_pi_1987} for background and proofs.
 
The \emph{complete elliptic integral of the first kind} with \emph{elliptic modulus} $k$ in the open unit disc $\disc$ is defined by
\begin{equation}
K(k) = \int_{0}^1 \frac{\rmd y}{\sqrt{(1-y^2)(1-k^2y^2)}},
\end{equation}
while the \emph{complete elliptic integral of the second kind} is
\begin{equation}
E(k) = \int_{0}^1 \sqrt{\frac{1-k^2 y^2}{1-y^2}}\rmd y.
\end{equation}
Both are analytic in $k \in \disc$ with series expansions around the origin given by \cite[17.3.11-12]{abramowitz_handbook_1964}
\begin{equation}\label{eq:Kexpansion}
K(k) = \frac{\pi}{2} \sum_{n=0}^\infty \binom{2n}{n}^2 \left(\frac{k}{4}\right)^{2n}, \qquad E(k) = \frac{\pi}{2} \sum_{n=0}^\infty \frac{1}{1-2n}\binom{2n}{n}^2 \left(\frac{k}{4}\right)^{2n}.
\end{equation}
For $k\in\disc$, the \emph{complementary modulus} $k'\in\disc$ and the associated elliptic integrals $K'$ and $E'$ are defined as
\begin{equation}\label{eq:kprime}
k' = \sqrt{1-k^2},\qquad K'(k) = K(k'), \qquad E'(k) = E(k'),
\end{equation}
where the primes should not be confused with the derivatives.
They satisfy Legendre's relation \cite[17.3.13]{abramowitz_handbook_1964}
\begin{equation}\label{eq:legendre}
E(k) K'(k) + E'(k)K(k) - K(k) K'(k) = \tfrac{\pi}{2}.
\end{equation}
The \emph{elliptic nome} $q_k \in \disc$ and the \emph{quarter-period ratio} $T_k \in (0,\infty) + i\R$ are 
\begin{equation}\label{eq:nomedef}
q_k = e^{-\pi K'(k) / K(k)} = e^{-4\pi T_k}, \qquad T_k = \frac{K'(k)}{4 K(k)}.
\end{equation}
The functions $K(k)$, $q_k$ and $T_k$ can be analytically continued to the double-slit plane $k\in\C \setminus \{ z\in \R : z^2 \geq 1\}$, where they satisfy \cite[Section 2]{walker_analyticity_2003}
\begin{equation}\label{eq:nomebound}
|q_k| < 1,\qquad \re(T_k) > 0, \qquad \text K(k) \neq 0.
\end{equation}
Note that \eqref{eq:nomedef} and \eqref{eq:kprime} imply that $q_{1/\sqrt{2}} = e^{-\pi}$, therefore
\begin{equation}\label{eq:nomebound2}
|q_k| \leq e^{-\pi} < \tfrac{1}{20}\quad \text{for } |k| \leq \tfrac{1}{\sqrt{2}}.
\end{equation}

For $q\in\disc$ and $z\in\C$, the \emph{Jacobi theta functions} are defined via the absolutely convergent series
\begin{align}
\theta_1(z,q) &= 2 \sum_{n=0}^\infty (-1)^n q^{(n+1/2)^2} \sin((2n+1)z), \label{eq:theta1def}\\
\theta_2(z,q) &= 2 \sum_{n=0}^\infty q^{(n+1/2)^2} \cos((2n+1)z), \\
\theta_3(z,q) &= 1+2 \sum_{n=1}^\infty q^{n^2} \cos(2nz), \\
\theta_4(z,q) &= 1+2 \sum_{n=1}^\infty (-1)^n q^{n^2} \cos(2nz).
\end{align}
By \eqref{eq:nomebound} the theta functions $k\mapsto\theta_i(z,q_k)$ evaluated at $q=q_k$ and $z$ fixed are analytic in the slit plane  $k\in\C \setminus \{ z\in \R : z^2 \geq 1\}$. 
The squared elliptic modulus and the corresponding complete elliptic integral can be expressed in terms of $q=q_k$ via \cite[Theorem 2.1 \& (2.1.13)]{borwein_pi_1987}
\begin{align}
k^2 &= \frac{\theta_2(0,q)^4}{\theta_3(0,q)^4} = 16 q - 128 q^2 + 704 q^4 + \cdots,\\
K(k) &= \frac{\pi}{2} \theta_3(0,q)^2 = \frac{\pi }{2}+2 \pi  q+2 \pi  q^2+2 \pi  q^4+ \cdots. \label{eq:Kfromq}
\end{align}
By reversion of the first of these series one obtains the series expansion of $q_k$ around $k=0$ (with radius of convergence equal to $1$),
\begin{equation}\label{eq:nomeexpansion}
q_k = \left(\frac{k}{4}\right)^2 + 8 \left(\frac{k}{4}\right)^4 + 84 \left(\frac{k}{4}\right)^6 + 992 \left(\frac{k}{4}\right)^8 + \cdots,
\end{equation}
which can also be understood on the level of formal power series as the reversion of the series $k^2 = \theta_2(0,q)^4/\theta_3(0,q)^4 \in \R[\![q]\!]$.

The \emph{Jacobi elliptic functions} $\sn(\cdot, k)$, $\cn(\cdot, k)$ and $\dn(\cdot,k)$ are defined in terms of the Jacobi theta functions $\theta_i(z,q)$ with argument $z = \frac{\pi}{2 K(k)} u$ and nome $q=q_k$ via
\begin{equation}
\sn(u,k) = \frac{\theta_3(0,q_k)}{\theta_2(0,q_k)} \frac{\theta_1(z,q_k)}{\theta_4(z,q_k)}, \quad
\cn(u,k) = \frac{\theta_4(0,q_k)}{\theta_2(0,q_k)} \frac{\theta_2(z,q_k)}{\theta_4(z,q_k)}, \quad 
\dn(u,k) = \frac{\theta_4(0,q_k)}{\theta_3(0,q_k)} \frac{\theta_3(z,q_k)}{\theta_4(z,q_k)}.
\end{equation}
Alternatively, they are the unique biperiodic meromorphic functions periodic under $u\to u+4K(k)$ and $u\to u+4 i K'(k)$ satisfying \cite[\S 24]{akhiezer_elements_1990}
\begin{align}\label{eq:ellipticintegrals}
u = \int_0^{\sn(u,k)} \!\!\!\!\frac{\rmd y}{\sqrt{(1-y^2)(1-k^2 y^2)}} = \int_{\cn(u,k)}^1 \frac{\rmd y}{\sqrt{(1-y^2)(k'^2+k^2 y^2)}} = \int_{\dn(u,k)}^1 \frac{\rmd y}{\sqrt{(1-y^2)(y^2 - k'^2)}},
\end{align}
in a neighbourhood of $u = 0$. 
They satisfy the quadratic relations \cite[16.9.1]{abramowitz_handbook_1964}
\begin{equation}\label{eq:quadraticrelations}
\sn^2(u,k) + \cn^2(u,k) = k^2 \sn^2(u,k) + \dn^2(u,k) = 1
\end{equation}
and various argument shift relations \cite[16.8]{abramowitz_handbook_1964}, e.g.
\begin{equation}\label{eq:snshift}
\sn(u +i K'(k), k) = \frac{1}{k \sn(u,k)}, \qquad\sn(u + 2i K'(k), k) =  \sn(u,k).
\end{equation}
Series representations of the Jacobi elliptic functions can be obtained from \eqref{eq:ellipticintegrals} by expanding the integrand in $y$, followed by integration and series reversion, \cite[16.22]{abramowitz_handbook_1964}
\begin{align*}
\sn(u,k) & = u - \frac{1}{6} (1+ k^2) u^3 + \frac{1}{120} (1+14 k^2 + k^4) u^5 + \cdots, \\
\cn(u,k) & = 1 - \frac{1}{2} u^2 + \frac{1}{24} (1+4k^2) u^4 - \frac{1}{720}(1+44 k^2+16 k^4) u^6 + \cdots, \\
\dn(u,k) & = 1 - \frac{1}{2} k^2 u^2 + \frac{1}{24} (4 k^2+ k^4) u^4 - \frac{1}{720}( 16 k^2 + 44 k^4 + k^6) u^6 + \cdots.
\end{align*}

The \emph{Jacobi zeta function} is given in terms of the Jacobi theta functions with argument $z=\frac{\pi}{2K(k)}u$ by \cite[16.34.1 \& 16.34.4]{abramowitz_handbook_1964}
\begin{equation}\label{eq:zetadef}
Z(u,k) = \frac{\pi}{2 K(k)} \frac{\theta_4'(z,k)}{\theta_4(z,k)} = \frac{\pi}{2 K(k)} \frac{\theta_1'(z,k)}{\theta_1(z,k)} - \frac{\cn(u,k)\dn(u,k)}{\sn(u,k)},  
\end{equation}
where $\theta_i'(z,k) = \frac{\partial}{\partial z}\theta_i(z,k)$. It is periodic in $u$ with period $2K(k)$ \cite[17.4.30]{abramowitz_handbook_1964} and satisfies the addition formula \cite[17.4.35]{abramowitz_handbook_1964}
\begin{equation}\label{eq:zetaaddition}
Z(u+v,k) = Z(u,k) + Z(v,k) - k^2\sn(u,k)\sn(v,k)\sn(u+v,k).
\end{equation}

We often encounter elliptic functions in which the modulus $k \in (0,1)$ is replaced by its \emph{descending Landen transformation} $k_1\in(0,1)$ defined as
\begin{equation}\label{eq:kone}
k_1 = \frac{1-k'}{1+k'} = \frac{1-\sqrt{1-k^2}}{1+\sqrt{1-k^2}}, \qquad k = \frac{2}{\sqrt{k_1}+1/\sqrt{k_1}}.
\end{equation}
The elliptic integrals, quarter-period ratio and nome transform as \cite[Theorem 1.2]{borwein_pi_1987}
\begin{equation}\label{eq:landenk}
K(k) = (1+k_1) K(k_1), \qquad K'(k) = \frac{1}{2}(1+k_1)K'(k_1), \qquad T_{k_1} = 2 T_k, \qquad q_{k_1} = q_k^2.
\end{equation}
The Jacobi elliptic and theta functions also transform relatively nicely under Landen transformations, e.g. \cite[16.12.2]{abramowitz_handbook_1964}
\begin{align}\label{eq:landensn}
\sn(u,k) &= \frac{(1+k_1)\sn(u/(1+k_1),k_1)}{1+ k_1 \sn^2(u/(1+k_1),k_1)}. 
\end{align}
See \cite[Section 16.12]{abramowitz_handbook_1964}, \cite[\S 38]{akhiezer_elements_1990} \cite[Section 2.7]{borwein_pi_1987} for more such identities.


\begin{thebibliography}{10}
	
	\bibitem{abramowitz_handbook_1964}
	{\sc Abramowitz, M., and Stegun, I.~A.}
	\newblock {\em Handbook of {Mathematical} {Functions}: {With} {Formulas},
		{Graphs}, and {Mathematical} {Tables}}.
	\newblock Courier Corporation, 1964.
	
	\bibitem{akhiezer_elements_1990}
	{\sc Akhiezer, N.~I.}
	\newblock {\em Elements of the {Theory} of {Elliptic} {Functions}}.
	\newblock Transl. of {Math}. {Monographs} vol. 79. American Mathematical Soc., 1990.
	
	\bibitem{arcozzi_dirichlet_2011}
	{\sc Arcozzi, N., Rochberg, R., Sawyer, E., and Wick, B.}
	\newblock The {Dirichlet} space: {A} {Survey}.
	\newblock {\em New York Journal of Mathematics 17a\/} (2011), 45--86.
	\newblock arXiv:1008.5342.
	
	\bibitem{bernardi_bijective_2007}
	{\sc Bernardi, O.}
	\newblock Bijective counting of {Kreweras} walks and loopless triangulations.
	\newblock {\em Journal of Combinatorial Theory, Series A 114}, 5 (2007),
	931--956.
	
	\bibitem{bernardi_counting_2017}
	{\sc Bernardi, O., Bousquet-M{\'e}lou, M., and Raschel, K.}
	\newblock Counting quadrant walks via {Tutte}'s invariant method.
	\newblock arXiv:1708.08215.
	
	\bibitem{borot_nesting_2016}
	{\sc Borot, G., Bouttier, J., and Duplantier, B.}
	\newblock Nesting statistics in the {$O(n)$} loop model on random planar maps.
	\newblock arXiv:1605.02239.
	
	\bibitem{borot_loop_2012}
	{\sc Borot, G., Bouttier, J., and Guitter, E.}
	\newblock Loop models on random maps via nested loops: the case of domain
	symmetry breaking and application to the {Potts} model.
	\newblock {\em J. Phys. A: Math. Theor. 45}, 49 (2012), 494017.
	
	\bibitem{borot_recursive_2012}
	{\sc Borot, G., Bouttier, J., and Guitter, E.}
	\newblock A recursive approach to the {$O(n)$} model on random maps via nested
	loops.
	\newblock {\em J. Phys. A: Math. Theor. 45}, 4 (2012), 045002.
	
	\bibitem{borwein_pi_1987}
	{\sc Borwein, J., and Borwein, P.}
	\newblock {\em Pi and the AGM: a study in analytic number theory and computational complexity}.
	\newblock Wiley, 1987.
	
	\bibitem{bostan_complete_2010}
	{\sc Bostan, A., and Kauers, M.}
	\newblock The complete generating function for {Gessel} walks is algebraic.
	\newblock {\em Proc. Amer. Math. Soc. 138}, 9 (2010), 3063--3078.
	
	\bibitem{bostan_human_2017}
	{\sc Bostan, A., Kurkova, I., and Raschel, K.}
	\newblock A human proof of {Gessel}’s lattice path conjecture.
	\newblock {\em Trans. Amer. Math. Soc. 369}, 2 (2017), 1365--1393.
	
	\bibitem{bousquet-melou_walks_2001}
	{\sc Bousquet-M{\'e}lou, M.}
	\newblock Walks on the {Slit} {Plane}: {Other} {Approaches}.
	\newblock {\em Advances in Applied Mathematics 27}, 2 (2001), 243--288.
	
	\bibitem{bousquet-melou_elementary_2016}
	{\sc Bousquet-M{\'e}lou, M.}
	\newblock An elementary solution of {Gessel}'s walks in the quadrant.
	\newblock {\em Advances in Mathematics 303\/} (2016), 1171--1189.
	
	\bibitem{bousquet-melou_square_2016}
	{\sc Bousquet-M{\'e}lou, M.}
	\newblock Square lattice walks avoiding a quadrant.
	\newblock {\em Journal of Combinatorial Theory, Series A 144\/} (2016), 37--79.
	
	\bibitem{bousquet-melou_walks_2010}
	{\sc Bousquet-M{\'e}lou, M., and Mishna, M.}
	\newblock Walks with small steps in the quarter plane.
	\newblock In {\em Algorithmic {Probability} and {Combinatorics}}, no.~520 in
	Contemp. {Math}. Amer. Math. Soc., Providence, 2010, pp.~1--39.
	
	\bibitem{bousquet-melou_walks_2002}
	{\sc Bousquet-M{\'e}lou, M., and Schaeffer, G.}
	\newblock Walks on the slit plane.
	\newblock {\em Probab Theory Relat Fields 124}, 3 (2002), 305--344.
	
	\bibitem{bruckman_evaluation_1977}
	{\sc Bruckman, P.~S.}
	\newblock On the {Evaluation} of {Certain} {Infinite} {Series} by {Elliptic}
	{Functions}.
	\newblock {\em The Fibonacci Quarterly 15.4\/} (1977), 293--310.
	
	\bibitem{budd_peeling_2016}
	{\sc Budd, T.}
	\newblock The {Peeling} {Process} of {Infinite} {Boltzmann} {Planar} {Maps}.
	\newblock {\em The Electronic Journal of Combinatorics 23}, 1 (2016), P1.28.
	
	\bibitem{budd_geometry_2017}
	{\sc Budd, T., and Curien, N.}
	\newblock Geometry of infinite planar maps with high degrees.
	\newblock {\em Electron. J. Probab. 22\/} (2017).
	
	\bibitem{budd_peeling_2018}
	{\sc Budd, T.}
	\newblock The peeling process on random planar maps	coupled to an {$O(n)$} loop model (with an appendix by Linxiao Chen).
	\newblock arXiv:1809.02012.
	
	\bibitem{belisle_windings_1989}
	{\sc Bélisle, C.}
	\newblock Windings of {Random} {Walks}.
	\newblock {\em The Annals of Probability 17}, 4 (1989), 1377--1402.
	
	\bibitem{belisle_winding_1991}
	{\sc Bélisle, C., and Faraway, J.}
	\newblock Winding angle and maximum winding angle of the two-dimensional random
	walk.
	\newblock {\em Journal of Applied Probability 28}, 4 (1991), 717--726.
	
	\bibitem{camia_brownian_2016}
	{\sc Camia, F.}
	\newblock Brownian {Loops} and {Conformal} {Fields}.
	\newblock In {\em Advances in {Disordered} {Systems}, {Random} {Processes} and
		{Some} {Applications}}. Cambridge University Press, 2016.
	\newblock arXiv:1501.04861.
	
	\bibitem{camia_conformal_2016}
	{\sc Camia, F., Gandolfi, A., and Kleban, M.}
	\newblock Conformal correlation functions in the {Brownian} loop soup.
	\newblock {\em Nuclear Physics B 902\/} (2016), 483--507.
	
	\bibitem{collet_simple_2014}
	{\sc Collet, G., and Fusy, {\'E}.}
	\newblock A {Simple} {Formula} for the {Series} of {Constellations} and
	{Quasi}-{Constellations} with {Boundaries}.
	\newblock {\em The Electronic Journal of Combinatorics 21}, 2 (2014), P2.9.
	
	\bibitem{drossel_winding_1996}
	{\sc Drossel, B., and Kardar, M.}
	\newblock Winding angle distributions for random walks and flux lines.
	\newblock {\em Phys. Rev. E 53}, 6 (June 1996), 5861--5871.
	
	\bibitem{flajolet_singularity_1990}
	{\sc Flajolet, P., and Odlyzko, A.}
	\newblock Singularity {Analysis} of {Generating} {Functions}.
	\newblock {\em SIAM J. Discrete Math. 3}, 2 (1990), 216--240.
	
	\bibitem{garban_expected_2006}
	{\sc Garban, C., and Trujillo-Ferreras, J.~A.}
	\newblock The {Expected} {Area} of the {Filled} {Planar} {Brownian} {Loop} is
	$\pi/5$.
	\newblock {\em Commun. Math. Phys. 264}, 3 (2006), 797--810.
	
	\bibitem{hartman_spectra_1954}
	{\sc Hartman, P., and Wintner, A.}
	\newblock The spectra of Toeplitz's matrices.
	\newblock {\em American Journal of Mathematics}, 76 (1954), 867--882.
	
	\bibitem{kauers_proof_2009}
	{\sc Kauers, M., Koutschan, C., and Zeilberger, D.}
	\newblock Proof of {Ira} {Gessel}'s lattice path conjecture.
	\newblock {\em PNAS 106}, 28 (2009), 11502--11505.
	
	\bibitem{kenyon_bipolar_2015}
	{\sc Kenyon, R., Miller, J., Sheffield, S., and Wilson, D.~B.}
	\newblock Bipolar orientations on planar maps and {SLE}$_{12}$.
	\newblock arXiv:1511.04068.
	
	\bibitem{kenyon_greens_2017}
	{\sc Kenyon, R.~W., and Wilson, D.~B.}
	\newblock The {Green}'s function on the double cover of the grid and
	application to the uniform spanning tree trunk.
	\newblock arXiv:1708.05381.
	
	\bibitem{lawler_random_2007}
	{\sc Lawler, G., and Trujillo~Ferreras, J.}
	\newblock Random walk loop soup.
	\newblock {\em Trans. Amer. Math. Soc. 359}, 2 (2007), 767--787.
	
	\bibitem{nesterenko_modular_1996}
	{\sc Nesterenko, Y. }
	\newblock Modular functions and transcendence.
	\newblock {\em Sbornik: Math. 187}, 9 (1996), 65-96. 
	
	\bibitem{reed1981functional}
	{\sc Reed, M., and Simon, B.}
	\newblock {\em Methods of Modern Mathematical Physics I: Functional Analysis.}
	\newblock Academic Press, San Diego, 1981.
	
	\bibitem{rudnick_winding_1987}
	{\sc Rudnick, J., and Hu, Y.}
	\newblock The winding angle distribution of an ordinary random walk.
	\newblock {\em J. Phys. A: Math. Gen. 20}, 13 (1987), 4421.
	
	\bibitem{schwartz_uber_1873}
	{\sc Schwartz, H. }
	\newblock \"Uber diejenigen F\"alle, in welchen die Gaussische hypergeometrische Reihe eine algebraische Function ihres vierten Elementes darstellt.
	\newblock {\em J. Reine Angew. Math. 75}, (1873), 292–335.
	
	\bibitem{shi_windings_1998}
	{\sc Shi, Z.}
	\newblock Windings of {Brownian} motion and random walks in the plane.
	\newblock {\em Ann. Probab. 26}, 1 (1998), 112--131.
	
	\bibitem{walker_analyticity_2003}
	{\sc Walker, P.}
	\newblock The Analyticity of Jacobian Functions with Respect to the Parameter k.
	\newblock Proc. R. Soc. Lond. A, 459 (2003), 2569-2574.
	
	\bibitem{yor_loi_1980}
	{\sc Yor, M.}
	\newblock Loi de l'indice du lacet {Brownien}, et distribution de
	{Hartman}-{Watson}.
	\newblock {\em Z. Wahrscheinlichkeitstheorie verw. Gebiete 53}, 1 (1980),
	71--95.
	
	\bibitem{zhu_operator_2007}
	{\sc Zhu, K.}
	\newblock {\em Operator theory in function spaces.}
	\newblock American Mathematical Soc., 2007.
	
	\bibitem{zimmer_essential_1990}
	{\sc Zimmer, J.}
	\newblock {\em Essential results of functional analysis.}
	\newblock Chicago Lectures in Mathematics, University of Chicago Press, 1990.
\end{thebibliography}
%

\end{document}